\documentclass[fleqn,12pt]{article}
\usepackage{amsmath,amssymb,amsthm,esint,dsfont,tikz,subcaption}
\usepackage[english]{babel}
\usepackage[margin=1in]{geometry}
\usepackage{csquotes}
\usepackage{hyperref}

\usepackage{color}

\newcommand{\Etx}{{\widetilde E}_\xi}

\newcommand{\Omt}{\widetilde\Omega}

\newcommand{\Qt}{\widetilde Q_\xi}

\newcommand{\Om}{\Omega}

\newcommand{\norm}[1]{{\Vert #1\Vert}}
\newcommand{\abs}[1]{{\vert #1\vert}}
\newcommand{\e}{\varepsilon}
\newcommand{\R}{{\mathbb R}}

\DeclareMathOperator{\tr}{tr}
\DeclareMathOperator{\dist}{dist}

\newtheorem{theorem}{Theorem}[section]
\newtheorem{proposition}[theorem]{Proposition}
\newtheorem{lemma}[theorem]{Lemma}
\newtheorem{corollary}[theorem]{Corollary}
\theoremstyle{definition}
\newtheorem{remark}[theorem]{Remark}

\title{Saturn ring defect around a spherical particle immersed in nematic liquid crystal}
\date{}
\author{Stan Alama
\footnote{Department of Mathematics \& Statistics, McMaster University, Hamilton, ON, L8S 4K1, Canada}
 \and Lia Bronsard
\footnotemark[\value{footnote}] 
\and Dmitry Golovaty  
\footnote{Department of Mathematics, The University of Akron, Akron, OH, 44325, USA}
  \and Xavier Lamy
\footnote{Institut de Math\'ematiques de Toulouse; UMR5219 - Universit\'e de Toulouse; CNRS - UPS IMT, F-31062 Toulouse Cedex 9, France} }

\begin{document}

\maketitle

\begin{abstract} 
We consider a nematic liquid crystal occupying the three-dimensional domain in the exterior of a spherical colloid particle. The nematic is subject to Dirichlet boundary conditions that enforce orthogonal attachment of nematic molecules to the surface of the particle. Our main interest is to understand the behavior of energy-critical configurations of the Landau-de Gennes $Q$-tensor model in the limit of vanishing correlation length.
We demonstrate existence of configurations with a single Saturn-ring defect approaching the equator of the particle  and no other line or point defects.
We show this by analyzing asymptotics of energy minimizers under two symmetry constraints: rotational equivariance around the vertical axis and reflection across the horizontal plane.
Energy blow-up at the ring defect is a significant obstacle to constructing well-behaved comparison maps needed to eliminate the possibility of point defects. The boundary estimates we develop to address this issue are new and
 should be applicable to a wider class of problems.
\end{abstract}

\section{Introduction}

The study of defects in liquid crystals is well-motivated from physical considerations, and is also closely connected to many fundamental questions in analysis and geometry. 
The intimate connection between nematic liquid crystals and $\mathbb{S}^2$-valued harmonic maps is well-established through director-based models such as Oseen-Frank \cite{HKL90}, and a comprehensive study of singularities in nematics will both exploit and expand the rich trove of analytical tools for studying geometrical variational problems.  To better describe nematics in settings involving non-orientability, biaxiality, and the presence of line defects, physicists and mathematicians have turned to the tensorial Landau-de Gennes model, which is in some sense a relaxation of the non-convex constraints of director models.  Indeed, much recent attention has concentrated on recovering the Oseen-Frank director and energy in the vanishing correlation length limit of Landau-de Gennes (see, e.g.,~\cite{baumanparkphillips12,canevari2d,singperturb,golovatymontero14,majumdarzarnescu10,nguyen_zarnescu_13}.)

In this paper we revisit an important model problem, that of a spherical colloid particle immersed in a nematic which fills the exterior domain, approaching a constant uniaxial state at infinity.  We work within the Landau-de Gennes framework, with homeotropic Dirichlet boundary conditions on the colloid surface.  Physicists have long expected that there are two competing candidates for minimizers in this geometry:  an orientable solution with dipolar symmetry, consisting of a single satellite point defect lying on the axis of symmetry, and a non-orientable director pattern having a circular ``Saturn ring'' singularity on the equatorial plane perpendicular to the symmetry axis.  The latter configuration exhibits quadrupolar symmetry.  In the first mathematical treatment of this problem \cite{colloid}, this expectation is confirmed in the case of very small colloids (for which the quadrupolar Saturn ring solution is minimizing,) or for very large colloids (in which, assuming axial symmetry, the dipolar satellite point defect prevails.)  
However, Saturn ring defects have been observed both experimentally and numerically in the physics literature (see, e.g.,~\cite{fukuda2004nematic,PhysRevE.57.610,muvsevivc2017liquid,stark01,terentjev95}) and appear to be energetically favorable in many settings, even for larger particles.
\subsection{Main results}
The goal of this paper is to produce solutions of the spherical colloid problem which exhibit Saturn-ring defects in the limit of small correlation length.  To do this, we minimize the Landau-de Gennes energy in a function space enforcing the expected (quadrupolar) symmetries of such a configuration.
The symmetry hypothesis will ensure the existence of at least one ring defect on the horizontal plane; a much more difficult issue is to eliminate the possibility of other defects (rings or point defects). 
Additional ring defects can be excluded by carefully adapting lower bound techniques developed for the Ginzburg-Landau problem \cite{jerrard99,sandier98,beaulieu_hadiji_98,andre_shafrir_98}. Ruling out point defects, however, presents a new and significant analytical challenge.
In general, determining the precise number of point defect in a three dimensional domain is a difficult task:
compared to line defects, point defects carry a negligible amount of energy and are thus harder to detect using energy estimates. Moreover, unlike line defects, the number of point defects can not be deduced from topological considerations as even topologically trivial boundary conditions may give rise to an arbitrary number of point defects \cite{hardtlin86}. Only very specific examples are known where the number of point defects can be determined (see e.g. \cite{HKL90,breziscoronlieb86,colloid}). In the present work this task is made even harder due to presence of a line defect approaching the boundary in that the boundary conditions are ``destroyed" by energy blow-up at the ring defect. This considerably complicates the  construction of well-behaved comparison maps.
 We overcome this  obstacle by proving a very  precise estimate in the blow-up region at the boundary. 
 This estimate appears to be new, even within the context of some well-studied variational problems (such as Ginzburg-Landau with a weight). 

\medskip

Let us now introduce the Landau-de Gennes functional, and the variational framework which we will use in our study. In nondimensional units the colloid particle is represented by the closed ball of radius one $B=\lbrace \abs{\cdot}\leq 1\rbrace\subset\R^3$, so that the liquid crystal is contained in the domain  $\Omega=\R^3\setminus B$. 
In these units the Landau-de Gennes energy depends on the nematic correlation length $\xi>0$, and is given by
\begin{align}
\label{eq:energy}
E_\xi(Q)&=\int_{\Omega} \left( \abs{\nabla Q}^2 +\frac{1}{\xi^2}f(Q)\right) dx.
\end{align}
The map $Q$ takes values into the space $\mathcal S_0$ of $3\times 3$ symmetric matrices with zero trace and describes nematic alignment.
The nematic potential is given by
\begin{equation*}
f(Q)=-\frac 12\abs{Q}^2-\tr(Q^3)+\frac 34 \abs{Q}^4+C,
\end{equation*}
where the constant $C$ is such that $f$ satisfies
\begin{equation}\label{eq:Ustar}
f(Q)\geq 0\text{ with equality iff }Q\in \mathcal U_\star:=\left\lbrace n\otimes n -\frac 13 I\colon n\in\mathbb S^2\right\rbrace.
\end{equation}
The correlation length $\xi$ is typically small and therefore we are going to be  interested in the limit $\xi\to 0$.

Anchoring at the particle surface is assumed to be radial:
\begin{equation}\label{eq:Qb}
Q=Q_b:=e_r\otimes e_r -\frac 13 I\qquad\text{on }\partial \Omega,\qquad e_r=\frac{x}{\abs{x}}.
\end{equation}
At infinity, the effect of the particle is not felt and the alignment is uniform, given by
\begin{equation}
\label{eq:Qinf}
Q_\infty:=e_3 \otimes e_3 -\frac 13 I.
\end{equation}
More precisely, this far field condition is enforced by considering configurations in the space $\mathcal H$ given by
\begin{equation}\label{eq:H}
\mathcal H :=Q_\infty +\dot {\mathcal H},\qquad
\dot {\mathcal H} :=\left\lbrace Q\in H^1_{loc}(\Omega;\mathcal S_0)\colon \int_\Omega \abs{\nabla Q}^2 + \int_\Omega \frac{\abs{Q}^2}{\abs{x}^2} <\infty \right\rbrace.
\end{equation}
We denote by $\mathcal H_b$ the space of such configurations that satisfy in addition the radial anchoring condition at the particle surface:
\begin{equation}\label{eq:Hb}
\mathcal H_b = 
\left\lbrace Q\in \mathcal H\colon Q\text{ satisfies \eqref{eq:Qb}}
\right\rbrace.
\end{equation}
We seek to construct critical points of $E_\xi$ with the quadrupolar symmetry of the Saturn ring configuration.  This entails two symmetry constraints on the admissible $Q\in \mathcal{H}_b$:
\begin{itemize}
\item rotation symmetry around the vertical axis, 
\item and reflection symmetry across the equatorial plane.
\end{itemize} 
See \eqref{eq:groups} for a precise definition of each, in terms of group actions.
We denote the space of maps in $\mathcal H_b$ satisfying these two symmetry constraints by $\mathcal H_{sym}$, i.e.
\begin{align*}
\mathcal H_{sym} =\lbrace Q\in\mathcal H_b\text{ with quadrupolar symmetry}\rbrace.
\end{align*}
A more complete discussion of the space $\mathcal H_{sym}$ will be given in Section~\ref{s:upper}.
Minimizers of $E_\xi$ in $\mathcal H_{sym}$ do exist, because Hardy's inequality ensures  the coercivity of the energy. Moreover, they are critical points of $E_\xi$ in the full space $\mathcal H$: this is a  consequence of the principle of symmetric criticality \cite{palais79} (see also \cite[Appendix~1]{bifur}). Our main result shows that the energy of a symmetric minimizer concentrates, as $\xi\to 0$, inside a Saturn ring shaped region around the particle. In the limit this region coincides with the equatorial circle
\begin{equation*}
\mathcal C:=\left\lbrace (\cos\varphi,\sin\varphi,0)\colon \varphi\in\R\right\rbrace =\partial B \cap\lbrace x_3=0\rbrace.
\end{equation*}
Since the energy will blow up around $\mathcal C$, the resulting limit configuration will not belong to $\mathcal H_{sym}$. To define the limit space we cut out a small neighborhood of $\mathcal C$, consider the exterior domain
\begin{align*}
\Omega_\delta^{ext} & = \left\lbrace x\in \Omega \colon  \dist(x,\mathcal C)>\delta\right\rbrace,
\end{align*}
and define the limit space $\mathcal H^\star_{sym}=\cap_{\delta >0}\mathcal H^\star_{sym}(\Omega^{ext}_\delta)$, where
\begin{align*}
\mathcal H^\star_{sym}(\Omega^{ext}_\delta)=
\Big\lbrace Q & \in H^1_{loc}(\Omega^{ext}_\delta;\mathcal U_\star) \ \text{with quadrupolar symmetry, s.t.} \\
&\int_{\Omega^{ext}_\delta}\abs{\nabla Q}^2 + \int_{\Omega^{ext}_\delta}\frac{\abs{Q-Q_\infty}^2}{\abs{x}^2}<\infty,\quad\text{and}\quad
 Q=Q_b\text{ for }\abs{x}=1\Big\rbrace.
\end{align*}

We may now state our result asserting the existence and asymptotic behavior of solutions with quadrupolar symmetry:
\begin{theorem}\label{thm:main}
For any $\xi\in (0,1]$ let $Q_\xi$ minimize $E_\xi$ in $\mathcal H_{sym}$. Then we have:
\begin{itemize}
\item[(i)] \emph{upper bound:} there exists a universal constant $C>0$ such that
\begin{equation*}
\frac{1}{2\pi}E_\xi(Q_\xi)\leq \pi\ln\frac 1\xi +  \pi\ln\ln\frac 1\xi + C.
\end{equation*}
\item[(ii)] \emph{lower bound:} there exists a universal constant $C>0$ such that for any $\delta\in (0,1)$, 
\begin{equation*}
\frac 1{2\pi}E_\xi(Q_\xi;\Omega_\delta^{int})\geq \pi\ln\frac 1\xi +   \pi\ln\ln\frac 1\xi - 2\pi \ln\frac 1\delta - C\qquad\text{as }\xi\to 0,
\end{equation*}
where $\Omega_\delta^{int}=\Omega\setminus\Omega_\delta^{ext}=\Omega\cap\lbrace\dist(\cdot,\mathcal C)\leq\delta\rbrace$ is the $\delta$-neighborhood of the ring defect $\mathcal C$.
\item[(iii)] \emph{limit configuration:} there is a subsequence $\xi\to 0$ such that $Q_\xi$ converges in $C^{1,\alpha}_{loc}(\overline\Omega\setminus\mathcal C)$ to a map $Q_\star\in \mathcal H^\star_{sym}$ which is smooth in $\overline\Omega\setminus\mathcal C$, and uniaxial,
\begin{equation*}
Q_\star(x)=n_\star(x)\otimes n_\star(x)-\frac 13 I,\quad n(x)\in\mathbb S^2,
\end{equation*}
where $n_\star$ is a smooth $\mathbb{S}^2$-valued locally minimizing harmonic map in $\Om$.  Furthermore, $n_\star$
satisfies the additional symmetry property that $n_\star(x_1,0,x_3)\perp e_2$.
\end{itemize}
\end{theorem}

\begin{remark}\label{r:limit_symmetry}
Using cylindrical coordinates $(\rho,\varphi,z)$ and corresponding orthonormal frame $(e_\rho,e_\varphi,e_z)$,  the additional symmetry statement in $(iii)$ amounts to  $n_\star(\rho,\varphi,z)\perp e_\varphi$, i.e. $n_\star(\rho,\varphi,z)$ lies in the azimuthal plane generated by $e_\rho$ and $e_z$.
\end{remark}

In proving Theorem~\ref{thm:main} we rely heavily on the symmetry constraint, which reduces the problem to two dimensions.  Indeed, the relevant analogy is to a two-dimensional Ginzburg-Landau energy with a weight $w=\rho$ arising from cylindrical symmetry. The asymptotic behavior of Ginzburg-Landau energies with weights have been studied in \cite{andre_shafrir_98,beaulieu_hadiji_98}.  The principal novelty of this work is that we must deal with $Q$-tensor-valued maps in an unbounded domain rather than complex-valued maps in a bounded domain. Points $(i)$ and $(ii)$ of Theorem~\ref{thm:main} follow from careful adaptation of the techniques in \cite{andre_shafrir_98,beaulieu_hadiji_98}, along with classical Ginzburg-Landau methods in \cite{struwe_94,jerrard99,sandier98}, and more recent arguments for $Q$-tensor-valued maps in \cite{golovatymontero14,canevari2d}.

The most delicate part of Theorem~\ref{thm:main} is the statement $(iii)$ asserting that the limit is smooth everywhere away from the equatorial ring defect $\mathcal C$. Proving this statement amounts to eliminating the possibility of point singularities appearing on the $z$-axis. As it was already pointed out, in general, energy minimization subject to topological constraints does not restrict the number of point defects \cite{hardtlin86}. In the context of our problem, one negatively charged Saturn ring defect is sufficient to compensate the topological constraint imposed by the boundary conditions but so is, for example, a positively charged Saturn ring together with a pair of negatively charged point defects on the $z$-axis. In the process of ruling out existence of such pair of point defects, a crucial ingredient is to show that energy of the core of a ring defect would be the same up to terms of order $o(1)$ in $\xi$, regardless of whether the ring is negatively or positively charged. In other words, the $O(1)$ remainder in the energy asymptotics in Theorem~\ref{thm:main} does not depend on the sign of the ring defect's charge.

Here we encounter an additional difficulty not present in determining the $O(1)$ core energy term for classical Ginzburg-Landau vortices.  Indeed, in \cite{bbh} the authors crucially use the fact that the energy of  Ginzburg-Landau vortices scales radially: at scale $r\gg\xi$, the energy of  a vortex goes as $\ln(r/\xi)=\ln (1/\xi)+\ln r$, and the effect of phase winding is separated from the cost of core formation.  Here, on the other hand, proximity to the boundary breaks this scaling invariance and influences the core shape, as seen by the presence of the $\ln\ln\xi$ term, and makes it much less clear that radial rescaling should reveal a universal $O(1)$ core energy term.

To obtain the core energy estimate, we deform the minimizers in a very narrow wedge domain emanating tangentially from the limiting equatorial defect  (see Lemma~\ref{l:Pxi}). This requires a sharp lower bound on the energy in a small disc tangential to the particle surface at its equator (see Lemma~\ref{l:finelowerboundhalfannulus}). The corresponding core energy estimate is new---to the best of our knowledge. Once the core energy is determined, the added energy cost of an anti-hedgehog pair may be computed thanks to ideas in \cite{colloid} (see Lemma~\ref{l:compareEtau}).

\subsection{Background and relevant numerical results}

The mathematical study of line defects in nematics was initiated in \cite{canevari3d}, in the singular limit as the correlation length $\xi\to 0$, for domains and boundary values which induce defects along line segments.  As mentioned earlier, global minimizers of the spherical colloid problem were first addressed mathematically in \cite{colloid}, in which  the size of the colloid plays a determining role.  
As has long been known by physicists, equatorial ring defects can be observed even around large colloid particles, for example in the presence of external electric or magnetic fields \cite{fukudaetal04, fukudayokoyama06,guabbott00,loudetpoulin01,stark02} or in confinement \cite{lavrentovich2014transport,stark02}.  The situation with a magnetic field was studied mathematically in \cite{colloid2}, via a Landau-de Gennes energy modified to model interaction with a constant  field.  The main result of \cite{colloid2} identifies the leading order term in an expansion of the energy, indicating the presence of an equatorial ring defect rather than a satellite point defect,
provided the magnetic field is high enough $h\gg \xi\:\vert \ln\xi\vert$. In the complementary low magnetic field regime $h\ll\xi\:\vert\ln\xi\vert$, 
the lower bound established here in Theorem~\ref{thm:main} directly implies, in view of upper bounds established in \cite{colloid2}, that minimizers cannot have quadrupolar symmetry, thus hinting at  the presence of a satellite point defect.
The asymptotics of that model are further and more precisely explored in \cite{alouges_chambolle_stantejsky}.

Even in the absence of external factors which appear to favor rings over satellite point defects, much physical evidence, both numerical  \cite{fukuda2004nematic} and formal  \cite{PhysRevE.57.610}, suggests that there is a range of intermediate particle sizes for which configurations with Saturn ring defects may be stable and coexist with  point defect configurations having lower energy.  

We do not consider the important question of stability in this paper, but numerical simulations suggest that the solutions found here may be locally stable. To illustrate these observations, in Figs.~\ref{fig:configurationsa}-\ref{fig:energy} we present the summary of simulations that reproduce and extend the results of \cite{fukuda2004nematic}. We numerically solved in COMSOL \cite{COMSOL} the equations for the gradient flow 
\[\frac{\partial Q}{\partial t}=-\frac{\delta E_\xi}{\delta Q}\]
for the energy $E_\xi$ defined in \eqref{eq:energy} in the domain in the form of a large cylinder with a spherical void of radius $1$ with the same center as that of the cylinder. The admissible $Q$-tensor fields have values in the set of symmetric traceless matrices and are rotationally equivariant with respect to $z$-axis that is also the axis of the cylinder. The $Q$-tensors are subject to the initial condition $Q(\cdot,0)=Q_{init}$ and the boundary conditions \eqref{eq:Qinf} and \eqref{eq:Qb} on the surfaces of the cylinder and the sphere, respectively. 

Following \cite{fukuda2004nematic} and assuming that $(\rho,\varphi)$ are polar coordinates in a plane perpendicular to the axis of the cylinder, the simulations were run starting from two initial conditions 
\begin{align*}
Q_{init}={e}_3\otimes{e}_3-\frac{1}{3}I\qquad\text{and}\qquad
Q_{init}={n}(\psi)\otimes{n}(\psi)-\frac{1}{3}I,
\end{align*}
where ${\bf n}(\psi)=(\cos{\psi}\cos{\varphi},\cos{\psi}\sin{\varphi},\sin{\psi})$ and
\[\psi=2\tan^{-1}\left({\frac{\rho}{z}}\right)-\tan^{-1}\left({\frac{\rho}{z+z_0}}\right)-\tan^{-1}\left({\frac{\rho}{z+z_0^{-1}}}\right).\]
Here the second choice of the initial condition represents an approximation of a nematic configuration with a hyperbolic point defect at distance $z_0$ below the sphere's center \cite{fukuda2004nematic,PhysRevE.57.610}. Note that, for $\xi<0.005$, the simulations leading to an equatorial Saturn ring were started from the critical point obtained for $\xi=0.005$.

Fig.~\ref{fig:configurationsa}-\ref{fig:configurationsb} shows the line fields of the nematic in $(r,z)$-coordinates when $\xi=1/70$. For this choice of the correlation length, the critical point approached by the gradient flow simulation depends on the initial condition; the critical point in Fig.~\ref{fig:configurationsa} has dipolar symmetry, while the critical point Fig.~\ref{fig:configurationsb} is quadrupolar. 
\begin{figure}
\centering
\begin{subfigure}{.25\textwidth}
\centering
  \includegraphics[width=\textwidth]{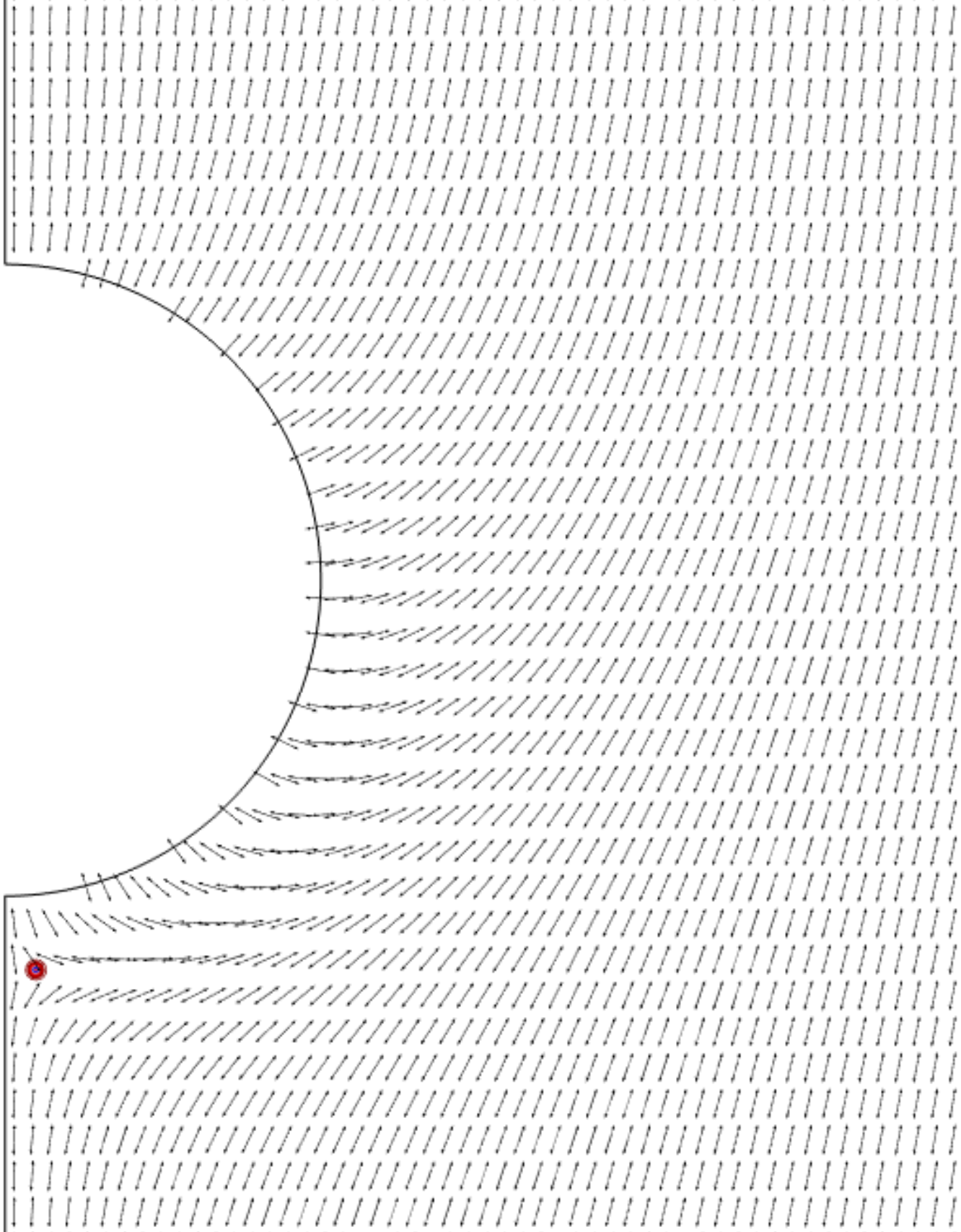}
 \caption{}
\label{fig:configurationsa}
\end{subfigure}
\hspace{1em}
\begin{subfigure}{.26\textwidth}
\centering
  \includegraphics[width=\textwidth]{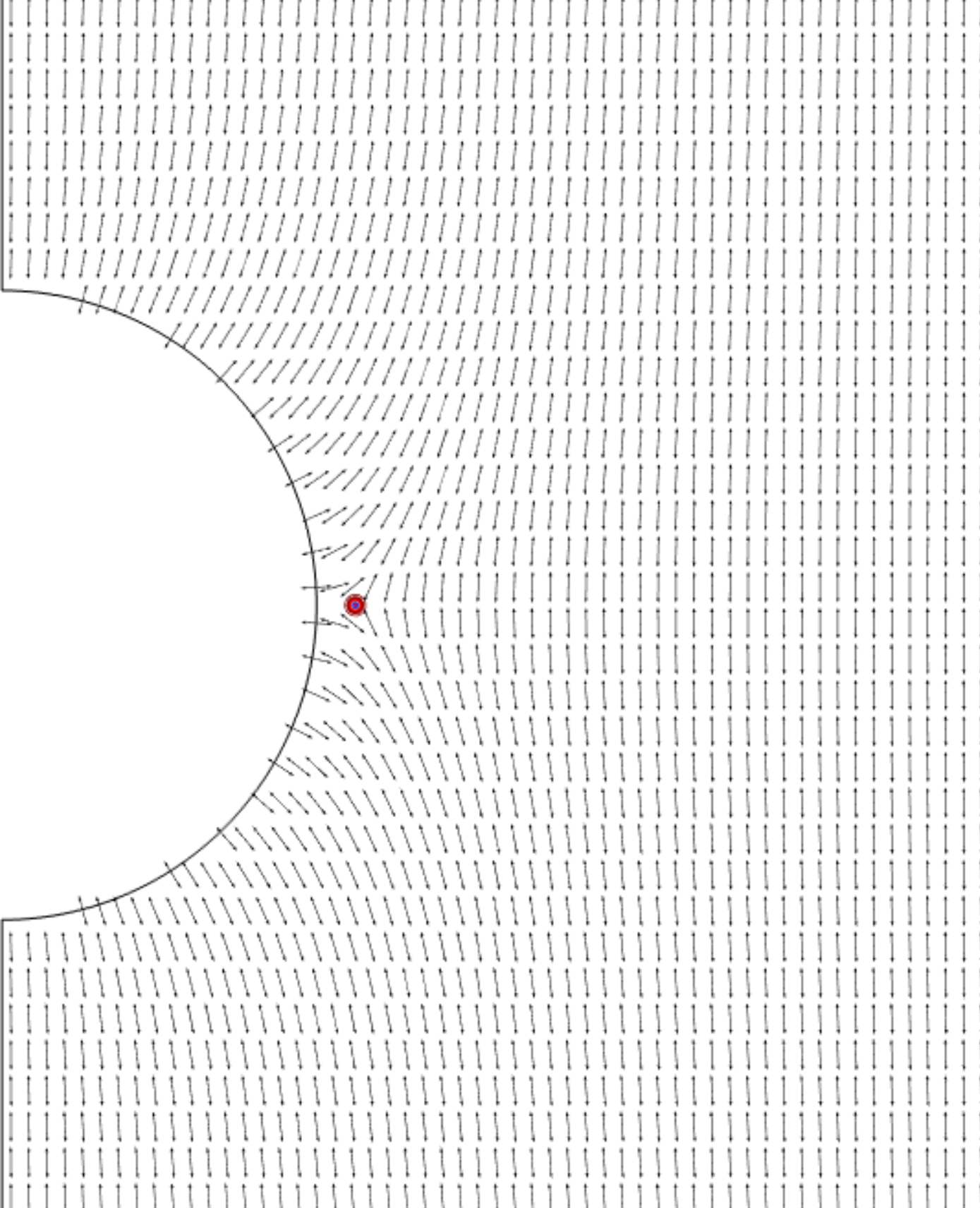}
\caption{}
\label{fig:configurationsb}
\end{subfigure}
\hspace{1em}
\begin{subfigure}{.41\textwidth}
\centering
  \includegraphics[width=\textwidth]{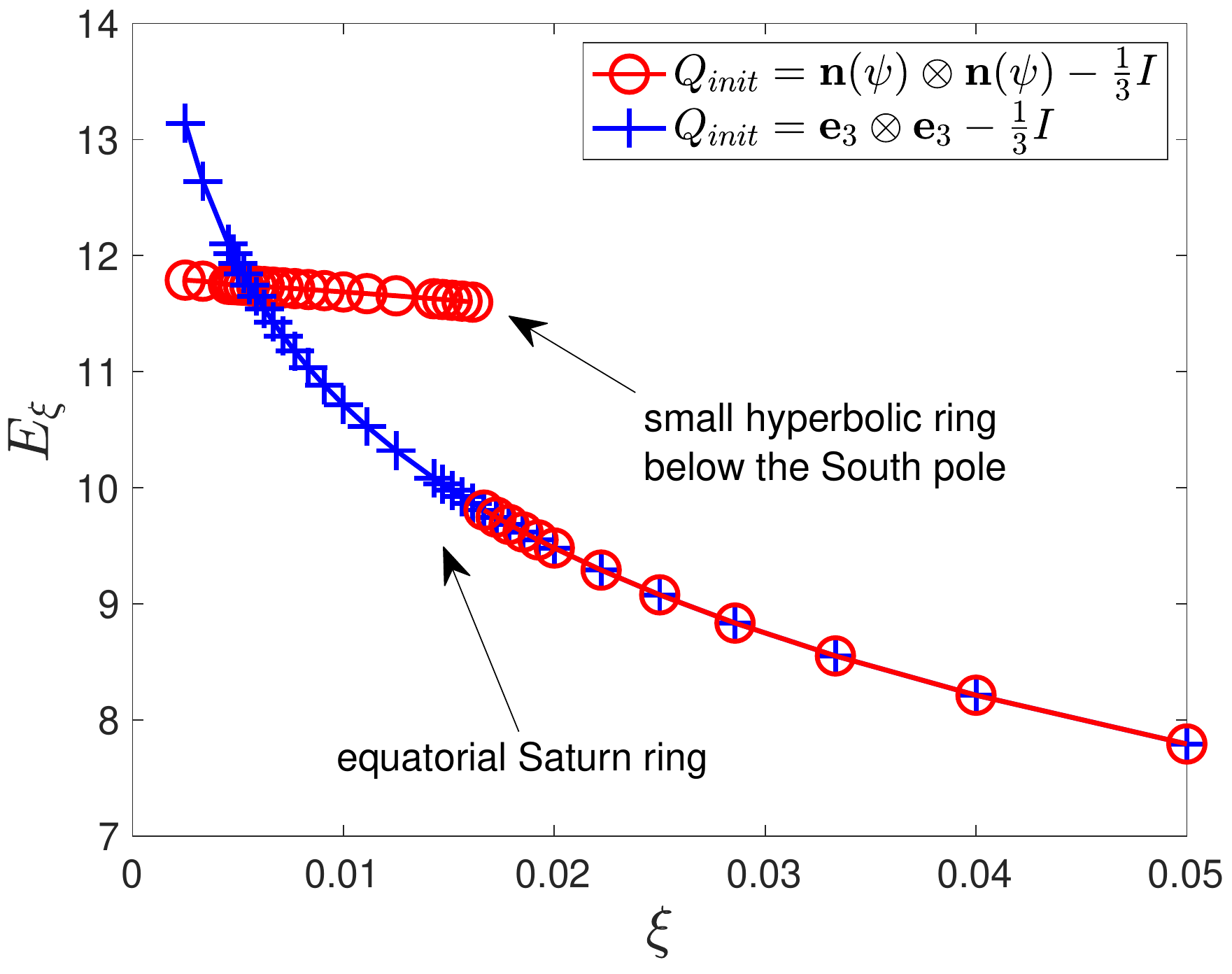}
\caption{}
\label{fig:energy}
\end{subfigure}
\caption{(a-b) Nematic configurations corresponding to critical points for $E_\xi$ when $\xi=1/70$. The red dot marks the location of a Saturn ring. (a) $Q_{init}={n}(\psi)\otimes{n}(\psi)-\frac{1}{3}I$ and the critical point is a small hyperbolic ring below the south pole of the particle. The ring shrinks to a hyperbolic point defect on $z$-axis when $\xi\to0$. (b) $Q_{init}={e}_3\otimes{e}_3-\frac{1}{3}I$ and the critical point is the equatorial Saturn ring. The Saturn ring approaches the surface of the colloid when $\xi\to0$. (c) Energy $E_\xi$ vs nematic correlation length $\xi$. The critical point reached from the constant initial condition (blue) is always the equatorial Saturn ring. The critical points with an equatorial Saturn ring and with a small hyperbolic ring coexist and appear to be stable for all values of $\xi<\xi_c$ for which the simulations were conducted.}
\end{figure}

For larger values of $\xi$, once it exceeds a critical value $\xi_c\approx0.017$, the simulations converge to the equatorial Saturn ring configuration, regardless of the initial conditions. In fact, for the initial condition with a hyperbolic point defect, this defect expands first into a small ring below the south pole of the colloid. This ring then expands and travels up the surface of the colloid, eventually stopping at its equator.

For $\xi<\xi_c$, the dichotomy between the initial conditions persists for all values of $\xi$ for which the simulations were run: the hyperbolic point defect expands into a small ring below the colloid and the constant initial condition leads to the equatorial Saturn ring. Even though the energy of the equatorial ring exceeds the energy of the small ring once $\xi$ becomes sufficiently small, the equatorial ring remains stable. These observations are summarized in Fig.~\ref{fig:energy}.

Although we do not address the question of their minimality, in the present work we establish that Saturn ring-like critical points of the Landau-de Gennes energy indeed do exist for large particles. Here, rather than varying the radius of the particle, we use an equivalent description in which the size of the particle is fixed and the nematic correlation length is assumed to converge to zero.

This paper is organized as follows. 
In Sections~\ref{s:upper} and \ref{s:lower}, we establish the upper and lower bounds given by the statements $(i)$ and $(ii)$ of  Theorem \ref{thm:main}. In Section~\ref{s:lim} we prove the statement $(iii)$ of Theorem~\ref{thm:main}; namely, a sequence of minimizers converges to a limiting map which is smooth away from the ring defect.

\subsubsection*{Acknowledgments} S.A. and L.B. are supported by NSERC (Canada) Discovery Grants, D.G. was supported in part by NSF grant DMS-1729538, and X.L. by ANR project ANR-18-CE40-0023.
We also thank the Erwin Schr\"odinger International Institute for Mathematics and Physics (Vienna) for their kind hospitality 
in December 2019, when this work was being completed.

\section{Upper bound}\label{s:upper}

To obtain the upper bound, i.e. part $(i)$ of Theorem~\ref{thm:main}, we construct an admissible map $Q\in\mathcal H_{sym}$ whose energy has the expected behavior.

Let us first describe more explicitly what quadrupolar symmetry means.
Symmetries are formalized in terms of the equivariant action of the orthogonal group $O(3)$ on maps $Q\colon \R^3\to\mathcal S_0$, given by
\begin{equation*}
(R\cdot Q)(x)=R Q(R^t x)R^t,\qquad R\in O(3).
\end{equation*}
The energy $E_\xi$ is invariant under this action, consistent with the physical requirement of frame invariance. Here we consider the subgroup
\begin{equation*}
G_{sym}=\left\lbrace R\in O(3)\colon R{e}_3 =\pm e_3\right\rbrace.
\end{equation*}
 This is the largest subgroup of $O(3)$ that maps the space $\mathcal H$ into itself.
Explicitly, $G_{sym}$ is generated by all rotations around the vertical axis $e_3$ and by the reflection with respect to the horizontal plane $\langle e_1,e_2\rangle$, that is,
\begin{align}\label{eq:groups}
G_{sym}=\left\langle \lbrace R_\varphi\rbrace_{\varphi\in\R},S\right\rangle,
\quad
R_\varphi = \left(\begin{array}{ccc}
\cos\varphi & -\sin\varphi & 0 \\
\sin\varphi & \cos\varphi & 0 \\
0&0&1
\end{array}\right),\quad
S  = \left(\begin{array}{ccc}
1 & 0 & 0 \\
0 & 1 & 0 \\
0&0& -1
\end{array}\right).
\end{align}
The critical points we study in this work are minimizers of the energy among symmetric configurations belonging to the space
\begin{equation}\label{eq:Hsym}
\mathcal H_{sym}=\left\lbrace Q\in\mathcal H_b\colon R\cdot Q=Q,\;\forall R\in G_{sym}\right\rbrace,
\end{equation}
 It is natural to use cylindrical coordinates to describe maps in the space $\mathcal H_{sym}$, i.e. coordinates $(\rho,\varphi,z)$ defined by
\begin{equation*}
x=R_\varphi \left(\begin{array}{c} \rho \\ 0 \\ z\end{array}\right) =  \rho R_\varphi e_1 + z e_3.
\end{equation*}
In these coordinates the domain $\Omega$ corresponds to $\lbrace \rho^2+z^2 >1\rbrace$, and maps $Q\in\mathcal H_{sym}$ can be written in the form
\begin{equation}\label{eq:axialsym}
Q(\rho,\varphi,z)=R_\varphi \widetilde Q(\rho,z) R_{\varphi}^t,
\end{equation}
where in addition  $\widetilde Q(\rho,z)=Q(\rho,0,z)$ satisfies the mirror symmetry constraint
\begin{equation}\label{eq:mirrorsym}
\widetilde Q(\rho,-z)=S\widetilde Q(\rho,z)S^t.
\end{equation}
Written in cylindrical coordinates, the energy of a map $Q\in \mathcal H_{sym}$ takes the form
\begin{align*}
\frac 1{2\pi}E_\xi(Q) & =\widetilde E_\xi(\widetilde Q) := \iint_{\widetilde\Omega}\left(\abs{\nabla\widetilde Q}^2 + \frac{1}{\rho^2}\Xi[\widetilde Q] + \frac{1}{\xi^2}f(\widetilde Q) \right)\rho\, d\rho dz,\\
\text{where }\Xi[\widetilde Q]&=\vert \partial_\varphi[R_\varphi\widetilde Q R_\varphi^t ]\vert^2 =8 \widetilde Q_{12}^2 + 2 \widetilde Q_{13}^2+2\widetilde Q_{23}^2 +2(\widetilde Q_{11}-\widetilde Q_{22})^2.
\end{align*}
Note that $\Xi[\widetilde Q]\leq 4\, \mathrm{dist}^2(\widetilde Q,\R Q_\infty),$ where the distance is induced by the Frobenius norm. We will construct $\widetilde Q$ in the region 
\begin{equation*}
D:=\lbrace \rho^2+z^2>1\colon\rho,z>0\rbrace,
\end{equation*}
with boundary constraints
\[{\left\{\begin{aligned}
&\widetilde Q(\rho,z)=e_r\otimes e_r -\frac 13 I, & \text{ for }&\rho^2+z^2=1,\\
&\widetilde Q(\rho,0)=S \widetilde Q(\rho,0) S^t, & \text{ for }&\rho>1,
\end{aligned}\right.}\]
 and then use the mirror symmetry \eqref{eq:mirrorsym} to extend $\tilde Q$ to the entire domain $\Omega$. Note that $\widetilde Q = S\widetilde Q S^t$ implies that $e_3$ is an eigenvector of $\widetilde Q$ on the horizontal axis $\lbrace z=0\rbrace$ in order for the mirror symmetry not to create a jump of $\tilde Q$. 
 
We begin by observing that in all estimates in this and subsequent sections the letter $C$ denotes a generic universal constant. To construct $\widetilde Q$ we first divide $D$ into 3 subdomains, $D=\overline D_1 \cup \overline D_2 \cup \overline D_3,$ where
\[
D_1= D\cap\lbrace \rho^2+z^2>2\rbrace,\quad
D_2= (D \setminus \overline D_1) \cap \lbrace (\rho-1)^2 + z^2 > 1/4\rbrace,\quad
D_3= D\setminus (\overline D_1 \cup \overline D_2).
\]
\begin{center}
\begin{tikzpicture}[scale=.7]

\draw[thick] (2,0) arc (0:90:2);
\draw[thick] (4,0) arc (0:90:4);
\draw[thick] (2,0) -- (6,0);
\draw[thick] (0,2) -- (0,6);
\draw[thick] (3,0) arc (0:105:1);

\draw (3.5,3.5) node {$D_1$};
\draw (1.5,2.5) node {$D_2$};
\draw (2.4,.4) node {$D_3$};

\draw[<->] (2,-.2) -- (3,-.2);
\draw (2.5,-.2) node [below] {$1/2$};
\draw[<->] (-.2,2) -- (-.2,4);
\draw (-.2,3) node [left] {$1$};

\end{tikzpicture}
\end{center}
Let
\begin{equation*}
\widetilde Q \equiv Q_\infty = e_3\otimes e_3-\frac 13 I\qquad\text{ in }D_1,
\end{equation*}
so that
\begin{equation}\label{eq:E_D1}
\widetilde E_\xi(\widetilde Q;D_1)=0.
\end{equation}
Then we define a Lipschitz map $n\colon \partial D_2\to\mathbb S^2$ as follows. We first set
\begin{align*}
n&\equiv e_3 \quad\text{ on }\partial D_2 \setminus \left(\partial D_3 \cup \lbrace \rho^2+z^2=1\rbrace\right),\\
n& = (\rho,0,z) \quad\text{ on }\partial D_2 \cap \lbrace \rho^2+z^2=1\rbrace. 
\end{align*}
To define $n$ on $\partial D_2\cap\partial D_3$ we use polar coordinates $(r,\theta)$ centered in $(1,0)$, that is, given by the relation $\rho+iz=1+re^{i\theta}$, so that
\begin{equation*}
\partial D_2\cap\partial D_3=\lbrace r=1/2,0<\theta<\theta_0:=\pi/2 + \arcsin(1/4)\rbrace.
\end{equation*}
We set
\begin{equation*}
n(1/2,\theta)=(\sin\theta,0,\cos\theta)\quad\text{ for }0<\theta<\pi/2,
\end{equation*}
and in the remaining part of $\partial D_2\cap\partial D_3$ we interpolate linearly. More precisely, at the point $(1/2,\theta_0)$, continuity of $n$ imposes $n=(\cos\theta_1,0,\sin\theta_1)$ where $\theta_1=2\theta_0-\pi = 2\arcsin(1/4)$, so we set
\begin{equation*}
n(1/2,\theta)=(\cos(2\theta-\pi),0,\sin(2\theta-\pi))\quad\text{ for }\pi/2 <\theta <\theta_0.
\end{equation*}

\begin{center}
\begin{tikzpicture}[scale=1]

\draw[thick] (29:2) arc (29:90:2);
\draw[gray,thick] (2,0) arc (0:29:2);
\draw[thick] (4,0) arc (0:90:4);
\draw[thick] (3,0) -- (4,0);
\draw[thick,gray] (2,0) -- (3,0);
\draw[thick] (0,2) -- (0,4);
\draw[thick] (3,0) arc (0:105:1);

\draw[very thick,->] (0,2) node {$\bullet$} -- (0,2.5);
\draw[very thick,->] (0,3) node {$\bullet$} -- (0,3.5);
\draw[thick,->] (0,4) node {$\bullet$} -- (0,4.5);
\draw[thick,->] (4,0) node {$\bullet$} --  ++(0,.5);
\draw[thick,->] (15:4) node {$\bullet$} --  ++(0,.5);
\draw[thick,->] (30:4) node {$\bullet$} --  ++(0,.5);
\draw[thick,->] (45:4) node {$\bullet$} --  ++(0,.5);
\draw[thick,->] (60:4) node {$\bullet$} --  ++(0,.5);
\draw[thick,->] (75:4) node {$\bullet$} --  ++(0,.5);
\draw[thick,->] (3,0) node {$\bullet$} --  ++(0,.5);
\draw[thick,->] (3.5,0) node {$\bullet$} --  ++(0,.5);
\draw[thick,->] (2,0)++(45:1) node {$\bullet$} --  ++(45:.5);
\draw[thick,->] (2,1) node {$\bullet$} --  ++(.5,0);
\draw[thick,->] (29:2) node {$\bullet$} --  ++(29:.5);
\draw[thick,->] (50:2) node {$\bullet$} --  ++(50:.5);
\draw[thick,->] (70:2) node {$\bullet$} --  ++(70:.5);

\draw (2.1,2.4) node {$n\lfloor\partial D_2$};

\end{tikzpicture}
\end{center}


The map $n\colon \partial D_2\to\mathbb S^2$ thus defined can be written in the form
\begin{equation*}
n=(\cos\varphi,0,\sin\varphi),\qquad \varphi\in\mathrm{Lip}(\partial D_2,\R).
\end{equation*}
Thus, considering an $H^1$ extension of $\varphi$ to $D_2$ we obtain $n\colon D_2\to\mathbb S^2$ with $\int_{D_2}\abs{\nabla n}^2\leq C$ and moreover, thanks to Hardy's inequality and since $n=e_3$ on $\partial D_2\cap\lbrace \rho=0\rbrace$, we have also $\int_{D_2}\rho^{-2}\abs{n- e_3}^2\leq C$. Then we set
\begin{equation*}
\widetilde Q=Q_n=n\otimes n-\frac 13 I\qquad\text{ in }D_2,
\end{equation*}
and, since $f(Q_n)=0$ and $\Xi[Q_n]\leq C \abs{n-e_3}^2$, we have 
\begin{equation}\label{eq:E_D2}
\widetilde E_\xi(\widetilde Q;D_2)\leq C.
\end{equation}

Next we need to define $\widetilde Q$ in $D_3$. We introduce a parameter $\sigma>0$ with $\xi<\sigma<1/8$, and further divide $D_3$ into 3 subdomains:
\begin{align*}
D_4&= D_3\cap \lbrace (\rho - 1)^2+z^2 >(4\sigma)^2\rbrace, \\
D_5&=(D_3\setminus \overline D_4)\cap \lbrace (\rho - 1-2\sigma)^2+z^2 >\sigma^2\rbrace,\\
 D_6 &=D_3\setminus (\overline D_4\cup\overline D_5).
\end{align*}

\begin{center}
\begin{tikzpicture}[scale=.75]

\draw[thick] (0,0) arc (0:35:8);
\draw[thick] (4,0) arc (0:104.5:4);
\draw[thick] (0,0) -- (4.4, 0);

\draw[thick] (2,0) arc (0:98:2);

\draw[thick] (1.5,0) arc (0:180:.5);

\draw[->] (0,0) -- ++ (106:4);
\draw (-1,2) node {{\scriptsize $1/2$}};

\draw[->] (0,0) -- ++(70:2);
\draw (.3,1.5) node { {\scriptsize $4\sigma$} };

\draw[->] (1,0) -- ++ (70:.5);
\draw (1.25,.2) node {{\scriptsize $\sigma$}};

\draw (1,0) node [below] {$D_6$};

\draw (2,2) node {$D_4$};
\draw (1,1) node {$D_5$};

\end{tikzpicture}
\end{center}


In the polar coordinates $(r,\theta)$ centered at $(1,0)$, the domain $D_4$ is given by
\begin{equation*}
D_4=\lbrace 1+re^{i\theta}\colon 4\sigma < r <1/2, 0<\theta<\theta_0(r):=\pi/2 + \arcsin(r/2)\rbrace.
\end{equation*}
In $D_4$ we define $\widetilde Q$ as 
\begin{equation*}
\widetilde Q=Q_n = n\otimes n -\frac 13 I,
\end{equation*}
for some map $n\colon D_4\to\mathbb S^2$. The boundary conditions on $\lbrace \rho^2+z^2=1\rbrace$ impose
\begin{equation*}
n(r,\theta_0(r))=(\cos \theta_1(r),0,\sin\theta_1(r)),\quad\theta_1(r):=2\theta_0(r)-\pi,
\end{equation*}
and we define
\begin{equation*}
n(r,\theta)=\begin{cases}
(\sin\theta,0,\cos\theta) &\text{ for }0<\theta<\pi/2,\\
(\cos(2\theta-\pi),0,\sin(2\theta-\pi))&\text{ for }\pi/2<\theta<\theta_0(r).
\end{cases}
\end{equation*}
That way we have
\begin{align*}
\int_{D_4}\abs{\nabla \widetilde Q}^2 \, \rho \, d\rho\, dz &
=2 \int_{D_4}\abs{\nabla n}^2 \,\rho \, d\rho\, dz \\
& = 2 \int_{4\sigma}^{1/2}\int_0^{\theta_0(r)}\frac{1}{r^2}\abs{\partial_\theta n}^2 (1+r\cos\theta)\, r  \, d\theta\, dr \\
&= 2 \int_{4\sigma}^{1/2}\int_0^{\pi/2}\frac{1+r\cos\theta}{r} \, d\theta\, dr\\
&\quad
+ 8 \int_{4\sigma}^{1/2}\int_{\pi/2}^{\theta_0(r)}\frac{1+r\cos\theta}{r} \, d\theta\, dr\\
& \leq \pi\ln\frac 1\sigma + C.
\end{align*}
Moreover, in $D_4$ we have $f(\widetilde Q)=0$ and $\rho^{-2}\Xi[Q]\leq C$, and therefore
\begin{equation}\label{eq:E_D4}
\widetilde E_\xi(\widetilde Q;D_4)\leq \pi\ln\frac 1\sigma +C.
\end{equation}
In the subdomain $D_5$ we are again going to define $\widetilde Q$ from a unit vector field $n\colon D_5\to\mathbb S^2$. There we use polar coordinates $(s,\phi)$ centered at $(1+2\sigma,0)$, that is, given by $\rho+iz = 1+2\sigma + s e^{i\phi}$. In these coordinates the domain $D_5$ is of the form
\begin{equation*}
D_5=\left\lbrace 1+2\sigma + s e^{i\phi}\colon 0<\phi<\pi,\, \sigma<s<\bar s(\phi)\right\rbrace,
\end{equation*}
where $\bar s(\phi)$ is a Lipschitz function of $\phi$, with $2\sigma \leq \bar s \leq 5\sigma$. On the part of $\partial D_5$ given by $\lbrace s=\bar s(\phi)\rbrace$, the values of $n$ are given by the boundary conditions on $\lbrace \rho^2+z^2=1\rbrace$ and on $\partial D_4$, and are of the form
\begin{equation*}
n(\bar s(\phi),\phi)=(\cos\alpha(\varphi), 0, \sin\alpha(\varphi))\qquad\text{ for }\phi\in (0,\pi),
\end{equation*}
for Lipschitz function $\alpha\colon [0,\pi]\to\R$, which satisfies $\alpha(0)=\pi/2$ and $\alpha(\pi)=0$. On the part of $\partial D_5$ given by $\lbrace s=\sigma\rbrace$, we set
\begin{equation*}
n(\sigma,\phi)=(\sin(\phi/2),0,\cos(\phi/2))\qquad\text{ for }\phi\in (0,\pi).
\end{equation*}
Then we define $n$ in $\partial D_5$ by interpolating in the $s$ variable, i.e. we set
\begin{align*}
n(s,\phi)&=(\cos \beta(s,\phi),0,\sin\beta(s,\phi)),\\
\beta(s,\phi)&=  \frac1 2\frac{\bar s(\phi)-s}{\bar s(\phi)-\sigma}\left(\pi-\phi\right) + \frac{s-\sigma}{\bar s(\phi)-\sigma}\alpha(\phi),\qquad\text{ for }\phi \in (0,\pi).
\end{align*}
Note that, by continuity, and since $n(\sigma,0)=e_3 = n(\bar s(0),0)$ and $n(\sigma,\pi)=e_1 = n(\bar s(\pi),\pi)$, this forces the trace of $n$ on $\partial D_5\cap\lbrace z=0\rbrace$, to be given by
\begin{equation*}
n(\rho,0)=\begin{cases}
e_1 &\text{ for }1<\rho<1+\sigma,\\
e_3 &\text{ for }1+3\sigma <\rho <1+4\sigma.
\end{cases}
\end{equation*}
Moreover, since it is directly checked that $\abs{\partial_s n} + \abs{\partial_\phi n}\leq C$, we deduce 
\begin{equation*}
\int_{D_5}\abs{\nabla n}^2 \leq C \int_{\sigma}^{5\sigma} \frac{ds}{s} \leq C.
\end{equation*}
Therefore, setting $\widetilde Q=Q_n$ in $D_5$ we obtain
\begin{equation}\label{eq:E_D5}
\widetilde E_\xi(\widetilde Q;D_5)\leq C.
\end{equation}
Finally, in $D_6$ we define $\widetilde Q$ in polar coordinates $(s,\phi)$ centered at $(1+2\sigma,0)$ as above, by
\begin{align*}
\widetilde Q & = \lambda(s)Q_n, \qquad 
n =(\sin(\phi/2),0,\cos(\phi/2)),\qquad
\lambda(s)=\min(1,s/\xi).
\end{align*}
Then we have
\begin{align*}
\int_{D_6}\abs{\nabla\widetilde Q}^2 \rho\, d\rho\, dz &
\leq C +2 \int_\xi^\sigma \int_0^\pi \frac{\abs{\partial_\phi n}^2}{s^2} (1+2\sigma + s\cos\phi) s\, d\phi\, ds \\
& \leq C + \frac\pi 2 (1+2\sigma)\ln\frac\sigma\xi \\
&\leq \frac \pi 2 \ln\frac \sigma\xi  + \pi \sigma\ln\frac 1\xi + C,
\end{align*}
and therefore, since $f(\widetilde Q_n)=0$ for $s>\xi$ and $\leq C$ for $s< \xi$, we deduce that
\begin{equation}\label{eq:E_D6}
\widetilde E_\xi(\widetilde Q;D_6)\leq \frac \pi 2 \ln\frac \sigma\xi  + \pi \sigma\ln\frac 1\xi +C.
\end{equation}
Gathering \eqref{eq:E_D1}-\eqref{eq:E_D6}, we obtain
\begin{equation*}
\widetilde E_\xi(\widetilde Q;D)\leq \frac\pi 2\ln\frac 1\xi + \frac{\pi}{2}\ln\frac 1\sigma + \pi\sigma\ln\frac 1\xi + C.
\end{equation*}
Optimizing in $\sigma$, we are led to choosing
\begin{equation*}
\sigma = \frac{1}{2\ln\frac 1\xi},
\end{equation*}
and conclude that
\begin{equation*}
\widetilde E_\xi(\widetilde Q;D)\leq \frac\pi 2\ln\frac 1\xi + \frac{\pi}{2} \ln\ln \frac 1\xi + C,
\end{equation*}
which, upon applying the mirror symmetry, proves part $(i)$ of Theorem~\ref{thm:main}.

\section{Lower bound}\label{s:lower}

In this section we prove the lower bound. We use the notation $\lesssim$ to denote inequality up to a universal multiplicative constant. We denote by $\widetilde Q_\xi$ the map defined by
\begin{equation*}
Q_\xi(\rho,\varphi,z)=R_\varphi \widetilde Q_\xi(\rho,z) R_{\varphi}^t,
\end{equation*}
which satisfies in addition the mirror symmetry \eqref{eq:mirrorsym}. The map $\widetilde Q_\xi$ is thus defined in 
\begin{equation*}
\widetilde \Omega = \lbrace (\rho,z)\colon \rho^2+z^2>1,\, \rho>0\rbrace,
\end{equation*}
uniquely determined by its values in the region
\begin{equation*}
D=\lbrace \rho^2+z^2>1\colon\rho,z>0\rbrace,
\end{equation*}
and minimizes
\begin{equation*}
\widetilde E_\xi(\widetilde Q) = \int_{D}\left(\abs{\nabla\widetilde Q}^2 + \frac{1}{\rho^2}\Xi[\widetilde Q] + \frac{1}{\xi^2}f(\widetilde Q) \right)\rho\, d\rho dz
\end{equation*}
under the boundary constraints
\begin{align*}
\widetilde Q(\rho,z)&= e_r\otimes e_r -\frac 13 I\qquad\text{ for }\rho^2+z^2=1,\qquad\text{where } e_r = (\rho,0,z),\\
\text{ and }\widetilde Q(\rho,0)&=S \widetilde Q(\rho,0) S^t\qquad\text{ for }\rho>1.
\end{align*}
For any $X\in \widetilde\Omega$ we will denote by $B(X,r)$ the disc of radius $r>0$ centered at $X$.

Note that our potential $f$ satisfies
\begin{equation}\label{eq:fdistU*}
f(Q)\lesssim \dist^2(Q,\mathcal U_\star)\lesssim f(Q)\qquad\text{for }\abs{Q}\lesssim 1,
\end{equation}
hence we may fix $\eta>0$ such that in the region $\lbrace Q\in\mathcal S_0\colon  f(Q) < 2 \eta\rbrace$, the nearest neighbor projection $\pi$ onto the smooth submanifold $\mathcal U_\star$ is well defined and smooth.

\begin{lemma}\label{l:firstbounds}
We have
\begin{equation*}
\norm{Q_\xi}_{L^\infty(\Omega)}\lesssim 1\quad\text{and}\quad \norm{\nabla Q_\xi}_{L^\infty}\lesssim \frac 1\xi.
\end{equation*}
\end{lemma}
\begin{proof}
The $L^\infty$ bound is proved in \cite[Lemma~5]{colloid}, and the gradient bound follows from rescaled elliptic estimates.
\end{proof}

Away from the vertical axis of symmetry the energy is almost two-dimensional, and we exploit this observation to use $\eta$-compactness methods developed for the Ginzburg-Landau functional by Struwe \cite{struwe_94}.

\begin{lemma}\label{l:boundpotentialsmallballs}
For any $\alpha\in (0,1]$ there exists $C_\alpha>0$ such that for any $X=(\rho_0,z_0)\in \widetilde\Omega\cap\lbrace \rho\geq \frac 12\rbrace$,
\begin{equation*}
\int_{\widetilde\Omega\cap B(X,\xi^\alpha)}\left[ \frac{1}{\rho^2}\Xi(\Qt) 
+\frac{1}{\xi^2} f(\tilde Q) \right]\rho \, d\rho dz\leq C_\alpha.
\end{equation*}
\end{lemma}
\begin{proof}  Fix any $X=(\rho_0,z_0)\in \widetilde\Omega\cap\lbrace \rho\geq \frac 12\rbrace$.
Define
$$  F(r)=F(r; X):= r\int_{\partial B(X,r)\cap \widetilde\Om}
   \left[ |\nabla \Qt|^2 + \frac{1}{\rho^2}\Xi(\Qt) +\frac{1}{\xi^2} f(\Qt)\right] \rho \, ds_r,
$$
where $ds_r$ denotes arclength measure on $\partial B(X,r)$.

Fix any $\beta\in (0,\alpha)$.  As in the proof of \cite[Lemma~3.3 (i)]{struwe_94}, by Fubini's Theorem there exists $r_\xi\in (\xi^\alpha,\xi^\beta)$ for which
\begin{align*}
\Etx (\Qt) \ge \Etx(\Qt; B(X,\xi^\beta) )
    \ge \int_{\xi^\alpha}^{\xi^\beta} F(r) \frac{dr}{r} 
  \ge (\beta-\alpha) F(r_\xi) \ln \frac 1\xi.
\end{align*}
In particular,
\begin{equation}\label{eq:Fbound}
F(r_\xi) \le  (\beta-\alpha) \frac{ \Etx(\Qt; B(X,\xi^\beta)) }{ |\ln\xi|}
   \le (\beta-\alpha) \frac{ \Etx(\Qt) }{ |\ln\xi|}\le C_1.
\end{equation}

To obtain the desired bound we require a version of the Pohozaev identity.  We treat our problem as if it were two dimensional, with domain $\Omt$ and 
 a nonconstant weight function $w=\rho$.  Consider a solution $u$ of the equation
$$   -\frac{1}{ w}\nabla\cdot w\nabla u + \partial_u g(x,u)=0 \qquad x\in \R^2.  $$
For a smooth vector field $Y$, we multiply by $Y\cdot\nabla u$ and integrate by parts on $D\subset\Omt$, to obtain:
\begin{multline*}
\int_{D} \left[ \frac12 \nabla\cdot(w Y) |\nabla u|^2 - w\, (\partial_j Y_k) \partial_j u \partial_k u \right] dx  \\
\qquad +  \int_{D} \left[  g(x,u)\nabla\cdot (wY) 
+ \partial_x g(x,u)\cdot wY\right] dx \\
= \int_{\partial D} \left[ g(x,y) Y\cdot\nu + \frac12 Y\cdot\nu |\nabla u|^2
    - (\nabla u\cdot \nu)(Y\cdot \nabla u)\right] w\, ds.
\end{multline*}

In our case, we choose the domain $D=D_r= B(X,r)\cap\widetilde\Om$, and the vector field $Y=(\rho-\rho_0,z-z_0)$.  The Euler-Lagrange equations have the form
\begin{equation}\label{ELeq}  -\Delta \Qt + \frac 12 \partial_Q \left(\frac{1}{\rho^2}\Xi(\Qt) 
    + \frac{1}{ \xi^{2}} f(\Qt)\right)= L(\Qt),  
\end{equation}
where $L=\frac 16 |\widetilde Q|^2 \mathbf{I}$ is a Lagrange multiplier due to the vanishing trace condition.  Since the test functions $Y\cdot\nabla \Qt$ are trace-free, (as observed in \cite[Remark 4.3]{canevari2d},) $L(\Qt)$ plays no role in the resulting identity, and so we may take $g(x,u)= g(\rho, \Qt)= \frac{1}{2\rho^2}\Xi(\Qt) 
    + \frac{1}{2\xi^{2}} f(\Qt)$.  Substituting in the above identity, with $u=\widetilde{Q}_{\xi,ij}$ and summing over $i,j$, we obtain:
\begin{multline}\label{Id2}
I_1:=\int_{D_r} (3\rho-\rho_0) g(\rho,\Qt)d\rho\, dz + 
 \int_{D_r} \left[ \frac12 |\nabla\Qt|^2 
   -\frac{1}{\rho^2} \Xi(\Qt) \right](\rho-\rho_0)\, d\rho\, dz \\
=  \int_{\partial D_r} \left[ g(\rho,\Qt) (Y\cdot\nu) + \frac12 (Y\cdot\nu) |\nabla \Qt|^2
   - (\nabla\Qt\cdot\nu)(Y\cdot\nabla \Qt)\right] \rho\, ds_r=: I_2.
\end{multline}

Recalling $X=(\rho_0,z_0)$ with $\rho_0\ge \frac12$, we will apply the above identity (here and in the next lemma) for $r\in (\xi,\xi^\beta)$, and so for $\xi$ sufficiently small the domain $D_r$ will be strongly starshaped, in the sense that $Y\cdot \nu\ge \frac{r}{4}$ on $\partial D_r$.   Following \cite[Lemma 2.3 (ii)]{struwe_94}, the right-hand side of \eqref{Id2} may be estimated as:
\begin{equation}\label{I2}
   I_2 \le C\, F(r) + C\, r\int_{\partial\Omt\cap B(X,r)} |\nabla Q_b|^2\, ds
    \le C\, F(r) + O(r^2),
\end{equation}
with constant $C$ independent of $X$, and recalling $\Qt=Q_b$ on $\partial B(0,1)$.
For the left-hand side, we note that in $D_r$, $|\rho-\rho_0|<r\le 3\rho\xi^\beta$, and hence
$$  I_1 \ge 2\int_{D_{r}} g(\rho,\Qt) \, \rho \, d\rho\, dz - O(\xi^\beta |\ln\xi|).   $$
Thus, we have for any $r\in (\xi,\xi^\beta)$,
\begin{equation}\label{nicebound}
 \int_{D_{r}} g(\rho,\Qt) \, \rho \, d\rho\, dz \le C\, F(r) + O(\xi^\beta |\ln \xi|).  
\end{equation}
Choosing $r=r_\xi$ as in \eqref{eq:Fbound}, and noting $B(X,\xi^\alpha)\subset B(X,r_\xi)$, the desired inequality is established. 
\end{proof}

The following is an adaptation of the $\eta$-compactness ($\eta$-ellipticity) condition to our setting.
\begin{lemma}\label{l:pointwiseboundpotential}
There exists $\gamma>0$ such that, for any $\alpha\in (0,1]$ there is $\xi_0(\alpha)>0$ with the following property. If $\xi\in (0,\xi_0)$ and $r\in [\xi,\xi^\alpha]$ are such that
\begin{equation*}  F(r;X) \leq \gamma \rho(X)\qquad\text{for some }X\in \widetilde\Omega\cap\lbrace \rho\geq \frac 12\rbrace,
\end{equation*}
then $f(\widetilde Q)\leq\eta$ in $B_r(X)\cap \widetilde\Omega$.
\end{lemma}
\begin{proof}  The proof is as in \cite[Lemma~2.3(ii)]{struwe_94}.
Suppose the contrary:  there exists $X'=(\rho',z')\in B(X,r)$ for which $f(\Qt(X'))>\eta$.  By Lemma~\ref{l:firstbounds}, there exists $c>0$ for which $f(\Qt(x))>\eta/2$ for all $x\in B(X',c\xi)$.  Thus, there is a constant $C_0>0$, independent of $X,\xi$,  for which 
\begin{equation}\label{contradict}
\int_{B(X',c\xi)\cap\Omt} \frac{1}{\xi^2} f(\Qt)\, \rho\, d\rho\, dz \ge C_0\rho(X').
\end{equation}
On the other hand, by \eqref{nicebound} we then have
$$  C_0\rho(X')\le  \int_{B(X',c\xi)\cap\Omt} \frac{1}{\xi^2} f(\Qt)\, \rho\, d\rho\, dz
   \le \int_{D_{r}} g(\rho,\Qt) \, \rho \, d\rho\, dz \le C\, F(r) \le\gamma\rho(X).  $$
For any $\gamma<C_0$ this is impossible, as $|\rho(X')-\rho(X)| < r \ll 1$, and hence the conclusion must hold.
\end{proof}

As in the Ginzburg-Landau case, we may now define the ``bad balls'' which contain the eventual defects:

\begin{lemma}\label{l:boundedbadballs}
There exist $M_0\in\mathbb N$ and $A_0\geq 1$ such that 
\begin{equation*}
\lbrace f(\widetilde Q_\xi)>\eta\rbrace\subset \lbrace \rho\leq A_0\rbrace,
\end{equation*}
and for any disjoint collection of balls $\lbrace B(X_j,\frac\xi 5)\rbrace_{j\in J}$ with centers
\begin{equation*}
X_j\in \lbrace f(\widetilde Q_\xi)>\eta\rbrace \cap \lbrace \rho\geq \frac 12\rbrace,
\end{equation*}
 the cardinality of $J$ must be $\leq M_0$.
\end{lemma}
\begin{proof}
For the first assertion, let $X'\in S_\xi:=\{x\in\Omt: \ f(\Qt(x))>\eta)\}$.  As in the proof of Lemma~\ref{l:pointwiseboundpotential}, there exists $c>0$ for which $f(\Qt(x))>\eta/2$ in $B(X',c\xi)$, with the same lower bound \eqref{contradict}.  Taking $r=r_\xi$ as in 
\eqref{eq:Fbound}, we obtain $C_0\rho(X')\le CF(r_\xi)\le C'$, and hence $\rho(X')$ is uniformly bounded in $\xi$.

The second assertion now follows as in the proof of \cite[Lemma~3.2]{struwe_94}.  Indeed, since for all $X\in S_\xi$, $\frac12\le \rho(X)\le A_0$ by the first assertion, the value of $\gamma$ in Lemma~\ref{l:pointwiseboundpotential} may be chosen independently of $X\in S_\xi$ to give an upper bound which is uniform in $X$.  Similarly,  the lower bound \eqref{contradict} on the potential term in balls $B(X,c\xi)$, may be chosen independently of $X\in S_\xi$.  Hence, the identical arguments (based on Vitali covering of $S_\xi$) as in \cite{struwe_94} assure the bounded cardinality of the collection of ``bad balls'' $\lbrace B(X_j,\frac\xi 5)\rbrace_{j\in J}$.
\end{proof}
Recall that given any Lipschitz simply connected bounded domain $R\subset\R^2$ and any continuous map $U\colon \partial R\to \mathcal U_\star$ we can consider its homotopy class in $\pi_1(\mathcal U_\star)\approx\mathbb Z/2\mathbb Z$. A loop is trivial in $\pi_1(\mathcal U_\star)$ if and only if it is orientable, i.e. it is of the form $\gamma\otimes \gamma-\frac 13 I$, for some continuous loop $\gamma\colon \mathbb S^1\to\mathbb S^2$.

\begin{lemma}\label{l:nontrivialrectangle}
Consider for some $z_0\in (0,1/2]$ and $\rho_0\geq A_0$ the domain 
\begin{equation*}
R=\lbrace \abs{z}<z_0,\; \rho< \rho_0\rbrace\cap \widetilde\Omega.
\end{equation*}
and assume that $f(\widetilde Q_\xi)\leq \eta$ on $\partial R$ and that $\widetilde Q_\xi$ restricted to $\partial R$ is continuous. Then the projected map
\begin{equation*}
U_\xi=\pi(\widetilde Q_\xi)\colon\partial R\to \mathcal U_\star,
\end{equation*}
has non trivial homotopy class in $\pi_1(\mathcal U_\star)$, that is,
 $U_\xi$ is non-orientable.
\end{lemma} 
\begin{proof}[Proof of Lemma~\ref{l:nontrivialrectangle}]
Recall that $Q_\xi\in \mathcal H_{sym}$, hence
\begin{equation*}
\int_{\widetilde\Omega}\left(\abs{\nabla\widetilde Q_\xi}^2 + \frac{\abs{\widetilde Q_\xi-Q_\infty}^2}{\rho^2+z^2}\right) \rho d\rho dz <\infty.
\end{equation*}
In particular, for any $\xi\in (0,\xi_0)$ and any $\delta>0$ we may choose $\rho_1\geq \rho_0$ such that
\begin{equation*}
\int_{-1}^1 \abs{\partial_z \widetilde Q_\xi(\rho_1,z)}^2 dz + \int_{-1}^1\abs{\widetilde Q_\xi(\rho_1,z)-Q_\infty}^2 dz <\delta,
\end{equation*}
which implies
\begin{equation}\label{eq:Qrho1}
\abs{\widetilde Q_\xi(\rho_1,z)-Q_\infty}^2\lesssim\delta\quad\forall z\in [-1,1].
\end{equation}
We denote by $R_1$ the domain
\begin{equation*}
R_1=\widetilde\Omega\cap\lbrace \abs{z}<z_0,\, \rho <\rho_1\rbrace,
\end{equation*}
and define $V_\xi=\pi(\widetilde Q_\xi)\colon\partial R_1\to\mathcal U_\star$. Since $\pi(\widetilde Q_\xi)$ is well defined and has finite energy in $R_1\setminus R_0$, the maps $U_\xi$ and $V_\xi$ are homotopically equivalent.  To prove Lemma~\ref{l:nontrivialrectangle} it thus suffices to prove that $V_\xi$ is non orientable. Assume that $V_\xi$ is orientable: there exists a continuous $n\colon \partial R_1\to\mathbb S^2$ such that
\begin{equation*}
V_\xi = n\otimes n -\frac 13 I.
\end{equation*}
Since $n$ is uniquely defined up to a sign and $V_\xi=e_1\otimes e_1-\frac 13 I$ at $(\rho,z)=(1,0)$, we may assume that $n(1,0)=e_1$.
The symmetry assumption \eqref{eq:mirrorsym} implies that
\begin{equation*}
n(\rho,-z)=\tau S n(\rho,z)\qquad\text{for some }\tau\in\lbrace \pm 1\rbrace.
\end{equation*}
Evaluating this at $(\rho,z)=(1,0)$ gives $e_1=\tau Se_1 =\tau e_1$, hence $\tau=1$. This implies that at $(\rho,z)=(\rho_1,0)$ one must have $n(\rho_1,0)=Sn(\rho_1,0)$, i.e. $n(\rho_1,0)\perp e_3$, and therefore $\abs{V_\xi(\rho_1,0)-Q_\infty}^2= 2$. For small enough $\delta$ this contradicts \eqref{eq:Qrho1}.
\end{proof}

The next lemmas deal with universal lower bounds in annular regions.

\begin{lemma}\label{l:lowerboundannulus}
There exists $C>0$ such that for any $\xi>0$, for any annulus $\omega\subset\R^2$ of the form
\begin{equation*}
\omega = B(0,R)\setminus B(0,r),\qquad 0<r<R\leq \frac 12,
\end{equation*}
and any $H^1$ map $Q\colon\overline \omega\to\mathcal S_0$ satisfying $f(Q)\leq \eta$ in $\omega$ and such that $U=\pi(Q)\colon\overline\omega\to\mathcal U_\star$ is continuous and nonorientable on $\partial B(0,R)$, we have
\begin{equation*}
\int_\omega \left(\abs{\nabla Q}^2+\frac{1}{\xi^2}f(Q)\right) \geq \pi \ln \frac{R}{r} - C \xi\left(\frac 1 r-\frac 1 R\right).
\end{equation*}
\end{lemma}
\begin{proof}[Proof of Lemma~\ref{l:lowerboundannulus}]
This is proved in \cite{canevari2d} and \cite{golovatymontero14}, but for completeness we include here a proof, following the method in \cite{jerrard99}.

Let $D=\abs{Q-U}=\dist(Q,\mathcal U_*)$, and denote by $P\colon\mathcal U_\star\to \mathcal L(\mathcal S_0)$ the smooth map given by $P(u)=P_{(T_u\mathcal U_\star)^\perp}$  the orthogonal projection onto the normal space to $\mathcal U_\star$ at $u$. Note that by definition $Q-U=P(U)(Q-U)$. Then for any direction $k$ we compute
\begin{align*}
\abs{\partial_k Q}^2 &=\abs{\partial_k U}^2 +\abs{\partial_k (Q-U)}^2 + 2\partial_k U\cdot \partial_k (Q-U)\\
& \geq \abs{\partial_k U}^2 +\abs{\partial_k D}^2 + 2\partial_k U\cdot \partial_k \left[ P(U)(Q-U)\right]\\
&= \abs{\partial_k U}^2 +\abs{\partial_k D}^2 + 2\partial_k U\cdot \partial_k \left[ P(U) \right] (Q-U) + 2\partial_k U\cdot P(U) \partial_k(Q_U).
\end{align*}
The last term is zero because $\partial_k U\in T_U\mathcal U_\star$ and therefore $P(U)\partial_k U=0$. So we have
\begin{align*}
\abs{\partial_k Q}^2 
& \geq \abs{\partial_k U}^2 +\abs{\partial_k D}^2 - 2\norm{DP(U)}D \abs{\partial_k U}^2 \\
&\geq  (1-c D)\abs{\partial_k U}^2 +\abs{\partial_k D}^2,
\end{align*}
for $c=2\sup_{\mathcal U_\star}\norm{DP(U)}$. This computation is very similar to \cite[Lemma~2.6]{canevari2d} and \cite[Lemma~4]{golovatymontero14}. Recalling now that $f(Q)\geq \alpha^2 D^2$ for some $\alpha>0$ we find
\begin{align*}
\abs{\partial_k Q}^2+\frac{1}{\xi^2}f(Q) &\geq (1-c D)\abs{\partial_k U}^2 + \abs{\partial_k D}^2 +\frac{\alpha^2}{\xi^2}D^2 \\
& \geq  (1-cD)\abs{\partial_k U}^2 + \frac{\alpha}{\xi}\abs{\partial_k [D^2]}.
\end{align*}
Hence for any $s\in (0,R)$ we have, letting $d(s)=\max_{\partial B(0,s)} D$,
\begin{align*}
\int_{\partial B(0,s)}\left(\abs{\nabla Q}^2 +\frac{1}{\xi^2}f(Q) \right) &
\geq (1-c d(s))\int_{\partial B(0,s)}\abs{\partial_\tau U}^2
 + \frac{\alpha}{\xi}d(s)^2
\end{align*}
Since $U$ is $H^1$ in $\omega$ and non orientable on $\partial B(0,R)$, it is continuous and nonorientable on $\partial B(0,s)$ for a.e. $s\in [r,R]$, and for such $s$ we have (see e.g. \cite[Corollary~3.8]{canevari2d})
\begin{equation*}
\int_{\partial B(0,s)}\abs{\partial_\tau U}^2 \geq \frac{\pi}{s},
\end{equation*}
and therefore
\begin{align*}
\int_{\partial B(0,s)}\left(\abs{\nabla Q}^2 +\frac{1}{\xi^2}f(Q) \right) &
\geq \frac{\pi}{s}(1-c d(s))  + \frac{\alpha}{\xi}d(s)^2 \\
&\geq \inf_{d\in [0,1]}\left( \frac{\pi}{s}(1-c d)  + \frac{\alpha}{\xi}d^2\right)\\
&=\frac\pi s -\frac{\xi}{s^2}\cdot \begin{cases}
\frac{c^2\pi^2}{4\alpha} &\quad\text{if }\frac{\xi}{s}\leq \frac{c\pi}{2\alpha},\\
\frac{s}{\xi}\left(c\pi -\alpha\frac{s}{\xi}\right) &\quad\text{if }\frac\xi s\geq \frac{c\pi}{2\alpha}
\end{cases}\\
&\geq \frac \pi s -\frac{c^2\pi^2}{4\alpha}\frac{\xi}{s^2}.
\end{align*}
Integrating we deduce
\begin{align*}
\int_\omega \left(\abs{\nabla Q}^2 +\frac{1}{\xi^2}f(Q) \right) &
\geq \pi\ln\frac Rr -\frac{c^2\pi^2}{4\alpha}\xi\left(\frac 1r-\frac 1R\right).
\end{align*}
\end{proof}

Using Lemma~\ref{l:lowerboundannulus} and the growing ball construction of Jerrard or Sandier \cite{jerrard99,sandier98} (which is adapted to our setting in \cite[Lemma~7]{golovatymontero14}), one obtains the following lower bound on perforated domains:
\begin{lemma}\label{l:lowerboundperforated}
There exists $C>0$ such that for any $\xi\in (0,1]$, for any perforated domain $\omega\subset\R^2$ of the form
\begin{equation*}
\omega = B(0,R)\setminus \bigcup_{j=1}^N B(x_j,r),\qquad B(x_j,r)\subset B(0,\frac R2)\text{ disjoint}, \;  r\geq \xi,
\end{equation*}
and any $H^1$ map $Q\colon\overline \omega\to\mathcal S_0$ satisfying $f(Q)\leq \eta$ in $\omega$ and such that $U=\pi(Q)\colon\overline\omega\to\mathcal U_\star$ is continuous and nonorientable on $\partial B(0,R)$, we have
\begin{equation*}
\int_\omega \left(\abs{\nabla Q}^2+\frac{1}{\xi^2}f(Q)\right) \geq \pi \ln \frac{R}{Nr} - C.
\end{equation*}
\end{lemma}

We will also need a boundary version of Lemma~\ref{l:lowerboundannulus}:
\begin{lemma}\label{l:lowerboundhalfannulus}
There exists $C>0$ such that for any $\xi>0$, for any annulus $\omega\subset\R^2$ of the form
\begin{equation*}
\omega = B(0,R)\setminus B(0,r),\qquad 0<r<R\leq \frac 12,
\end{equation*}
and any $H^1$ map $Q\colon\overline \omega\to\mathcal S_0$ satisfying $f(Q)\leq \eta$ in $\omega$ and such that $U=\pi(Q)\colon\overline\omega\to\mathcal U_\star$ is continuous and nonorientable on $\partial B(0,R)$, if in addition $Q=U_0$ on the left half $\omega_- = \omega\cap \lbrace (x_1,x_2)\colon x_1<0\rbrace$ of the annulus, for some Lipschitz $U_0\colon \R^2\to \mathcal U_\star$ with $\abs{\nabla U_0}\leq 1$, then for  any function $w$ such that $w(x)\geq 1-\abs{x}^2$,
\begin{equation*}
\int_{\omega_+} \left(\abs{\nabla Q}^2+\frac{1}{\xi^2}f(Q)\right)w(x)\, dx \geq 2\pi \ln \frac{R}{r} - C \left(1+ \frac{\xi}{r}\right),
\end{equation*}
where $\omega_+ =\omega\cap \lbrace (x_1,x_2)\colon x_1>0\rbrace$ is the right half of the annulus.
\end{lemma}
\begin{proof}[Proof of Lemma~\ref{l:lowerboundhalfannulus}]
We claim that 
\begin{align}\label{eq:halfgeod}
\int_{\partial B(0,s)\cap\omega_+}\abs{\partial_\tau U}^2\geq \frac{2\pi}{s}-16\qquad\text{for a.e. }s\in [r,R].
\end{align}
Then, going through the computations in the proof of Lemma~\ref{l:lowerboundannulus} gives
\begin{align*}
\int_{\omega_+}\left(\abs{\nabla Q}^2 +\frac{1}{\xi^2}f(Q) \right)w(x)\,dx &
\geq \int_{r}^{R}\left(\frac{2\pi}{s} -\frac{c^2\pi^2}{\alpha}\frac{\xi}{s^2}-16\right)(1-s^2)\, s\, ds
\\
&\geq 2\pi\ln\frac Rr -C\left(\frac{c^2\pi^2}{\alpha}\frac\xi r +1\right),
\end{align*}
which proves Lemma~\ref{l:lowerboundhalfannulus}, provided we show \eqref{eq:halfgeod}. We fix $s\in [r,R]$ such that $U_{\lfloor\partial B(0,s)}$ is in $H^1$ and non-orientable. Since $\R^2$ and $\partial B(0,s)\cap\overline\omega_+$ are simply connected we may orient $U_0$ and $U$ on these domains, i.e. write 
\begin{align*}
U_0&=n_0\otimes n_0-\frac 13 I, & n&\colon\R^2\to\mathbb S^2,\\
U(se^{i\theta})&=n_s(\theta)\otimes n_s(\theta)-\frac 13 I & n_s &\colon [-\frac\pi 2,\frac\pi 2]\to\mathbb S^2.
\end{align*}
Since $U(se^{-i\frac\pi 2})=U_0(se^{-i\frac\pi 2})$ we may, up to switching the orientation of $n_0$, assume that
\begin{equation*}
n_s(-\frac\pi 2)=n_0(se^{-i\frac\pi 2}).
\end{equation*}
Since $U(se^{i\frac\pi 2})=U_0(se^{i\frac\pi 2})$ we have $n_s(\frac\pi 2)=\pm n_0(se^{\frac i\pi 2})$, and because $U$ is non-orientable it must be
\begin{equation*}
n_s(\frac\pi 2)=-n_0(se^{i\frac\pi 2}).
\end{equation*}
Note also that, as $\abs{\nabla n_0}=\frac 1{\sqrt 2}\abs{\nabla U_0}\leq 1$, we have
\begin{equation*}
\abs{n_0(se^{-i\frac\pi 2})-n_0(e^{i\frac\pi 2})}\leq 2s\leq 1,
\end{equation*}
which implies that
\begin{equation*}
\dist_{\mathbb S^2}(n_0(se^{-i\frac\pi 2}),-n_0(se^{i\frac\pi 2}))\geq \pi- 4 s
\end{equation*}
Hence we find
\begin{align*}
\int_{\partial B(0,s)\cap\omega_+}\abs{\partial_\tau U}^2 &=\frac 1s \int_{-\frac\pi 2}^{\frac\pi 2}2\abs{n_s'(\theta)}^2\,d\theta \\
&\geq \frac 1s \frac 2{\pi} \left(\int_{-\frac\pi 2}^{\frac\pi 2}\abs{n_s'(\theta)}d\theta\right)^2\\
&\geq \frac 1s \frac{2}{\pi}\dist_{\mathbb S^2}(n_s(-\frac\pi 2),n_s(\frac\pi 2))^2 \\
&\geq\frac 1s \frac{2\pi^2 -16\pi s + 32 s^2}{\pi}\\
&\geq \frac{2\pi}{s}-16,
\end{align*}
thus proving \eqref{eq:halfgeod}.
\end{proof}

In Section~\ref{s:lim} we will also need the following refinement of Lemma~\ref{l:lowerboundhalfannulus}.  In particular, we show that (to $O(1)$) we obtain the same lower bound in a smaller domain which pulls away from the boundary of the half-disk along a curve $x_1=\lambda x_2^\beta$ with $\lambda>0$ and $\beta>1$.

\begin{lemma}\label{l:finelowerboundhalfannulus}
There exists $C>0$ such that the following holds. Let $\xi>0$, $\omega\subset\R^2$ of the form
\begin{equation*}
\omega = B(0,R)\setminus B(0,r),\qquad 0<r<R\leq \frac 12,
\end{equation*}
 $Q\colon\overline \omega\to\mathcal S_0$  an $H^1$ map satisfying $f(Q)\leq \eta$ in $\omega$, such that $U=\pi(Q)\colon\overline\omega\to\mathcal U_\star$ is continuous and nonorientable on $\partial B(0,R)$, and $Q=U_0$ on the left half $\omega_- = \omega\cap \lbrace (x_1,x_2)\colon x_1<0\rbrace$ of the annulus, for some Lipschitz $U_0\colon \R^2\to \mathcal U_\star$ with $\abs{\nabla U_0}\leq 1$. Let also $w$ be a function such that $w(x)\geq 1-\abs{x}^2$. Then, in the right half annulus $\omega_+ =\omega\cap \lbrace (x_1,x_2)\colon x_1>0\rbrace$, we have the lower bound provided by Lemma~\ref{l:lowerboundhalfannulus}. If $Q$ also satisfies a matching upper bound, in the sense that 
\begin{align*}
\int_{\omega_+} \left(\abs{\nabla Q}^2+\frac{1}{\xi^2}f(Q)\right)w(x)\, dx \leq 2\pi \ln \frac{R}{r} +K,
\end{align*}
for some $K>0$, then the lower bound is also valid in the slightly smaller domain 
\begin{equation*}
\tilde\omega_\beta =\omega\cap \lbrace (x_1,x_2)\colon x_1>\lambda x_2^\beta\rbrace,\qquad\lambda>0,\beta>1,
\end{equation*}
in the sense that
\begin{equation*}
\int_{\tilde\omega_\beta} \left(\abs{\nabla Q}^2+\frac{1}{\xi^2}f(Q)\right)w(x)\, dx \geq 2\pi \ln \frac{R}{r} - C \left(1+ \frac{\xi}{r}\right) -CK-\frac{C\lambda}{\beta-1}.
\end{equation*}
\end{lemma}
\begin{proof}[Proof of Lemma~\ref{l:finelowerboundhalfannulus}]
We use the same notations as in the proof of Lemma~\ref{l:lowerboundhalfannulus}, and refine the lower bound obtained on each slice $\partial B(0,s)\cap\omega_+$. The crucial computation is 
\begin{align*}
\int_{-\frac\pi 2}^{\frac\pi 2}\left\vert\abs{n'_s(\theta)}-1\right\vert^2\, d\theta &= \int_{-\frac\pi 2}^{\frac\pi 2}\abs{n'_s(\theta)}^2\, d\theta -2\int_{-\frac\pi 2}^{\frac\pi 2}\abs{n'_s(\theta)}\, d\theta +\pi \\
&\leq \int_{-\frac\pi 2}^{\frac\pi 2}\abs{n'_s(\theta)}^2\, d\theta -2\dist_{\mathbb S^2}(n_s(-\frac\pi 2),n_s(\frac\pi 2)) +\pi \\
&\leq \int_{-\frac\pi 2}^{\frac\pi 2}\abs{n'_s(\theta)}^2\, d\theta -2(\pi-4s) +\pi\\
&\leq \int_{-\frac\pi 2}^{\frac\pi 2}\abs{n'_s(\theta)}^2\, d\theta -\pi +8s.
\end{align*}
This can be interpreted as a stability estimate for geodesics in $\mathbb S^2$: if the lower bound for the energy of a curve (minimized by constant speed geodesics) is almost saturated, then this curve's speed must be close to constant. Plugging this back into the estimates performed in Lemma~\ref{l:lowerboundannulus}, we deduce
\begin{align*}
&\int_{\partial B(0,s)\cap\omega_+}\left(\abs{\nabla Q}^2 +\frac{1}{\xi^2}f(Q) \right)\\
&\geq (1-cd(s))\left(\frac{2\pi}{s}-16 + \frac{2}{s} \int_{-\frac\pi 2}^{\frac\pi 2}\left\vert\abs{n'_s(\theta)}-1\right\vert^2\, d\theta\right)+\frac{\alpha}{\xi}d(s)^2,
\end{align*}
where recall that $c>0$ depends only on $\mathcal U_\star$ and $0\leq d(s)\lesssim \eta$. Hence, possibly lowering $\eta$ we can assume $cd(s)\leq\frac 12$, and arguing as in Lemma~\ref{l:lowerboundannulus} we obtain
\begin{align*}
&\int_{\partial B(0,s)\cap\omega_+}\left(\abs{\nabla Q}^2 +\frac{1}{\xi^2}f(Q) \right)\\
&\geq (1-cd(s))\frac{2\pi}{s} +\frac{\alpha}{\xi}d(s)^2 -16 + \frac{1}{s} \int_{-\frac\pi 2}^{\frac\pi 2}\left\vert\abs{n'_s(\theta)}-1\right\vert^2\, d\theta \\
&\geq \frac{2\pi}{s}-\frac{c^2\pi^2}{\alpha}\frac{\xi}{s^2} -16 + \frac{1}{s} \int_{-\frac\pi 2}^{\frac\pi 2}\left\vert\abs{n'_s(\theta)}-1\right\vert^2\, d\theta.
\end{align*}
Integrating we deduce
\begin{align*}
\int_{\omega_+}\left(\abs{\nabla Q}^2 +\frac{1}{\xi^2}f(Q) \right)w(x)\,dx &
\geq 2\pi\ln\frac Rr -C\left(\frac\xi r +1\right) \\
&\quad + \int_{r}^{R} \int_{-\frac\pi 2}^{\frac\pi 2}\left\vert\abs{n'_s(\theta)}-1\right\vert^2 d\theta \frac{ds}{s}.
\end{align*}
Combining this with the assumption that we have a matching upper bound gives
\begin{align}\label{eq:dnclosetoconstant}
 \int_{r}^{R} \int_{-\frac\pi 2}^{\frac\pi 2}\left\vert\abs{n'_s(\theta)}-1\right\vert^2 d\theta \frac{ds}{s} &\leq K + C\left(\frac\xi r +1\right).
\end{align}
We will use this to show that the part that we \enquote{forget} by integrating over $\tilde\omega_\beta$ instead of $\omega_+$ is bounded. Indeed we have
\begin{align*}
\int_{\partial B(0,s)\cap\tilde\omega_\beta}\abs{\partial_\tau U}^2 
& = \int_{\partial B(0,s)\cap\omega_+}\abs{\partial_\tau U}^2 
-\int_{\partial B(0,s)\cap\omega_+\setminus \tilde\omega_\beta}\abs{\partial_\tau U}^2 \\
& \geq 
\frac{2\pi}{s}-16 \\
&\quad
-\frac{2}{s}\int_{-\frac\pi 2}^{-\frac\pi 2 + 2\lambda s^{\beta-1}}\abs{n_s'(\theta)}^2\,d\theta -\frac{2}{s}\int_{\frac\pi 2-2\lambda s^{\beta-1}}^{\frac\pi 2}\abs{n_s'(\theta)}^2\,d\theta \\
&\geq \frac{2\pi}{s}-16 -\frac{2}{s}\int_{-\frac\pi 2}^{\frac\pi 2}\left\vert\abs{n'_s(\theta)}-1\right\vert^2 d\theta - 4\lambda s^{\beta-2}.
\end{align*}
Notice that since $\beta>1$ the last term is summable with respect to $s$ small .
Hence upon arguing as above we find
\begin{align*}
\int_{\tilde\omega_\beta}\left(\abs{\nabla Q}^2 +\frac{1}{\xi^2}f(Q) \right)w(x)\,dx &\geq 2\pi\ln\frac Rr -C\left(\frac\xi r +1\right) -\frac{4\lambda}{\beta-1}\\
&\quad  -2\int_{r}^{R} \int_{-\frac\pi 2}^{\frac\pi 2}\left\vert\abs{n'_s(\theta)}-1\right\vert^2 d\theta \frac{ds}{s},
\end{align*}
and combining this with \eqref{eq:dnclosetoconstant} enables us to conclude.
\end{proof}

We are finally ready to prove the lower bound on the energy stated in Theorem~\ref{thm:main}.

\begin{proof}[Proof of Part (ii) of Theorem~\ref{thm:main}] 
We start by fixing a Lipschitz  $\mathcal U_\star$-valued extension $U_0$ of $\widetilde Q_\xi$ outside of $\widetilde\Omega$.
In particular, for any circle $\partial B(a,r)$ such that $f(\widetilde Q_\xi)<\eta$ on $\partial B(a,r)\cap \widetilde \Omega$, it makes sense to consider $U_\xi=\pi(\widetilde Q_\xi)$ on $\partial B(a,r)$  (even if part of this circle lies outside of $\widetilde \Omega$). Moreover $U_\xi$ is orientable on $\partial B(a,r)$ if and only if it is orientable on $\partial (B(a,r)\cap\widetilde\Omega)$.

Thanks to Lemma~\ref{l:boundedbadballs} and arguing as in \cite[Theorem~IV.1]{bbh}, there exist $\lambda>0$ and a family of disjoint balls $B(x_j^\xi,\lambda\xi)$ such that
\begin{align*}
&\lbrace f(\widetilde Q_\xi)>\eta\rbrace\cap\lbrace\rho\geq \frac 12\rbrace\subset\bigcup_{j=1}^{J_\xi} B(x_j^\xi,\lambda\xi),\qquad J_\xi\leq M_0,\\
&\abs{x_i^\xi-x_j^\xi}>8\lambda\xi \qquad\forall i\neq j,\\
&x_i^\xi\in\lbrace z=0\rbrace \quad\text{or}\quad \dist(x_i^\xi,\lbrace z=0\rbrace)>8\lambda\xi\qquad\forall i,\\
& x_i^\xi\in\partial\widetilde\Omega \quad\text{or}\quad \dist(x_i^\xi,\partial\widetilde\Omega)>8\lambda\xi\qquad\forall i
\end{align*}

Let us now consider a sequence $\xi=\xi_n\to 0$. To keep notation simple we will not write explicitly the dependence on $n$, and will not relabel subsequences. We may extract converging subsequences of the $x_j^\xi$ that lie inside $\lbrace \abs{z}\leq \frac 12,\rho\leq A_0\rbrace$, and we denote by $a_1,\ldots,a_K$ those of the limit points that lie in the axis $\lbrace z=0\rbrace$, and $z_1=\min(\frac 12,\frac d2)$, where $d$ is the minimal distance of any other limit point to the axis $\lbrace z=0\rbrace$ or of any of the $a_j$'s to the other $a_j$'s. Then for any $\delta\in (0,z_1)$  the balls $B(a_j,\delta)$ are disjoint and we have
\begin{equation*}
\lbrace f(\widetilde Q_\xi)>\eta\rbrace\cap\lbrace\abs{z}\leq z_1\rbrace\subset\bigcup_{j=1}^K B(a_j,\delta)\qquad\text{for small enough }\xi.
\end{equation*}
By Fubini's theorem we may find $z_0\in [z_1/2,z_1]$ and $\rho_0\geq A_0$ such that $\widetilde Q_\xi$ restricted to $\partial R$ is continuous, where $R=\lbrace \abs{z}<z_0,\rho<\rho_0\rbrace\cap\widetilde \Omega$. By the above $f(\widetilde Q_\xi)\leq\eta$ on $\partial R$ for small enough $\xi$, so we may apply Lemma~\ref{l:nontrivialrectangle} to deduce that the $U_\xi=\pi(\widetilde Q_\xi)$ is non-orientable on $\partial R$. As a consequence,  $U_\xi$ must be non-orientable on $\partial B(a_j,z_1/2)$ for at least one index $j$.

Relabeling we assume $j=1$. We claim that $a_1$ must be the leftmost possible point $a=(1,0)$. Assume indeed that $\rho(a_1)>1$ and fix $\delta\in (0,z_1)$ such that $\rho(a_1)-\delta>1$. Then applying Lemma~\ref{l:lowerboundperforated} on $B(a_1,\delta)\setminus \bigcup_{j} B(x_j^\xi,\lambda\xi)$ for small enough $\xi$ we deduce that
\begin{equation*}
\widetilde E_\xi(\widetilde Q_\xi)\geq (\rho(a_1)-\delta)\pi\ln \frac \delta\xi -C,
\end{equation*}
but since $\rho(a_1)-\delta>1$ this implies that $\widetilde E_\xi(\widetilde Q_\xi)-\pi\ln\frac 1\xi\to\infty$ as $\xi\to 0$, thus contradicting the upper bound obtained in Section~\ref{s:upper}. 

Gathering the above, we have that $a=(1,0)$ is a limit of some $x_j^\xi$ and for any $\delta>0$, if $\xi$ is small enough then $U_\xi=\pi(\widetilde Q_\xi)$ is well defined and non-orientable on $\partial B(a,\delta)$. Of the above defined $x_j^\xi$, we now consider only those which converge to $a$. 

Since there is a bounded number of $x_j^\xi$, arguing as in \cite{beaulieu_hadiji_95} one may find a bounded $\mu>1$ and radii $\sigma_i^\xi$ such that
\begin{align*}
& \xi =  \sigma_0^\xi \ll \cdots \ll \sigma_L^\xi = \delta,\qquad L\leq M_0,
\end{align*}
and each ball $B(\sigma_i^\xi,\lambda\xi)$ is contained either in $B(a,\mu\sigma_0)$ or for some $\ell\geq 1$ in  the annulus $B(a,\mu\sigma_{\ell})\setminus B(a,\sigma_\ell)$ for $\ell\geq 1$. Inside each annulus $A_\ell =B(a,\sigma_{\ell})\setminus B(a,\mu\sigma_{\ell-1})$  the map $U_\xi$ is well defined and its homotopy class on $C_s=\partial B(a,s)$ is constant for $s\in [\mu\sigma_{\ell-1},\sigma_\ell]$. We will refer to this constant homototpy class as the homotopy class of $U_\xi$ on the annulus $A_\ell$. For $\ell=L$ we know that $U_\xi$ is non-orientable on $A_L$.

We claim that there exists $\ell_0\in\lbrace 1,\ldots,L\rbrace$ such that $U_\xi$ is orientable on $A_\ell$. Otherwise, $U_\xi$ is non-orientable on all annuli $A_\ell$ which lie outside $B(a,\mu\xi)$.
Thus, flattening the boundary $\partial\widetilde\Omega$ we may for small enough $\delta$ apply Lemma~\ref{l:lowerboundhalfannulus} on each annulus. We deduce the lower bound
\begin{align*}
\widetilde E(\widetilde Q_\xi;\widetilde\Omega\cap B(a,\delta)\setminus B(a,\mu\xi))&\geq \sum_{\ell=1}^L 2\pi\ln\frac{\sigma_\ell}{\mu\sigma_{\ell-1}} -C \\
& \geq 2\pi\ln\frac{\delta}{\xi} -C,
\end{align*}
and this contradicts the upper bound obtained in Section~\ref{s:upper}.

Let us then fix the largest $\ell_0\in\lbrace 1,\ldots,L-1\rbrace$ such that $U_\xi$ is orientable on $A_\ell$, and set $\hat\sigma_\xi = \sigma_{\ell_0}^\xi$. In particular $U_\xi$ is non-orientable on $A_\ell$ for all $\ell\geq \ell_0+1$, hence arguing as above we find the lower bound
\begin{align}
\widetilde E(\widetilde Q_\xi;\widetilde\Omega\cap B(a,\delta)\setminus B(a,2\mu\hat \sigma_\xi))&\geq \sum_{\ell=\ell_0+1}^L 2\pi\ln\frac{\sigma_\ell}{\mu\sigma_{\ell-1}} -C \nonumber\\
& \geq 2\pi\ln\frac{\delta}{\hat\sigma_\xi} -C.\label{eq:lowerboundlargeannulus}
\end{align}
Moreover, since $U_\xi$ is orientable on $\partial B(a,\hat\sigma_\xi)$ and non-orientable on $\partial B(a,\mu\hat\sigma_\xi)$, there must be at least one $x_j^\xi$ such that $B(x_j^\xi,\lambda\xi)\subset B(a,\mu\hat\sigma_\xi)\setminus B(a,\hat\sigma_\xi)$ and $U_\xi$ is non-orientable on $\partial B(x_j^\xi,\lambda\xi)$.

Now we consider only those $x_j^\xi$ which are in $B(a,\mu\hat\sigma_\xi)\setminus B(a,\hat\sigma_\xi)$. Considering the balls $B(x_j^\xi,c_0\hat\sigma_\xi)$ for some small enough $c_0>0$ (depending only on the number of $x_j^\xi$'s) and applying the procedure in \cite[Theorem~IV.1]{bbh}, we obtain a bounded number $c_0\leq\gamma\leq \frac{1}{16}$ and a collection of balls $B(y_j^\xi,\gamma\hat\sigma_\xi)$ satisfying the following:
\begin{align*}
&\lbrace f(\widetilde Q_\xi)>\eta\rbrace \cap B(a,\mu\hat\sigma_\xi)\setminus B(a,\hat\sigma_\xi) \subset\bigcup_{j=1}^{\hat J} B(y_j^\xi,\gamma\hat\sigma_\xi),\qquad\hat J\leq M_0,\\
&\abs{y_i^\xi-y_j^\xi}>8\gamma\hat\sigma_\xi\qquad\forall i\neq j,\\
& y_j^\xi\in\lbrace z=0\rbrace\quad\text{or}\quad\dist(y_j^\xi,\lbrace z=0\rbrace)>8\gamma\hat\sigma_\xi.
\end{align*}
Since all balls $B(y_j^\xi,2\gamma\hat\sigma_\xi)$ are contained in the annulus
$B(a,2\mu\hat\sigma_\xi)\setminus B(a,\hat\sigma_\xi/2)$ 
and for small enough $\xi$ by the above $U_\xi$ is non-orientable on $\partial B(a,2\mu\hat\sigma_\xi)$ 
and orientable on $\partial B(a,\hat\sigma_\xi/2)$, 
we deduce that $U_\xi$ must be non-orientable on $\partial B(y_j^\xi,2\gamma\hat\sigma_\xi)$ 
for at least one $y_j^\xi$. Relabeling we assume this is $y_1^\xi$.

We may apply Lemma~\ref{l:lowerboundperforated} to obtain a lower bound on the energy in $B(y_1^\xi,2\gamma\hat\sigma_\xi)$. This ball may happen to intersect $\partial\widetilde\Omega$, but considering the fixed Lipschitz extension $U_0$ os $\widetilde Q_\xi$ outside $\widetilde\Omega$ we still have the same lower bound. 

Let us first assume that this $y_1^\xi$ does not lie on the axis $\lbrace z=0\rbrace$. Then, taking into accound that $\rho\geq 1-\hat\sigma_\xi^2$ we obtain
\begin{align*}
\widetilde E_\xi(\widetilde Q_\xi;B(y_1^\xi,2\gamma\hat\sigma_\xi)&\geq (1-\hat\sigma_\xi^2)\pi\ln\frac{\hat\sigma_\xi}{\xi} -C.
\end{align*}
By symmetry, if $y_1^\xi$ lies above the axis $\lbrace z=0\rbrace$ then the same lower bound will be obtained below the axis, so we deduce
\begin{align*}
\widetilde E_\xi(\widetilde Q_\xi;\widetilde\Omega\cap B(a,2\mu\hat\sigma_\xi)\setminus B(a,\hat\sigma_\xi/2))&
\geq 2(1-\hat\sigma_\xi^2)\pi\ln\frac{\hat\sigma_\xi}{\xi}-O(1).
\end{align*}
Adding this to the lower bound \eqref{eq:lowerboundlargeannulus} on $B(a,\delta)\setminus B(a,2\mu\hat\sigma_\xi)$, we obtain
\begin{align*}
\widetilde E_\xi(\widetilde Q_\xi)&\geq 2\pi \ln\frac\delta{\hat\sigma_\xi} + 2\pi (1-\hat\sigma_\xi^2)\ln\frac{\hat\sigma_\xi}{\xi} - O(1) \\
&\geq 2\pi (1-\hat\sigma_\xi^2)\ln\frac 1\xi  - O(1).
\end{align*}
Since $\hat\sigma_\xi\to 0$ this contradicts the upper bound from Section~\ref{s:upper}. 

Thus the point $y_1^\xi$ must lie on the axis $\lbrace z=0\rbrace$, hence $\rho\geq 1+\hat\sigma_\xi/2$ on $B(y_1^\xi,2\gamma\hat\sigma_\xi)$, and applying Lemma~\ref{l:lowerboundperforated} we have
\begin{align*}
\widetilde E_\xi(\widetilde Q_\xi;B(y_1^\xi,2\gamma\hat\sigma_\xi))&\geq (1+\frac 12\hat\sigma_\xi)\pi\ln\frac{\hat\sigma_\xi}{\xi} -C.
\end{align*}
Together with the lower bound \eqref{eq:lowerboundlargeannulus} on $B(a,\delta)\setminus B(a,2\mu\hat\sigma_\xi)$, this implies
\begin{align*}
\widetilde E_\xi(\widetilde Q_\xi;\widetilde\Omega\cap B(a,\delta))&\geq 2\pi\ln\frac{\delta}{\hat\sigma_\xi} +\pi(1+\frac 12\hat\sigma_\xi)\ln\frac{\hat\sigma_\xi}{\xi} -C \\
&\geq  \pi\ln\frac 1\xi +\pi\ln\frac{1}{\hat\sigma_\xi} + \frac\pi 2 \hat\sigma_\xi\ln\frac 1\xi -2\pi\ln\frac 1\delta -C.
\end{align*}
Now notice that
\begin{align*}
\inf_{0<\sigma<1}\left\lbrace \pi\ln\frac 1\sigma + \frac\pi 2\sigma\ln\frac 1\xi\right\rbrace
\end{align*}
is attained at $\sigma=2/\ln(\frac 1\xi)$, so the above lower bound implies
\begin{equation*}
\widetilde E_\xi(\widetilde Q_\xi;\widetilde\Omega\cap B(a,\delta))
\geq  \pi\ln\frac 1\xi +\pi\ln\ln\frac 1\xi  -2\pi\ln\frac 1\delta -C.
\end{equation*}
This is the lower bound we have been after. Now recall we have been arguing on an arbitrary sequence $\xi=\xi_n\to 0$ and taking subsequences, so what this proves is that for any $\delta\in (0,\delta_0)$,
\begin{equation*}
\liminf_{\e\to 0}\left(\widetilde E(\widetilde Q_\xi;\widetilde\Omega\cap B(a,\delta))-\pi\ln\frac 1\xi -\pi\ln\ln\frac 1\xi\right) \geq -2\pi\ln\frac 1\delta - C,
\end{equation*}
and this is part $(ii)$ of Theorem~\ref{thm:main}.
\end{proof}

\begin{remark}\label{r:lowerbound}
In the above proof, we obtain lower bounds on $B(a,\delta)\setminus B(a,2\mu\hat\sigma_\xi)$ and on $B(y_1^\xi,2\gamma\hat\sigma_\xi)$ and realize that their sum corresponds to the upper bound obtained in Section~\ref{s:upper}. Therefore in those sets the lower bounds must be sharp, in the sense that matching upper bounds are valid. In particular we have
\begin{equation*}
\widetilde E_\xi(\widetilde Q_\xi;\widetilde\Omega\cap B(a,\delta)\setminus B(a,2\mu\hat\sigma_\xi))\leq 2\pi\ln\frac{\delta}{\hat\sigma_\xi}+C,
\end{equation*}
and find ourselves in the situation of Lemma~\ref{l:finelowerboundhalfannulus}. As a consequence, a lower bound similar to the one in part $(ii)$ of Theorem~\ref{thm:main} is valid on the slightly smaller set
\begin{equation*}
D^{int}_\delta\cap\lbrace \rho\geq 1+\lambda z^{\beta}\rbrace,
\end{equation*}
for any $\lambda>0$ and $\beta>1$.
Combining this with the upper bound obtained in Section~\ref{s:upper} then yields
\begin{equation*}
\widetilde E_\xi(\widetilde Q_\xi;D^{ext}_\delta\cup\lbrace \rho\leq 1+\lambda z^\beta\rbrace)\leq 
2\pi\ln\frac 1\delta + C(\beta,\lambda)\qquad\forall \delta\in (0,1),\,\lambda>0,\,\beta>1.
\end{equation*}
\end{remark}

\section{Limit configuration}\label{s:lim}

In this section we prove part $(iii)$ of Theorem~\ref{thm:main}.  Although in the previous sections we have already established tight upper and lower bounds on the energy away from the equator defect, extracting precise information about the structure of the minimizers and the absence of point singularities will be a multi-step process:

\begin{itemize}
\item
In \S~\ref{ss:limconv} we use standard arguments to establish strong $H^1$ convergence away from the ring defect.
\item
In \S~\ref{ss:limsym} we then establish the additional symmetry property that the director points inside the azimuthal plane (in the sense of Remark~\ref{r:limit_symmetry}). 
The argument of \cite{sandier_shafrir_93} implies that as long as  this symmetry is satisfied at the boundary, it should also be satisfied inside the domain. 
Applying this result in our context is however not straightforward: the limiting map is only minimizing away from the ring defect and it is necessary to cut out a small disc around the defect. 
The boundary conditions are not fixed there and they need to be carefully estimated using the rigidity imposed by the energy asymptotics (in the  spirit of Lemma~\ref{l:finelowerboundhalfannulus}). 
\item
Finally, in \S~\ref{ss:limsmooth}, we use an argument of \cite{colloid} to rule out point defects which are not topologically necessary.
That argument requires even more precise estimates on the boundary conditions. We can only rule out point defects provided that the ring defect is negatively charged---otherwise a positively charged ring defect would have to be compensated by a pair of point defects.  
The possibility of a positively charged ring defect is eventually eliminated by establishing that in a small region around the ring, it could not---compared to a negatively charged ring---improve the energy by more than $o(1)$ as $\xi\to 0$. 
Effectively, this amounts to showing that the core energy of a positively or negatively charged ring defect are the same.
\end{itemize}

\subsection{Strong convergence}\label{ss:limconv}
We start by establishing strong $H^1$ convergence along a subsequence.

\begin{lemma}\label{l:H1conv}
There is a subsequence $\xi\to 0$ and $Q_\star$ such that for all $\delta>0$, the sequence $Q_\xi$ converges strongly in $H^1_{loc}(\Omega^{ext}_\delta)$ to $Q_\star\in\mathcal H^\star_{sym}$ and
\begin{align*}
E(Q_\xi;\Omega^{ext}_\delta)\longrightarrow E(Q_\star;\Omega^{ext}_\delta)\quad\text{and}\quad
\frac{1}{\xi^2}\int_{\Omega^{ext}_\delta}f(Q_\xi)\longrightarrow 0,
\end{align*}
as $\xi\to 0$. Moreover, in any relatively open subset $U\subset\overline\Omega$ where $Q_\star$ is smooth, we have $Q_\xi\to Q_\star$ in $C^{1,\alpha}_{loc}(U)$ for all $\alpha\in (0,1)$.
\end{lemma}
\begin{proof}[Proof of Lemma~\ref{l:H1conv}]
For any $\delta>0$, the upper and lower bound imply
\begin{equation*}
\frac 1{2\pi} E(Q_\xi;\Omega_\delta^{ext})\leq 2\pi\ln\frac 1\delta +C.
\end{equation*}
By a diagonal argument, this bound implies the existence a limit map $Q_\star\in \mathcal H^\star_{sym}$ such that, up to a subsequence, $Q_\xi$ converges weakly to $Q_\star$ in $H^1_{loc}(\Omega^{ext}_\delta)$ for every $\delta>0$.
We denote by $T_\delta$ the part of $\partial\Omega^{ext}_\delta$ that lies inside $\Omega$, namely in cylindrical coordinates
\begin{equation*}
T_\delta=\Omega\cap \lbrace (\rho-1)^2+z^2=\delta^2\rbrace.
\end{equation*}
 By Fubini's theorem we may (upon extracting a further subsequence) fix $\delta$ arbitrarily small  such that
\begin{equation}\label{eq:QTdelta}
E_\xi(Q_\xi;T_\delta)+E_\star(Q_\star;T_\delta)\lesssim \frac 1\delta \left(E_\xi(Q_\xi;\Omega^{ext}_{\delta/2}) + E_\star(Q_\star;\Omega^{ext}_{\delta/2})\right)\lesssim C(\delta).
\end{equation} 
In particular  the trace of $Q_\xi$ on $T_\delta$ is bounded in $H^1(T_\delta)$, and converges therefore also weakly in $H^1(T_\delta)$.
The map $Q_\xi$ is of the form
\begin{equation*}
Q_\xi(\rho,\varphi,z)=R_\varphi\widetilde Q_\xi(\rho,z) R_\varphi^t,
\end{equation*}
and $\widetilde Q_\xi$ converges weakly in $H^1_{loc}(\widetilde \Omega^{ext}_\delta)$ and in $H^1(\widetilde T_\delta)$  to $\widetilde Q_\star$ such that
\begin{equation*}
Q_\star(\rho,\varphi,z)=R_\varphi\widetilde Q_\star(\rho,z) R_\varphi^t.
\end{equation*}
(Here $\widetilde \Omega^{ext}_\delta$ and $\widetilde T_\delta$ denote the two-dimensional sections of $\Omega_\delta^{ext}$ and $T_\delta$ corresponding to the $(\rho,z)$ coordinates.)
Consider a map $Q_0\in\mathcal H_{sym}^\star(\Omega^{ext}_\delta)$ satisfying 
\begin{equation*}
E_\star(Q_0;\Omega^{ext}_\delta)=\min\left\lbrace E_\star(Q;\Omega^{ext}_\delta)\colon Q\in\mathcal H_{sym}^\star(\Omega^{ext}_\delta),\,Q=Q_\star\text{ on }T_\delta\right\rbrace.
\end{equation*}
Since $\widetilde Q_\xi$ converges weakly in $H^1(\widetilde T_\delta)$ to $ (\widetilde Q_\star)_{\lfloor \widetilde T_\delta}=(\widetilde Q_0)_{\lfloor \widetilde T_\delta}$ and $\widetilde T_\delta$ is a one-dimensional curve, on this curve  $\widetilde Q_\xi$ converges to  $\widetilde  Q_0$ in $C^{0,\alpha}(T_\delta)$ for any $\alpha\in (0,1/2)$. We claim that this allows to construct a map $\overline Q_\xi\in \mathcal H_{sym}$ such that
\begin{equation}\label{eq:overlineQ}
\overline Q_\xi=Q_\xi\quad\text{in }\Omega^{int}_\delta=\Omega\setminus\Omega_\delta^{ext}\qquad\text{ and } E_\xi(\overline Q_\xi;\Omega^{ext}_\delta)-E_\star(Q_0;\Omega^{ext}_\delta)\to 0\quad\text{as }\xi\to 0.
\end{equation}  
In order to define $\overline Q_\xi$ in $\Omega^{ext}_\delta$ we introduce the cross-sectional domains
\begin{equation*}
D=\left\lbrace (\rho,z)\colon \rho,z>0,\,\rho^2+z^2>1\right\rbrace,
\end{equation*}
and
\begin{equation}\label{eq:Dext}
\begin{aligned}  D^{ext}_\delta&=\left\lbrace (\rho,z) \colon \ (\rho-1)^2+z^2>\delta^2, \ z>0 \right\rbrace,  \\
D^{int}_\delta 
   &= \left\{(\rho,z) \colon \ (\rho-1)^2+z^2\le \delta^2, \ z>0 \right\rbrace
      = D\setminus D^{ext}_\delta .
\end{aligned}
\end{equation}
We use polar coordinates $(r,\theta)$ centered at $(\rho,z)=(1,0)$. In these coordinates the domains $D$ and $D^{ext}_\delta$ are given by 
\begin{gather}\nonumber
D=\{ (r,\theta): \ 0<\theta<\theta_0(r), \ r>0\}, \quad
D^{ext}_\delta = \{ (r,\theta): \ 0<\theta<\theta_0(r), \ r>\delta\}\\
  \label{eq:Ddelta}
   \text{where} \ \theta_0(r)=\begin{cases}
    \frac\pi 2 +\arcsin\frac r2, &\text{if $r\le\sqrt{2}$,}\\
           \frac\pi 2 +\arcsin\frac{1}{r}, &\text{if $r>\sqrt{2}$.} \end{cases}
\end{gather}
Then, the desired map $\overline Q_\xi$ will have the form
\begin{align*}
\overline Q_\xi(\rho,\varphi,z)&=R_\varphi \widehat Q_\xi(\rho,z)R_\varphi^t,
\end{align*}
and it suffices to define $\widehat Q_\xi$ in the domain $D^{ext}_\delta$.

 The standard idea is to introduce a thin slice $\lbrace \delta < r<(1+\lambda)\delta\rbrace$  where we interpolate from $\widehat Q_\xi =\widetilde Q_\xi$ at $r=\delta$ to $\widehat Q_\xi=\widetilde Q_0$ at $r=(1+\lambda)\delta$, and choose $\lambda$ in order that the energy of $\widehat Q_\xi$ in that thin slice be negligible. Then we extend by $\widetilde Q_0$ (slightly rescaled) in $D_{(1+\lambda)\delta}$. Here, however, we need to be more careful because we have to preserve the boundary condition at $\rho^2+z^2=1$, i.e., in polar coordinates at $\theta=\theta_0(r)=\frac\pi 2+\arcsin\frac r2$. Hence we interpolate instead in a deformed slice $S_\lambda$ of the form
\begin{equation*}
S_\lambda = \lbrace \delta < r < (1+\tilde\lambda(\theta))\delta,\, 0<\theta<\theta_0(\delta)\rbrace,\qquad
\tilde\lambda(\theta)=(\theta_0(\delta)-\theta)\lambda,
\end{equation*}
for some $\lambda>0$ to be chosen later. 
Next we define $\widehat Q_\xi$ in the deformed slice $S_\lambda$. As noted in \cite[Lemma~B.2]{anisotrop},  with a simple linear interpolation we might fail at controlling the potential part of the energy $\int f(\widehat Q_\xi)$, and we need to proceed in two steps as in \cite{anisotrop}, namely first interpolate linearly between $\widetilde Q_\xi$ and its projection $\pi(\widetilde Q_\xi)$ onto $\mathcal U_\star$, and then interpolate geodesically between $\pi(\widetilde Q_\xi)$ and $\widetilde Q_0$ inside $\mathcal U_\star$. To this end we denote by 
\begin{equation*}
\gamma\colon [0,1]\times V\to\mathcal U_\star,\qquad V\text{ is a neighborhood of }\lbrace (U,U)\colon U\in\mathcal U_\star\rbrace\text{ in }\mathcal U_\star\times\mathcal U_\star,
\end{equation*}
the smooth map such that $t\mapsto \gamma(t;U_1,U_2)$ is the constant speed geodesic from $U_1$ to $U_2$ in $\mathcal U_\star$, which is indeed unique, well-defined and depends smoothly on $U_1,U_2$ provided $U_1,U_2$ are close enough to each other. This map satisfies the bounds
\begin{equation*}
\abs{\partial_t\gamma(t;U_1,U_2)}\lesssim\abs{U_1-U_2},\qquad \abs{\partial_U\gamma}\lesssim 1.
\end{equation*}
In $S_\lambda$ we define
\begin{align*}
&\widehat Q(r,\theta)=\widetilde Q_\xi(\delta,\theta)+\mu_1(r,\theta)\left(\pi(\widetilde Q_\xi(\delta,\theta))-\widetilde Q_\xi(\delta,\theta)\right) \quad\text{for }\delta <r <\delta + \frac{\tilde\lambda(\theta)}{2}\delta,\\
&\text{where }\mu_1(r,\theta) =\frac{2}{\tilde\lambda(\theta)\delta}(r-\delta),\\
&\widehat Q(r,\theta)=
\gamma\left(\mu_2(r,\theta);\pi(\widetilde Q_\xi(\delta,\theta),\widetilde Q_0(\delta,\theta)\right)\quad\text{for }\delta + \frac{\tilde\lambda(\theta)}{2}\delta<r <\delta +\tilde\lambda(\theta)\delta,\\
&\text{where }\mu_2(r,\theta)= \frac{2}{\tilde\lambda(\theta)\delta}(r-\delta-\frac{\tilde\lambda(\theta)}{2}\delta).
\end{align*}
Note that $\pi(\widetilde Q_\xi(\delta,\theta))$ is well-defined for small $\xi$ because $\widetilde Q_\xi$ converges uniformly to $\widetilde Q_0$ on $\widetilde T_\delta$.
Direct computations then show that
\begin{align*}
\abs{\partial_r \widehat Q}&\lesssim \frac{1}{\delta\lambda}\frac{1}{\theta_0(\delta)-\theta}\left(\abs{\pi(\widetilde Q_\xi(\delta,\theta))-\widetilde Q_\xi(\delta,\theta)}
+ \abs{\pi(\widetilde Q_\xi(\delta,\theta))-\widetilde Q_0(\delta,\theta)}\right)
\\
&\lesssim \frac{1}{\delta\lambda}\frac{1}{\theta_0(\delta)-\theta}\abs{\widetilde Q_\xi(\delta,\theta)-\widetilde Q_0(\delta,\theta)}\\
\abs{\partial_\theta\widehat Q}&\lesssim\abs{\partial_\theta\widetilde Q(\delta,\theta)}+\abs{\partial_\theta\widetilde Q_0(\delta,\theta)}\\
&\quad +\frac{1}{\theta_0(\delta)-\theta}\left(\abs{\pi(\widetilde Q_\xi(\delta,\theta))-\widetilde Q_\xi(\delta,\theta)}
+ \abs{\pi(\widetilde Q_\xi(\delta,\theta))-\widetilde Q_0(\delta,\theta)}\right),\\
&\lesssim 
\abs{\partial_\theta\widetilde Q(\delta,\theta)}+\abs{\partial_\theta\widetilde Q_0(\delta,\theta)}
+\frac{1}{\theta_0(\delta)-\theta}\abs{\widetilde Q_\xi(\delta,\theta)-\widetilde Q_0(\delta,\theta)}
\end{align*}
Next, we denote by $\sigma=\sigma(\xi)$ the quantity
\begin{equation*}
\sigma=\norm{\widetilde Q_\xi-\widetilde Q_0}_{C^{1/4}(\widetilde T_\delta)}\longrightarrow 0\qquad\text{as }\xi\to 0,
\end{equation*}
and notice that since the boundary conditions on $\rho^2+z^2=1$ ensure that $\widetilde Q_\xi(\delta,\theta_0(r))=\widetilde Q_0(\delta,\theta_0(r))$, we have
\begin{equation*}
\abs{\widetilde Q_\xi(\delta,\theta)-\widetilde Q_0(\delta,\theta)}\leq \sigma \cdot(\theta_0(\delta)-\theta)^{\frac 14}.
\end{equation*}
Therefore, the above estimates on the derivatives of $\widehat Q$ imply
\begin{align*}
\abs{\nabla \widehat Q}^2&=\abs{\partial_r\widehat Q}^2+\frac{1}{r^2}\abs{\partial_\theta \widehat Q}^2\\
&\lesssim \frac{1}{\delta^2\lambda^2}\frac{1}{(\theta_0(\delta)-\theta)^{\frac 32}}\sigma^2 +\frac{1}{\delta^2}\left(\abs{\partial_\theta\widetilde Q(\delta,\theta)}^2+\abs{\partial_\theta\widetilde Q_0(\delta,\theta)}^2 \right),
\end{align*}
and
\begin{align*}
\int_{S_\lambda} \abs{\nabla \widehat Q}^2 &= \int_0^{\theta_0(\delta)}\int_{\delta}^{\delta+\delta\lambda(\theta_0(\delta)-\theta)}rdr \abs{\nabla \widehat Q}^2 d\theta \\
&\lesssim  \delta^2 \lambda \int_0^{\theta_0(\delta)}(\theta_0(\delta)-\theta)\\
&\hspace{3em}\cdot\left[
\frac{1}{\delta^2\lambda^2}\frac{1}{(\theta_0(\delta)-\theta)^{\frac 32}}\sigma^2 +\frac{1}{\delta^2}\left(\abs{\partial_\theta\widetilde Q(\delta,\theta)}^2+\abs{\partial_\theta\widetilde Q_0(\delta,\theta)}^2 \right)
\right]d\theta\\
&\lesssim \frac{\sigma^2}{\lambda} +\lambda C(\delta).
\end{align*}
For the last inequality we used the fact that, thanks to \eqref{eq:QTdelta}, the $L^2$ norms of $\partial_\theta \widetilde Q_\xi(\delta,\theta)$ and $\partial_\theta \widetilde Q_0(\delta,\theta)$ are bounded for $\delta$ fixed. Note moreover that the definition of $\widehat Q$ ensures that 
\begin{equation*}
f(\widehat Q)\lesssim \dist^2(\widehat Q,\mathcal U_\star)\leq\dist^2(\widetilde Q_\xi(\delta,\theta),\mathcal U_\star)\lesssim f(\widetilde Q_\xi(\delta,\theta)),
\end{equation*}
thanks to the nonegeneracy property \eqref{eq:fdistU*} of the potential $f$. Using this pointwise inequality and \eqref{eq:QTdelta} we infer that
\begin{equation*}
\frac{1}{\xi^2}\int_{S_\lambda}f(\widehat Q_\xi)\lesssim \lambda C(\delta).
\end{equation*}
Finally, since $\rho\approx 1$ and $\Xi[\widehat Q]\lesssim 1$ in $S_\lambda$ we deduce from the above that
\begin{equation*}
\widetilde E_\xi(\widehat Q_\xi;S_\lambda)\lesssim \frac{\sigma^2}{\lambda} +\lambda C(\delta).
\end{equation*}
Choosing $\lambda=\sigma$ then implies that
\begin{equation*}
\widetilde E_\xi(\widehat Q_\xi;S_\lambda)\lesssim  C(\delta)\sigma\longrightarrow 0\quad\text{as }\xi\to 0.
\end{equation*}
Next we define $\widehat Q_\xi$ in $D^{ext}_\delta\setminus S_\lambda$ by setting
\begin{align*}
\widehat Q_\xi(r,\theta)&=\widetilde Q_0\left(\delta+\frac{1}{1-\tilde\lambda(\theta)}(r-\delta-\tilde\lambda(\theta)\delta),\theta\right)\qquad \text{for }\delta+\tilde\lambda(\theta)\delta < r < 2\delta,\\
\widehat Q_\xi &=\widetilde Q_0\qquad\text{in }D_{2\delta}.
\end{align*}
Then we have
\begin{align*}
\abs{\widetilde E_\xi(\widehat Q_\xi;D^{ext}_\delta\setminus S_\lambda)-\widetilde E_\star(\widetilde Q_0;D^{ext}_\delta)}&=\abs{\widetilde E_\star(\widehat Q_\xi;D^{ext}_\delta\setminus S_\lambda)-\widetilde E_\star(\widetilde Q_0;D^{ext}_\delta)}\\
&\lesssim \lambda \widetilde E_\star(\widetilde Q_0;D^{ext}_\delta)\longrightarrow 0,
\end{align*}
as $\xi\to 0$, since $\lambda=\sigma\to 0$. Combining this and the fact that $\widetilde E_\xi(\widehat Q_\xi;S_\lambda)\to 0$, we deduce that
\begin{equation*}
\widetilde E_\xi(\widehat Q_\xi;D^{ext}_\delta)-\widetilde E_\star(\widetilde Q_0;D^{ext}_\delta)\longrightarrow 0,
\end{equation*}
which implies the desired estimate \eqref{eq:overlineQ}.
By minimality of $Q_\xi$ we then have
\begin{equation*}
E_\xi(Q_\xi;\Omega^{ext}_\delta)\leq E_\xi(\overline Q_\xi;\Omega^{ext}_\delta)\leq E_\star(Q_0;\Omega^{ext}_\delta)+o(1),
\end{equation*}
and by lower semicontinuity we deduce
\begin{align*}
E_\star(Q_\star;\Omega^{ext}_\delta)&\leq\liminf E_\xi(Q_\xi;\Omega^{ext}_\delta) \leq \limsup E_\xi(\overline Q_\xi;\Omega^{ext}_\delta)\leq E_\star( Q_0;\Omega^{ext}_\delta).
\end{align*}
It follows that $Q_\star$ minimizes $E_\star(\cdot;\Omega^{ext}_\delta)$ among all maps $Q\in\mathcal H^\star_{sym}(\Omega^{ext}_\delta)$ such that $Q=Q_\star$ on $T_\delta$, and that all above inequalities are in fact equalities. Thus we have
\begin{equation*}
\int_{\Omega^{ext}_\delta}\abs{\nabla Q_\xi}^2\longrightarrow \int_{\Omega^{ext}_\delta}\abs{\nabla Q_\star}^2\quad\text{ and }\quad \frac{1}{\xi^2}\int_{\Omega^{ext}_\delta}f(Q_\xi)\longrightarrow 0.
\end{equation*}
In particular, together with the weak convergence, this implies that $\nabla Q_\xi$ converges strongly in $L^2(\Omega^{ext}_\delta)$ towards $\nabla Q_\star$, that is, the convergence is in fact strong.

The local $C^{1,\alpha}$ convergence away from singularities follows from the analysis in \cite{majumdarzarnescu10,nguyen_zarnescu_13}. There the authors consider minimizers which are not subject to the constraint of axial symmetry, but they only use the minimizing property to obtain the strong $H^1$ convergence, which we have just obtained, so their results apply directly.
\end{proof}

\subsection{Reduction to a director in a cross-section}\label{ss:limsym}

The limiting $Q$-tensor $Q_*$ inherits the symmetries \eqref{eq:Hsym} of the space $\mathcal{H}_{sym}$, but it also exhibits further symmetry by virtue of energy minimization.  Here we show that the map $Q_\star$ may be represented by a uniaxial tensor with a unit director field $n=(n_1(\rho,z),0,n_3(\rho,z))$, expressed in cylindrical coordinates in a cross-section of $\Omega$.

\medskip

In the sequel, whenever we refer to $\xi\to 0$, we will always mean convergence along the subsequence obtained in Lemma~\ref{l:H1conv}.

Since the limit $Q_\star$ is symmetric, it is characterized by a map defined in the two-dimensional domain $D$.
To describe $Q_\star$ further, it will also be convenient to introduce the following notations:
\begin{align*}
\hat E(n;D^{ext}_\delta)& =\int_{D^{ext}_\delta} \left(\abs{\nabla n}^2+\frac{n_1^2+n_2^2}{\rho^2} \right)\rho\, d\rho\, dz\quad\text{for }n\in H^1_{loc}(D^{ext}_\delta;\mathbb S^2),\\
\hat{\mathcal H}(D^{ext}_\delta) &=\Big\lbrace n\in H^1_{loc}(D^{ext}_\delta;\mathbb S^2)\colon \hat E(n;D^{ext}_\delta)<\infty,\\
&\hspace{10em} n\otimes n =e_r\otimes e_r \text{ on }\partial D^{ext}_\delta\cap\lbrace\rho^2+z^2=1\rbrace,\\
&\hspace{10em} n\otimes n =e_3\otimes e_3\text{ on }\partial D^{ext}_\delta\cap\lbrace z=0\rbrace \Big\rbrace.\\
\hat {\mathcal H} &=\bigcap_{\delta>0} \hat{\mathcal H}(D^{ext}_\delta),
\end{align*}
where $D^{ext}_\delta$ is defined in \eqref{eq:Dext}.
The symmetry and minimizing property of $Q_\star$ allow us to express it in terms of a map $n$ which minimizes $\hat E$ in $\hat{\mathcal H}$. Specifically, we have:

\begin{lemma}\label{l:n}
The map $Q_\star$ is given in cylindrical coordinates by
\begin{equation*}
Q_\star(\rho,\varphi,z) = R_\varphi n(\rho,z) \otimes R_\varphi n(\rho,z)-\frac 13 I,
\end{equation*}
where $n\in H^1_{loc}(\widetilde\Omega)$ satisfies $n(\rho,-z)=-Sn(\rho,z)$ and, when restricted to $\lbrace z>0\rbrace$, $n\in\hat{\mathcal H}$ minimizes $\hat E(\cdot;D^{ext}_\delta)$ for all $\delta>0$, among all maps $m\in\hat{\mathcal H}(D^{ext}_\delta)$ such that $m\otimes m=n\otimes n$ on $\partial D^{ext}_\delta \cap D$.

Moreover, up to replacing $n$ by $-n$, we have
\begin{equation}\label{eq:noriented}
n=e_r\text{ on }\partial D^{ext}_\delta\cap\lbrace \rho^2+z^2=1\rbrace\quad\text{and}\quad n=\tau e_3\text{ on }\partial D^{ext}_\delta\cap\lbrace z=0\rbrace,
\end{equation}
for some $\tau\in\lbrace \pm 1\rbrace$.
\end{lemma}

\begin{proof}[Proof of Lemma~\ref{l:n}]
Since $\Omega^{ext}_\delta$ is simply connected, $Q_\star$ can be lifted \cite{bethuelchiron07,ballzarnescu11}: there exists a map $n_\star\in H^1_{loc}(\Omega;\mathbb S^2)$ such that
\begin{equation*}
Q_\star =n_\star\otimes n_\star-\frac 13 I.
\end{equation*}
As $\abs{\nabla Q_\star}^2=2\abs{\nabla n_\star}^2$, the symmetry of $Q_\star$ implies that $n(\rho,z)= n_\star(\rho,0,z)$ belongs to $H^1_{loc}(\widetilde\Omega^{ext}_\delta)$, and we have
\begin{align*}
\frac 1{8\pi}E(Q_\star;\Omega^{ext}_\delta)& = \hat E(n;D^{ext}_\delta).
\end{align*}
Since $Q_\star$ is minimizing in $\mathcal H_{sym}^\star(\Omega^{ext}_\delta)$, we deduce that
 $n$ minimizes $\hat E(\cdot;D^{ext}_\delta)$ among maps $m\colon D^{ext}_\delta\to\mathbb S^2$ satisfying the boundary conditions
\begin{align*}
&m\otimes m = e_r\otimes e_r\quad\text{ on }\partial D^{ext}_\delta\cap\lbrace\rho^2+z^2 =1\rbrace,\\
&m\otimes m = n\otimes n\quad\text{ on }\partial D^{ext}_\delta\cap D,\\
&m\otimes m = (Sm)\otimes (Sm)\quad\text{ on }\partial D^{ext}_\delta\cap\lbrace z=0\rbrace.
\end{align*}
Note that $\abs{m\otimes m-e_3\otimes e_3}^2=2(m_1^2+m_2^2)$, so that the far-field condition which requires $\int\abs{Q-Q_\infty}^2\abs{x}^{-2}<\infty$ is obsolete here: any map $m$ with finite energy satisfies
\begin{equation*}
\int_{D^{ext}_\delta}\frac{m_1^2+m_2^2}{\rho^2+z^2}\rho\, d\rho\, dz\leq \hat E(m;D^{ext}_\delta) <\infty.
\end{equation*}
The boundary condition on $\lbrace z=0\rbrace$ comes from the requirement that $m\otimes m$ can be extended to an $H^1$ map in $\widetilde \Omega^{ext}_\delta$ via the mirror symmetry. It is equivalent to $m_3(\rho,0)\in\lbrace 0,\pm 1\rbrace$ for almost all $\rho>1+\delta$. Since the trace $\rho\mapsto m_3(\rho,0)$ has $H^{1/2}_{loc}$ regularity, being integer valued it has to be constant: there exists $\tau\in\lbrace 0,\pm 1\rbrace$ such that $m_3(\rho,0)=\tau$ for almost all $\rho>1+\delta$. One can rule out $\tau=0$. Indeed, assume by contradiction that $\tau=0$, i.e. $m_3(\rho,0)=0$. Then we have
\begin{align*}
m_1^2+m_2^2 &= 1-m_3^2 = 1 -\left(\int_0^z\partial_z m_3\right)^2 \\
&\geq 1-\abs{z}\int_0^z\abs{\partial_z m}^2,
\end{align*}
and therefore, for almost all $z>0$,
\begin{align*}
+\infty =\int_2^\infty \frac{d\rho}{\rho} & \leq \int_2^\infty \frac{m_1^2+m_2^2}{\rho}\,d\rho
 + \int_2^\infty \frac{\abs{z}}{\rho^2}\int_0^\infty \abs{\partial_z m}^2\,dz\,\rho\,d\rho\\
&\leq \int_2^\infty \frac{m_1^2+m_2^2}{\rho}d\rho + \abs{z}\int_{D_1}\abs{\nabla m}^2\rho\,d\rho\,dz.
\end{align*}
This clearly contradicts the finiteness of $\int(m_1^2+m_2^2)/\rho\, d\rho dz$. We deduce that $\tau =\pm 1$, that is,
\begin{equation*}
m\otimes m =e_3\otimes e_3\quad\text{for }z=0,
\end{equation*}
so that $m\in \hat{\mathcal H}(D^{ext}_\delta)$ and, as a consequence, $n$ minimizes $\hat E(\cdot;D^{ext}_\delta)$ in $\hat{\mathcal H}(D^{ext}_\delta)$ for all $\delta>0$ with respect to its own boundary conditions on $\partial D^{ext}_\delta\cap D$.

Moreover, the boundary conditions on $\partial D^{ext}_\delta\cap\lbrace \rho^2+z^2=1\rbrace$ require that the $H^{1/2}$ trace $n\cdot e_r$ take values into $\lbrace \pm 1\rbrace$ and thus be constant. Then, up to changing $n$ to $-n$ and $\tau$ to $-\tau$, one obtains the boundary conditions \eqref{eq:noriented}.

It remains to show that $n(\rho,-z)=-Sn(\rho,z)$. The mirror symmetry implies that $n(\rho,-z)=\pm S n(\rho,z)$, and therefore the $H^1$ function $n(\rho,-z)\cdot Sn(\rho,z)$ takes values into $\lbrace \pm 1\rbrace$ and must be constant. By the above, its trace on $\lbrace z=0\rbrace$ is equal to $\tau^2 e_3 \cdot (-e_3)=-1$. We conclude that $n(\rho,-z)=-Sn(\rho,z)$.
\end{proof}

Next we turn to proving the additional symmetry property $n_2\equiv 0$. The idea is that the use of an appropriate comparison map as in \cite{sandier_shafrir_93} implies that symmetry provided it is satisfied at the boundary. But we can only use energy comparison in $D^{ext}_\delta$, and then we actually do not know that $n_2=0$ on the whole boundary, due to the undetermined part $\partial D^{ext}_\delta \cap D$. This will make the proof quite technical. 

To gather more information about the behavior of $n$ on $\partial D^{ext}_\delta\cap D$, it is natural to use  polar coordinates $(r,\theta)$ around $(\rho=1,z=0)$, so that $\partial D^{ext}_\delta\cap D$ corresponds to fixing $r=\delta$. In those coordinates, the domain $D$ is given by $0<\theta<\theta_0(r)$, where $\theta_0(r)$ is defined in \eqref{eq:Ddelta}.
The upper and lower bound, together with strong $H^1$ convergence, provide the estimate
\begin{equation*}
\frac 1{2\pi} E(Q_\star;\Omega\setminus\mathcal N_\delta(\mathcal C))\leq 2\pi\ln\frac 1\delta +C,
\end{equation*}
which for $n(\rho,z)$ translates into
\begin{equation*}
\hat E(n;D^{ext}_\delta)\leq\frac \pi 2\ln\frac 1\delta + C.
\end{equation*}
In coordinates $(r,\theta)$ this implies 
\begin{equation}\label{eq:energybound_n}
\int_\delta^1  \left[\int_0^{\theta_0(r)} \abs{\partial_\theta n}^2d\theta -\frac \pi 2 \right]\frac{dr}{r} + \int_{\delta}^1 \int_0^{\theta_0(r)}\abs{\partial_r n}^2 d\theta\,r\, dr \leq C.
\end{equation}
For $r<\sqrt 2$, the boundary conditions \eqref{eq:noriented} become
\begin{equation}\label{eq:theta1}
  n(r,\theta_0(r))=\cos\theta_1e_1 + \sin\theta_1e_3, \quad\text{where} \quad \theta_1=2\theta_0-\pi,
  \end{equation}
and $n(r,0)=\tau e_3$. Remarking (as in Lemma~\ref{l:finelowerboundhalfannulus}) that we also have the lower bound
\begin{align*}
\int_0^{\theta_0}\abs{\partial_\theta n}^2 d\theta &
\geq\frac{1}{\theta_0}
\left(\int_0^{\theta_0}\abs{\partial_\theta n} d\theta\right)^2 \\
&\geq \frac{1}{\theta_0}
[\dist_{\mathbb S^2}(n(r,0),n(r,\theta_0))]^2 \\
&=\frac{(\tau\pi/2-\theta_1)^2}{\theta_0}= \frac\pi 2 + O(r),
\end{align*}
from  \eqref{eq:energybound_n} we deduce
\begin{equation}\label{eq:bound_dn}
\int_0^1
\left|{\int_0^{\theta_0(r)}\abs{\partial_\theta n}^2d\theta-\frac\pi 2}\right|
\frac{dr}{r}
+ \int_0^1\int_0^{\theta_0(r)}\abs{\partial_r n}^2 d\theta\, r \, dr <\infty.
\end{equation}
We are now ready to prove:

\begin{proposition}\label{p:addsym}
We have $n_2\equiv 0$ in $D$.
\end{proposition}
\begin{proof}[Proof of Proposition~\ref{p:addsym}]
The starting idea is to use as a comparison map
\begin{equation*}
\tilde n = \left(\sqrt{n_1^2+n_2^2},0,n_3\right),
\end{equation*}
which has lower energy than $n$ since $n_1^2+n_2^2=\tilde n_1^2+ \tilde n_2^2$ and
\begin{equation*}
\abs{\partial_j n}^2-\abs{\partial_j \tilde n}^2 = -\frac{(n_1\partial_j n_2-n_2\partial_j n_1)^2}{n_1^2 + n_2^2},
\end{equation*}
 and to conclude that $n_2=0$. One just needs to take care of technical difficulties that arise due to the undetermined part of the boundary $\partial D^{ext}_\delta\cap D$: there, one does not know that $n_2=0$ and $n_1\geq 0$, hence $\tilde n$ cannot be used directly as a comparison map. One needs to introduce a transition layer, and to ensure that this transition layer's excess energy is negligible as $\delta\to 0$. We will define a good comparison map $\bar n$ in $D^{ext}_\delta$ by setting
\begin{equation*}
\bar n = \tilde n\qquad\text{in }D^{ext}_{\sqrt{2\delta}},
\end{equation*}
and by constructing an appropriate transition layer in $D^{ext}_\delta\setminus D^{ext}_{\sqrt{2\delta}}$. 
 We first set
\begin{align*}
\bar n(r,\theta)&=
\left(\begin{array}{c}
\sqrt{n_1^2 + (1-\lambda^2)n_2^2}\\
\lambda n_2\\
n_3
\end{array}
\right),\qquad
\lambda =\lambda(r) = 1- \frac{2}{\ln(2\delta)}\ln\frac{2\delta}r,
\end{align*}
in the region $D^{ext}_{2\delta}\setminus D^{ext}_{\sqrt{2\delta}}$ corresponding  to $2\delta\leq r\leq \sqrt{2\delta}$. On the boundary of $D^{ext}_{2\delta}\setminus D^{ext}_{\sqrt{2\delta}}$ this map satisfies
\begin{align*}
\bar n &=\tilde n\quad\text{for }r=\sqrt{2\delta},&
\bar n & = (\abs{n_1},n_2,n_3)\quad\text{for }r=2\delta,\\
\bar n & = \tau e_3 = n \quad\text{for }\theta=0,&
\bar n & = n \quad\text{for }\theta=\theta_0(r).
\end{align*}
For $2\delta<r<\sqrt{2\delta}$ we have
\begin{align*}
\abs{\partial_\theta \bar n}^2-\abs{\partial_\theta  n}^2 & =-(1-\lambda^2)\frac{(n_1\partial_\theta n_2-n_2\partial_\theta n_1)^2}{n_1^2+(1-\lambda^2)n_2^2}\leq 0,\\
\abs{\partial_r \bar n}^2 &\leq 2(\lambda'(r))^2+2\abs{\partial_rn}^2
\leq \frac{4}{\ln^2(2\delta)}\frac{1}{r^2} + 2\abs{\partial_r n}^2,\\
\bar n_1^2 + \bar n_2^2 & = n_1^2+n_2^2,
\end{align*}
and therefore the excess energy satisfies
\begin{align*}
& \hat E(\bar n;D^{ext}_{2\delta}\setminus D^{ext}_{\sqrt{2\delta}})-\hat E(n;D^{ext}_{2\delta}\setminus D^{ext}_{\sqrt{2\delta}})\\
&\leq \frac{4\pi}{\ln^2(2\delta)}\int_{2\delta}^{\sqrt{2\delta}}\frac{dr}{r} + 2\int_{2\delta}^{\sqrt{2\delta}}\int_0^{\theta_0(r)}\abs{\partial_r n}^2d\theta\, r\, dr \\
&\leq \frac{2\pi}{\abs{\ln(2\delta)}} + 2\int_{r\leq\sqrt{2\delta}}\abs{\partial_r n}^2 d\theta\, r\, dr.
\end{align*}
Since $\abs{\partial_r n}^2$ is integrable \eqref{eq:bound_dn}, this quantity is negligible as $\delta\to 0$.

 It remains to define $\bar n$ in $D^{ext}_{\delta}\setminus D^{ext}_{2\delta}$, that is, to build a transition layer between $n$ at $r=\delta$ and $\bar n=(\abs{n_1},n_2,n_3)$ at $r=2\delta$.
To this end, recalling that $n_1(r,0)=0$ and that $\theta\mapsto n(r,\theta)$ is continuous for almost all $r>0$, we define
\begin{equation*}
\gamma(r)=\sup \left\lbrace \theta\in [0,\theta_0(r))\colon n_1(r,\theta)=0\right\rbrace,
\end{equation*}
so that $\bar n(r,\theta)=n(r,\theta)$ for $\gamma(r)\leq\theta\leq\theta_0(r)$. Since $n_1(r,\gamma(r))=0$, we have
\begin{align*}
\int_{\gamma(r)}^{\theta_0(r)}\abs{\partial_\theta n}^2 d\theta &\geq
\frac{1}{\theta_0(r)-\gamma(r)}[\dist_{\mathbb S^2}(n(r,\gamma),n(r,\theta_0)]^2\\
&= \frac{(\pi/2+O(r))^2}{\theta_0(r)-\gamma(r)}=
\frac\pi 2 + \frac{\gamma(r)+O(r)}{\theta_0(r)-\gamma(r)},
\end{align*}
and therefore,
\begin{align}\label{eq:etacontrolsgamma}
\eta(r):=\left|{\int_0^{\theta_0(r)}\abs{\partial_\theta n}^2d\theta -\frac \pi 2}\right|\geq  \frac{\gamma(r)-Cr}{\theta_0(r)-\gamma(r)} + \int_0^{\gamma(r)}\abs{\partial_\theta n}^2 d\theta.
\end{align}
Recall \eqref{eq:bound_dn} that $\eta$ is integrable on $(0,1)$ with respect to the measure $dr/r$, hence
\begin{align*}
\int_0^1\eta(r)\frac{dr}r & =\sum_{n=0}^\infty \int_{2^{-n-1}}^{2^{-n}}\eta(r)\frac{dr}r  = \sum_{n=0}^\infty \int_{1/2}^1 \eta(2^{-n}r)\frac{dr}{r} <\infty.
\end{align*}
The sequence $v_n=\int_{1/2}^1 \eta(2^{-n})dr/r$ thus being summable, there exists a subsequence $(n_k)$ such that
\begin{equation*}
\int_{1/2}^1 \eta(2^{-n_k}r)\frac{dr}{r}\leq \frac{1}{n_k}.
\end{equation*}
Moreover there exists $r_k\in (1/2,1)$ such that
\begin{equation*}
\eta(2^{-n_k}r_k)\leq C \int_{1/2}^1 \eta(2^{-n_k}r)\frac{dr}{r}\leq \frac{1}{n_k}.
\end{equation*}
We set $\delta=\delta_k=2^{-n_k-1}r_k$, so that $\delta_k\to 0$ and we have
\begin{equation}\label{eq:etadeltasmall}
\eta(2\delta)\leq C \left(\ln\frac 1\delta\right)^{-1}.
\end{equation}
Thanks to \eqref{eq:etacontrolsgamma} and abbreviating $\gamma=\gamma(2\delta)$, we have
\begin{equation}\label{eq:gammasmall}
\gamma+\int_0^{\gamma}\abs{\partial_\theta n}^2(2\delta,\theta)\,d\theta \leq C \left(\ln\frac 1\delta\right)^{-1}.
\end{equation}
It will also be useful to remark that
\begin{equation}\label{eq:oscsmall}
\max_{\theta\in[0,\gamma]}\abs{n(2\delta,\theta)-\tau e_3}^2\leq \gamma \int_0^{\gamma}\abs{\partial_\theta n}^2(2\delta,\theta)\,d\theta \leq C\left(\ln\frac 1\delta\right)^{-2}.
\end{equation}
Next we define $\bar n$  in $D^{ext}_{2\delta}\setminus D^{ext}_\delta$. In polar coordinates $(r,\theta)$, this corresponds to the region
\begin{align*}
R_\delta &= \left \lbrace (r,\theta)\colon \delta<r\leq 2\delta,\,0<\theta<\theta_0(r)\right\rbrace=R_1\sqcup R_2 \sqcup R_3,\\
R_1&=\lbrace \delta < r \leq \sigma(\theta)\rbrace,\\
\sigma(\theta)& = \begin{cases}
(2-\gamma^2)\delta &\quad\text{for }0<\theta\leq\gamma,\\
2\delta - \frac{\gamma^2\delta}{\theta_0(2\delta)-\gamma}(\theta_0(2\delta)-\theta)&\quad\text{for }\gamma<\theta<\theta_0(2\delta),
\end{cases}
\\
R_2&=\lbrace \sigma(\theta)<r\leq 2\delta,\,\gamma<\theta<\theta_0(2\delta)\rbrace,\\
R_3&=\lbrace \sigma(\theta)<r\leq 2\delta,\,0<\theta\leq\gamma(2\delta)\rbrace.
\end{align*}

\begin{center}
\begin{tikzpicture}[scale=1.2]

\draw[thick] (-.1,0) node [left] {$\delta$} -- (6,0);
\draw[thick] (0,-.1) node [below] {$0$} -- (0,3.1);
\draw[thick] (-.1,3) node [left] {$2\delta$} -- (6.5,3);
\draw[thick] (-.1,2) node [left] {$(2-\gamma^2)\delta$} -- (1,2) -- (6.5,3);
\draw[thick] (1,2) -- (1,3.1) node [above] {$\gamma$};
\draw[thick] (5.97,-.1) node [below] {$\theta_0(\delta)$} .. controls ++(.1,1) and ++(-.25,-1).. (6.53,3.1) node [above] {$\theta_0(2\delta)$};

\draw (3,1.5) node {$R_1$};
\draw (.5,2.5) node {$R_3$};
\draw (1.5,2.5) node {$R_2$};

\draw[thick,->] (-2,.5)--(-1.5,.5) node [below] {$\theta$};
\draw[thick,->] (-2,.5)--(-2,1) node [left] {$r$};
\end{tikzpicture}
\end{center}

In $D^{ext}_{2\delta}\setminus D^{ext}_\delta=R_\delta$ we thus define $\bar n$ by
\begin{align*}
\bar n(r,\theta)&=
\begin{cases}
n(\delta+\frac{\delta}{\sigma(\theta)-\delta}(r-\delta),\theta)&\quad\text{in }R_1,\\
n(2\delta,\theta)&\quad\text{in }R_2,\\
\pi_{\mathbb S^2}(n(2\delta,\theta)+\lambda(r)[\bar n(2\delta,\theta)-n(2\delta,\theta)])&\quad\text{in }R_3,
\end{cases}\\
\lambda(r)&=1-\left(\ln\frac 2{2-\delta^2}\right)^{-1}\ln\frac{2\delta}r.
\end{align*}
Here $\pi_{\mathbb S^2}$ is the nearest point projection onto $\mathbb S^2$, which is well-defined and uniformly Lipschitz because \eqref{eq:oscsmall} ensures
\begin{equation*}
\max_{[0,\gamma(2\delta)]}\abs{\bar n(2\delta,\cdot)-\tau e_3}= \max_{[0,\gamma(2\delta)]}\abs{n(2\delta,\cdot)-\tau e_3}\leq C\left(\ln\frac 1\delta\right)^{-2}.
\end{equation*}
At the boundary of $R_\delta$, this definition ensures that $\bar n = n$ on $\lbrace r=\delta\rbrace$, $\lbrace\theta=0\rbrace$ and $\lbrace \theta=\theta_0(r)\rbrace$, and that $\bar n$ does not jump across $\lbrace r=2\delta\rbrace$. Moreover we have
\begin{align*}
\hat E(\bar n; R_\delta)-\hat E(n;R_\delta)&
\leq \hat E(\bar n;R_1)-\hat E(n;R_\delta) \\
&\quad + \hat E(\bar n;R_2) + \hat E(\bar n; R_3).
\end{align*}
Direct computations and \eqref{eq:etadeltasmall}-\eqref{eq:gammasmall} show that
\begin{align*}
\hat E(\bar n ;R_1)-\hat E(n;R_\delta) &\leq C \gamma^2 \hat E(n;R_\delta) \leq C \left(\ln\frac 1\delta\right)^{-1},\\
\hat E(\bar n ;R_3)&\leq  C(\gamma^2+\eta(2\delta))\leq C \left(\ln\frac 1\delta\right)^{-1}, \\
\hat E(\bar n;R_2)&\leq C \eta(2\delta)\leq C \left(\ln\frac 1\delta\right)^{-1},
\end{align*}
so that in $D^{ext}_{\delta}\setminus D^{ext}_{2\delta}$ the energy excess of our comparison map is also negligible. From the minimizing property of $n$ we thus obtain
\begin{align*}
0 &\leq \hat E(\bar n;D^{ext}_\delta)-\hat E(n;D^{ext}_\delta)\\
&\leq \int_{D^{ext}_{\sqrt{2\delta}}}\left(\abs{\nabla \tilde n}^2-\abs{\nabla n}^2\right)\rho\,d\rho\,dz + o(1)\quad\text{as }\delta\to 0.
\end{align*}
Since  $\abs{\nabla\tilde n}\leq\abs{\nabla n}$ in $D$, we deduce that we have in fact $\abs{\nabla \tilde n}=\abs{\nabla n}$ in $D$, which implies
\begin{equation*}
n_1\nabla n_2 - n_2\nabla n_1 =0\quad\text{in }D.
\end{equation*}
In particular, for any fixed $\rho,z>0$ with $\rho^2+z^2=1$, the map $m\colon t\mapsto (n_1,n_2)(t\rho,tz)$ satisfies $m \wedge \dot m =0$, and thus $\dot m= \alpha m$, where  $\alpha=(m\cdot\dot m) / \abs{m}^2$ is continuous on $[1,\infty)$ because $\abs{m}^2$ does not vanish there. Since $m_2(1)=0$ this implies that $m_2(t)=0$ for all $t\geq 1$, hence $n_2\equiv 0$ in $D$.
\end{proof}

Thanks to Proposition~\ref{p:addsym}, we can apply the regularity results on symmetric harmonic maps in \cite{HKL90} to deduce that $n_\star$ is analytic in $\overline\Omega\setminus( \mathcal C \cup Z)$, where $Z\subset\Omega\cap\lbrace\rho=0\rbrace$ is a discrete set of singular points on the vertical axis. By Lemma~\ref{l:H1conv} we therefore have $Q_\xi\to Q_\star$ in $C_{loc}^{1,\alpha}(\overline\Omega\setminus(\mathcal C\cup Z))$. 

\subsection{The absence of point defects}\label{ss:limsmooth}

To conclude the proof of Theorem~\ref{thm:main} $(iii)$ it remains to show that $Z$ is empty. 
The starting idea is to try and apply
the argument of \cite[Theorem 13]{colloid} (using reflections of the image in $\mathbb{S}^2$ and analyticity) to eliminate all point defects that are not required by topology. 
If the defect ring is negatively charged, that is $\tau=+1$ 
 in Lemma~\ref{l:n},  this argument may be applied to eliminate all point defects, and 
we carry out this analysis in \S~\ref{sss:pos}.
However, if the ring happens to be positively charged---corresponding to $\tau=-1$---then an additional pair of point defects would be required. To complete the proof we therefore need to show that the case $\tau=-1$ can not occur, and this is demonstrated in \S~\ref{sss:neg}.

\subsubsection{The case of a negatively charged ring $\tau=+1$}\label{sss:pos}

The argument of \cite[Theorem 13]{colloid} used to eliminate extraneous point defects relies on the construction of comparison maps, and thus we again face the sticky issue of controlling the boundary conditions on $\partial D^{ext}_\delta\cap D$.
 Specifically, we need to know that $n_3$ does not change sign on that boundary part, and 
our first step is therefore to gather stronger information about the trace of $n$ there.

\begin{lemma}\label{l:n3monot}
There exists $r_n\to 0$ such that $\theta\mapsto n_3(r_n,\theta)$ is strictly monotone on $[0,\theta_0(r_n)]$. 
\end{lemma}
\begin{proof}[Proof of Lemma~\ref{l:n3monot}]
Since $D$ is simply connected, there exists a lifting $\varphi\in H^1_{loc}(D;\mathbb R)$ such that
\begin{equation*}
n=(\cos\varphi,0,\sin\varphi).
\end{equation*}
This lifting $\varphi$ is in fact smooth up to the boundary of $D$, except at points of the singular set $Z$ and at $(\rho,z)=(1,0)$. It is defined up to a constant multiple of $2\pi$, that one may fix by imposing $\varphi(r,\theta_0(r))=\theta_1(r)$, where we recall the definition \eqref{eq:theta1} of $\theta_1(r)$. For $\theta=0$ we then have $\varphi\equiv\tau\pi/2 + 2N\pi$ for some $N\in\mathbb Z$. This implies
\begin{align*}
\int_0^{\theta_0}\abs{\partial_\theta n}^2d\theta & =\int_0^{\theta_0}\abs{\partial_\theta\varphi}^2d\theta  \geq \frac{1}{\theta_0}\left(\int_0^{\theta_0}\partial_\theta\varphi \, d\theta \right)^2 \\
& = \frac{(\theta_1-\tau\pi/2 +2N\pi)^2}{\theta_0} = \frac \pi 2 (1-4\tau N)^2 + O(r).
\end{align*}
Recalling \eqref{eq:bound_dn}, we deduce that $N=0$. Therefore, for $r$ small we expect $n$ to be close to the  map
\begin{equation}\label{eq:varphi0}
n_0 =(\cos\varphi_0,0,\tau\sin\varphi_0),\qquad\varphi_0=\pi/2-\theta.
\end{equation}
In fact we have
\begin{align*}
\int_0^{\theta_0} \abs{\partial_\theta \varphi-\tau\partial_\theta \varphi_0}^2 \, d\theta
&=\int_0^{\theta_0}\abs{\partial_\theta\varphi}^2\,d\theta-\frac\pi 2 +(1+2\tau)\arcsin\frac r2 \\
&=\int_0^{\theta_0}\abs{\partial_\theta n}^2d\theta-\frac \pi 2 +O(r),
\end{align*}
and since $\varphi=\tau\varphi_0+O(r)$ at $\theta=0$ and $\theta=\theta_0(r)$, together with \eqref{eq:bound_dn} this shows that
\begin{equation}\label{eq:phi0arcs}
\int_0^1 \norm{\varphi(r,\cdot)-\tau\varphi_0(r,\cdot)}_{H^1(0,\theta_0(r))}\,\frac{dr}{r} <\infty.
\end{equation}
In particular, there are arbitrarily small $\delta$'s such that $n$ is very close to $n_0$ on $\partial D^{ext}_\delta\cap D$, in $H^1$ and thus also in $L^\infty$. 
Using the equation satisfied by $\varphi$ one can obtain a stronger estimate. Since $\varphi$ minimizes
\begin{equation*}
F(\varphi;D^{ext}_\delta)=\hat E(n;D^{ext}_\delta)=\int_{D^{ext}_\delta}\left[\abs{\nabla\varphi}^2+\frac{\cos^2\varphi}{\rho^2}\right]\rho\, d\rho\, dz,
\end{equation*}
it solves the Euler Lagrange equation
\begin{equation*}
\Delta\varphi +\frac 1\rho \partial_\rho\varphi =  -\frac{1}{\rho^2}\sin(2\varphi).
\end{equation*}
Using also $\Delta\varphi_0=0$,  rescaled elliptic estimates enable us to obtain a stronger control on $\varphi-\tau\varphi_0$ in the annular domain
\begin{equation*}
A_\lambda =\left\lbrace \frac\lambda 2 \leq r \leq 2\lambda\right\rbrace \cap D.
\end{equation*}
Specifically, elliptic estimates in the rescaled domain $\lambda^{-1}A_\lambda$ (which is Lipschitz independently of $\lambda$) give control on the following scale invariant quantity:
\begin{align*}
g(\lambda)&:=\norm{\varphi-\tau\varphi_0}^2_{L^\infty(A_\lambda)}
+\norm{\nabla\varphi -\tau\nabla\varphi_0}^2_{L^2(A_\lambda)}
+\norm{\nabla^2\varphi -\tau\nabla^2\varphi_0}^2_{L^1(A_\lambda)}\\
&\leq C \left(\lambda^2 + \norm{\varphi-\tau\varphi_0}^2_{L^\infty(\lbrace r=\lambda/2\rbrace)} + \norm{\varphi-\tau\varphi_0}^2_{L^\infty(\lbrace r=2\lambda\rbrace)} \right)\\
&\leq C\Big(\lambda^2 +  \norm{\varphi(\lambda/2,\cdot)-\tau\varphi_0(\lambda/2,\cdot)}^2_{H^1(0,\theta_0(\lambda/2))} \\
&\hspace{12em}+ \norm{\varphi(2\lambda,\cdot)-\tau\varphi_0(2\lambda,\cdot)}^2_{H^1(0,\theta_0(2\lambda))}\Big).
\end{align*}
Hence from \eqref{eq:phi0arcs} we deduce
\begin{align*}
\int_0^1 g(\lambda)\frac{d\lambda}{\lambda} <\infty,
\end{align*}
and there exists $\lambda_n\to 0$ such that $g(\lambda_n)\to 0$.
Moreover, by Sobolev embedding, it we have
\begin{align*}
&\int_{\lambda/2}^{2\lambda}\norm{\partial_\theta\varphi(r,\cdot)-\tau\partial_\theta\varphi_0(r,\cdot)}_{L^\infty(0,\theta_0(r))}\frac{dr}{r}
\\
&\leq C\int_{\lambda/2}^{2\lambda}\left( \norm{\partial_\theta\varphi(r,\cdot)-\tau\partial_\theta\varphi_0(r,\cdot)}_{L^2(0,\theta_0(r))} +  \norm{\partial_\theta^2\varphi(r,\cdot)-\tau\partial_\theta^2\varphi_0}_{L^1(0,\theta_0(r))}\right)\frac{dr}{r}\\
&\leq C\sqrt{g(\lambda)},
\end{align*}
thus we may find $r_n\in [\lambda_n/2,2\lambda_n]$ such that
\begin{equation*}
\norm{\partial_\theta\varphi(r_n,\cdot)-\tau\partial_\theta\varphi_0(r_n,\cdot)}_{L^\infty(0,\theta_0(r_n))}\leq C \sqrt{g(\lambda_n)}\longrightarrow 0.
\end{equation*}
Since $\partial_\theta\varphi_0=-1$, this implies in particular that $\varphi(r_n,\cdot)$ is strictly monotone for $n$ large enough. 
\end{proof}

Equipped with Lemma~\ref{l:n3monot}, we are now in a position to apply an argument similar to the one in \cite[Theorem~13]{colloid}. It enables us to conclude that $Z$ is empty if $\tau=+1$:

\begin{corollary}\label{cor:Z}
The set of singular points $Z\cap \lbrace z>1\rbrace$ is empty if $\tau=1$.
\end{corollary}
\begin{proof}[Proof of Corollary~\ref{cor:Z}]
Assume that $\tau=+1$. Take $\delta=r_n$ provided by Lemma~\ref{l:n3monot}. Then $n_3>0$ on $\partial D^{ext}_\delta\cap\lbrace \rho>0\rbrace$. Therefore, the map $\tilde n =(n_1,n_2,\abs{n_3})$ is an admissible comparison map in $D^{ext}_\delta$ and has the same energy as $n$, hence is a minimizer and must be analytic inside $D^{ext}_\delta$. This implies that $n_3\geq 0$ inside $D^{ext}_\delta$. Since $Z$ corresponds to changes of sign of $n_3$, we deduce that $Z=\emptyset$.
%
\end{proof}

%

\subsubsection{Ruling out the case of a positively charged ring $\tau=-1$}\label{sss:neg}

 The end of the proof will consist in ruling out the case $\tau=-1$. This is the most delicate part of the argument, and it has two main steps:
 \begin{itemize}
 \item  In the first step we consider the complement of a small region around the ring defect, and show that the energy cost of a point defect away from this neighborhood is a strictly positive quantity of order $O(1)$.  This is done in Lemma~\ref{l:compareEtau}, using a variation of the argument already used in Corollary~\ref{cor:Z}, but with more precise boundary estimates.
\item
In the second step we derive a more precise estimate of the energy concentrated in
a small region around the ring defect, in order to conclude that the O(1) increase obtained in the first step (away from the ring defect) leads to a strict increase of the total energy. Specifically we show that the core energy of the ring defect is independent of the ring's charge, up to an error of smaller order.  To this end, we construct (in  Lemma~\ref{l:Pxi}) a comparison map which modifies boundary values in a singular region between the particle and a curve tangent to it; the error thus introduced can be controlled thanks to Lemma~\ref{l:finelowerboundhalfannulus}. 
\end{itemize}
For later use, we derive the following useful estimates, based on those established in the  proof of  Lemma~\ref{l:n3monot}:
\begin{lemma}\label{l:estimphi0}
As $\delta\to 0$, we have
\begin{equation*}
\norm{\varphi-\tau\varphi_0}_{L^\infty(D^{int}_\delta)} + \norm{\nabla\varphi-\tau\nabla\varphi_0}_{L^2(D^{int}_\delta)}\longrightarrow 0,
\end{equation*}
where $D^{int}_\delta=D\setminus D^{ext}_\delta=\{(\rho,z)\in D: \ (\rho-1)^2 + z^2\le \delta^2\}$.
\end{lemma}
\begin{proof}[Proof of Lemma~\ref{l:estimphi0}]
It is shown in Lemma~\ref{l:n3monot} that
\begin{equation*}
f(\lambda)=\norm{\varphi-\tau\varphi_0}^2_{L^\infty(A_\lambda)}
+\norm{\nabla\varphi -\tau\nabla\varphi_0}^2_{L^2(A_\lambda)},
\end{equation*}
satisfies
\begin{equation*}
\int_0^1 f(\lambda)\frac{d\lambda}{\lambda} = \sum_{n=0}^\infty \int_{1/2}^1 f(2^{-n}\lambda)\frac{d\lambda}{\lambda} <\infty.
\end{equation*}
We may thus pick $\lambda_n\in[1/2,1]$ such that
\begin{equation*}
\sum_{n=0}^\infty f(2^{-n}\lambda_n)<\infty.
\end{equation*}
Since $A_{(2^{-n-1}\lambda_{n+1})}$ overlaps with $A_{(2^{-n}\lambda_n)}$, we deduce that 
\begin{align*}
& \norm{\varphi-\tau\varphi_0}^2_{L^\infty(D\setminus D^{ext}_{2^{-N}})} + \norm{\nabla\varphi-\tau\nabla\varphi_0}^2_{L^2(D\setminus D^{ext}_{2^{-N}})}\\
&\leq \sum_{n=N-1}^\infty f(2^{-n}\lambda_n) \longrightarrow 0,
\end{align*}
as $N\to\infty$.
\end{proof}

In order to rule out the case $\tau=-1$, we wish to estimate the difference of total energy between the two cases $\tau=\pm 1$. To this end, we recall the definition \eqref{eq:theta1} of $\theta_1(r)$  in representing the boundary condition on the colloid, and introduce the notation:
\begin{align}
\mathbb E^\tau [\delta] & = \min \Big\lbrace F(\varphi;D^{ext}_\delta)\colon
\varphi = \tau\pi/2\text{ for }\theta=0,\nonumber\\
&\hspace{8.5em}\varphi(r,\theta_0(r)) = \theta_1(r)\text{ for }\delta< r<\sqrt 2\nonumber\\
&\hspace{8.5em}\varphi = \phi_{0}^{\tau}\text{ for }r=\delta\Big\rbrace,\nonumber\\
\phi_{0}^{\tau}(r,\theta)&=\tau\frac\pi 2 -\lambda(r)\theta,\quad\lambda(r)=\frac{\tau\pi/2-\theta_1(r)}{\theta_0(r)}.\label{eq:phi0pm}
\end{align}
Then we compare $\mathbb E^+[\delta]$ and $\mathbb E^-[\delta]$ for small $\delta$:
\begin{lemma}\label{l:compareEtau}
we have
\begin{equation*}
\limsup_{\delta\to 0}\left(\mathbb E^+[\delta]-\mathbb E^-[\delta]\right) < 0.
\end{equation*}
\end{lemma}
\begin{proof}[Proof of Lemma~\ref{l:compareEtau}]
We denote by $\varphi_\delta^\tau$ the minimizer in $\mathbb E^\tau[\delta]$. A simple construction provides the upper bound
\begin{equation*}
F(\varphi_\delta^\tau;D^{ext}_\eta)\leq \frac\pi 2\ln\frac 1\eta +  C,\qquad\forall\eta\geq\delta.
\end{equation*}
Arguing as in Lemma~\ref{l:H1conv} we thus have, up to extracting a further subsequence, a limit $\varphi_\delta^\tau\to\varphi^\tau$ in $H^1_{loc}(D)$ as $\delta\to 0$, and  $F(\varphi_\delta^\tau;D^{ext}_\eta)\to F(\varphi^\tau;D^{ext}_\eta)$ for all $\eta>0$. The function $\varphi^\tau$ satisfies the boundary conditions
\begin{align*}
\varphi &= \tau\pi/2\text{ for }\theta=0,\\
\varphi(r,\theta_0(r)) &= \theta_1(r)=2\theta_0(r)-\pi\text{ for }\delta< r<\sqrt 2,
\end{align*}
and minimizes $F(\cdot;D^{ext}_\eta)$ among functions that agree with $\varphi^\tau$ on $\partial D^{ext}_\eta \cap \lbrace \rho>0\rbrace$, for all $\eta>0$. The arguments in Lemma~\ref{l:n3monot} and in Lemma~\ref{l:estimphi0} carry over, and we find that $\varphi^\tau$ is analytic in $\overline D \setminus (Z\cup \lbrace (1,0)\rbrace)$, where $Z=\emptyset$ if $\tau=+1$, and $Z=\lbrace (0,z_0)\rbrace$ for some $z_0>0$ if $\tau=-1$. Moreover we have
\begin{equation*}
\norm{\varphi^\tau-\tau\varphi_0}_{L^\infty(D^{int}_\eta)} + \norm{\nabla\varphi^\tau-\tau\nabla\varphi_0}_{L^2(D^{int}_\eta)}\to 0\quad\text{as }\eta\to 0.
\end{equation*}
We set $\psi = \abs{\varphi^-}$, so that
\begin{equation*}
F(\psi;D^{ext}_\delta)=F(\varphi^-;D^{ext}_\delta)\quad\text{ and }\psi=\varphi^+\text{ on }\partial D^{ext}_\delta \cap\lbrace \rho>0\rbrace.
\end{equation*}
for all $\delta>0$. Next, from the estimates for $\varphi^-$ and easy estimates on $(\abs{\varphi_0}-\varphi_0)$ and $(\varphi_0-\phi_0^+)$ we deduce that  
\begin{equation}\label{eq:estimpsi}
\norm{\psi-\phi_{0}^+}_{L^\infty(D^{int}_\eta)} + \norm{\nabla\psi-\nabla\phi_0^+}_{L^2(D^{int}_\eta)}\to 0\quad\text{as }\eta\to 0.
\end{equation}
Also note that, since $\varphi^-$ changes sign at one point of $\partial D^{ext}_\delta\cap D$ for small enough $\delta$, $\psi$ can not be locally analytic inside $D^{ext}_\delta$. In particular, it is certainly not a minimizer of $F(\cdot;D^{ext}_\delta)$ for any $\delta$ small enough. Therefore, for some small fixed $\delta_0$ there exists a function $\xi$ such that $\xi=\psi$ on $\partial D^{ext}_{\delta_0}\cap\lbrace \rho>0\rbrace$,  and
\begin{equation*}
\varepsilon := F(\psi;D^{ext}_{\delta_0})-F(\xi;D^{ext}_{\delta_0})>0.
\end{equation*}
Let $\eta<\delta_0/2$.
Consistently with its boundary conditions, we may extend $\xi$ to $D$ by setting 
\begin{equation*}
\xi=\psi\text{ in }D\setminus D^{ext}_{\delta_0}.
\end{equation*}
 Next we introduce a modified function $\xi_\eta$  given by
\begin{align*}
\xi_\eta & = \mu \xi + (1-\mu)\phi_0^+,\\
\mu & = \mu(r) = \begin{cases}
1 & \text{ for }r>2 \eta,\\
0 & \text{ for }r<\eta,\\
\frac 1{\ln 2}\ln\frac{r}{\eta} &\text{ for }\eta<r<2\eta,
\end{cases}
\end{align*}
so that $\xi_\eta=\xi$ in $D^{ext}_{2\eta}$, $\xi_\eta=\varphi_\eta^+$ on $\partial D^{ext}_\eta\cap\lbrace \rho>0\rbrace$, and
\begin{align*}
&\norm{\xi-\xi_\eta}^2_{L^\infty(D^{int}_{2\eta})}+\norm{\nabla\xi-\nabla\xi_\eta}^2_{L^2(D^{int}_{2\eta})}\\
 &\leq \left(1+2 \int_\eta^{2\eta} (\mu')^2 r\,dr \right)\norm{\psi-\phi_0^+}^2_{L^\infty(D^{int}_{2\eta})}
 +2 \norm{\nabla\psi-\nabla\phi_0^+}^2_{L^2(D^{int}_{2\eta})}\\
 & = 3 \norm{\psi-\phi_0^+}^2_{L^\infty(D^{int}_{2\eta})}
 +2 \norm{\nabla\psi-\nabla\phi_0^+}^2_{L^2(D^{int}_{2\eta})}.
\end{align*}
Thanks to \eqref{eq:estimpsi}, we therefore have a function $R(\eta)$ which tends to zero as $\eta\to 0$, such that
\begin{align*}
\varepsilon & = F(\psi;D^{ext}_{\eta})-F(\xi;D^{ext}_{\eta})=F(\varphi^-;D^{ext}_\eta)-F(\xi;D^{ext}_{\eta}) \\
&\leq F(\varphi^-;D^{ext}_\eta)-F(\xi_\eta;D^{ext}_\eta) + R(\eta) \\
&\leq F(\varphi^-;D^{ext}_\eta)-\mathbb E^+[\eta] + R(\eta).
\end{align*}
The last inequality holds by definition of $\mathbb E^+$ because $\xi_\eta=\varphi_\eta^+$ on $\partial D^{ext}_\eta\cap\lbrace \rho>0\rbrace$. Recalling the definition of $\mathbb E^-$ and taking the limit as $\eta\to 0$, we find
\begin{equation*}
\limsup_{\eta\to 0}\left(\mathbb E^+[\eta]-\mathbb E^-[\eta]\right) \leq
-\varepsilon + \liminf_{\eta\to 0}\left(F(\varphi^-;D^{ext}_\eta)-F(\varphi_\eta^-;D^{ext}_\eta)\right).
\end{equation*}
The lemma will be proven once we show that
\begin{equation}\label{eq:limFeta}
\limsup_{\eta\to 0}\left(F(\varphi^-;D^{ext}_\eta)-F(\varphi_\eta^-;D^{ext}_\eta)\right) \leq 0.
\end{equation}
To this end we may, consistently with its boundary conditions, extend $\varphi_\eta^-$ to $D$ by setting
\begin{equation*}
\varphi_\eta^-=\phi_0^-\quad\text{ in }D^{int}_\eta.
\end{equation*}
We also introduce a parameter $\nu<\eta/2$. We have
\begin{align*}
F(\varphi^-;D^{ext}_\eta)-F(\varphi_\eta^-;D^{ext}_\eta)
\leq F(\varphi^-;D^{ext}_\nu)-F(\varphi_\eta^-;D^{ext}_\nu) + C U(\eta),
\end{align*}
where 
\begin{equation*}
U(\eta) =\norm{\varphi^--\phi_0^-}^2_{L^\infty(D^{int}_\eta)} + \norm{\nabla\varphi^--\nabla \phi_0^-}^2_{L^2(D^{int}_\eta)}\to 0\quad\text{as }\eta\to 0.
\end{equation*}
Next we modify $\varphi_\eta^-$ in order to use the minimizing property of $\varphi^-$ in $D^{ext}_\nu$. Similarly to the above definition of $\xi_\eta$, we set
\begin{align*}
\tilde \varphi_\nu & = \mu \varphi_\eta^- + (1-\mu)\varphi^-,\\
\mu & = \mu(r) = \begin{cases}
1 & \text{ for }r>2 \nu,\\
0 & \text{ for }r<\nu,\\
\frac 1{\ln 2}\ln\frac{r}{\nu} &\text{ for }\nu<r<2\nu,
\end{cases}
\end{align*}
so that $\tilde\varphi_\nu=\varphi_\eta^-$ in $D^{ext}_{2\nu}$, and $\tilde\varphi_\nu =\varphi^-$ on $\partial D^{ext}_\nu \cap \lbrace \rho>0\rbrace$. Moreover, since $\varphi_\eta^-=\phi_0^-$ in $D^{int}_{2\nu}$, we have
\begin{align*}
&F(\varphi_\eta^-;D^{ext}_\nu)-F(\tilde\varphi_\nu;D^{ext}_\nu)\\
&\leq C\left( \norm{\varphi^--\phi_0^-}^2_{L^\infty(D^{int}_{2\nu})} + \norm{\nabla\varphi^--\nabla \phi_0^-}^2_{L^2(D^{int}_{2\nu})}\right)\\
&\leq C U(\eta).
\end{align*}
We deduce
\begin{align*}
F(\varphi^-;D^{ext}_\eta)-F(\varphi_\eta^-;D^{ext}_\eta) 
&\leq F(\varphi^-;D^{ext}_\nu) -F(\tilde\varphi_\nu;D^{ext}_\nu) + C U(\eta)\\
&\leq C U(\eta).
\end{align*}
The last inequality holds because $\tilde\varphi_\nu=\varphi^-$ on $\partial D^{ext}_\nu\cap\lbrace\rho>0\rbrace$ and $\varphi^-$ is minimizing in $D^{ext}_\nu$. This obviously implies \eqref{eq:limFeta}.
\end{proof}

We would like to use Lemma~\ref{l:compareEtau} to show that $\tau$ must be $+1$. From now on we assume that $\tau=-1$. We will then construct a map $P_\xi\in\mathcal H_{sym}$ with lower energy than $Q_\xi$, hence contradicting the minimality of $Q_\xi$ and proving that $\tau=+1$ and $Z=\emptyset$.  
We introduce the notations
\begin{align*}
n_\delta^\pm &= (\cos\varphi_\delta^\pm,0,\sin\varphi_\delta^\pm),\\
n_0^\pm&=(\cos\phi_0^\pm,0,\sin\phi_0^\pm),
\end{align*} 
where $\varphi_\delta^\pm$ are the minimizers corresponding to the minimization problems $\mathbb E^\pm[\delta]$. Note in particular that $n_\delta^\pm =n_0^\pm$ for $r=\delta$.
 
First we show that we may, without messing too much with the energy of $\widetilde Q_\xi$ inside $D^{int}_\delta$, replace it with a map that equals $n_0^-$ on $D\cap\partial D^{ext}_\delta$.

\begin{lemma}\label{l:Rxi}
Consider $\widetilde R_\xi\colon D^{int}_{\delta}\to\mathcal S_0$ minimizing $\widetilde E_\xi(\cdot;D^{int}_\delta)$ among all maps $R$ with the boundary constraints
\begin{align*}
R & = e_r\otimes e_r -\frac 13 I\quad\text{for }\rho^2+z^2=1,\\
R & = n_0^-\otimes n_0^- -\frac 13 I\quad\text{for }r=\delta,\\
R & = S R S^t\quad\text{for }z=0.
\end{align*}
Then we have that 
\begin{equation*}
\widetilde E_\xi(\widetilde R_\xi;D^{int}_\delta)\leq \widetilde E_\xi(\widetilde Q_\xi;D^{int}_\delta) + \sigma_1(\delta,\xi) + \zeta_1(\delta),
\end{equation*}
where $\zeta_1(\delta)\to 0$ as $\delta\to 0$, and $\sigma_1(\delta,\xi)\to 0$ as $\xi\to 0$ for all fixed $\delta>0$.
\end{lemma} 

\begin{proof}[Proof of Lemma~\ref{l:Rxi}]
 We construct a test configuration $R_\xi$ in $D^{int}_{\delta}$ and evaluate its energy $\widetilde E_\xi$. We systematically denote by $\sigma(\delta,\xi)$ (resp. $\zeta(\delta)$) functions that tend to 0 as $\xi\to 0$ for any fixed $\delta$ (resp. $\delta\to 0$), although they may change from one line to another.
 
By Fubini's theorem, we may choose $\hat\delta\in (\frac\delta 4,\frac\delta 3)$ such that (possibly along a subsequence)
 \begin{equation*}
 \widetilde E_\xi(\widetilde Q_\xi;D\cap\partial D^{ext}_{\hat \delta})+\widetilde E_\star(\widetilde Q_\star;D\cap\partial D^{ext}_{\hat\delta})\lesssim\frac 1\delta \left(\widetilde E_\xi(\widetilde Q_\xi;D^{ext}_{\frac\delta 4})+\widetilde E_\star(\widetilde Q_\star;D^{ext}_{\frac\delta 4}) \right).
 \end{equation*}
 Arguing exactly as in the proof of \eqref{eq:overlineQ} in Lemma~\ref{l:H1conv}, this enables us to construct $R_\xi$ in $D^{ext}_{\hat\delta}\setminus D^{ext}_{\delta/2}$ such that $R_\xi$ satisfies the boundary conditions of $\widetilde R_\xi$ for $z=0$ and $\rho^2+z^2=1$, and 
 \begin{align*}
 R_\xi & =\widetilde Q_\xi \quad\text{for }r=\hat\delta,\qquad R_\xi =\widetilde Q_\star\quad\text{for }r=\frac\delta 2,\\
 \text{and }\sigma(\delta,\xi)&=\widetilde E_\xi(R_\xi;D^{ext}_{\hat\delta}\setminus D^{ext}_{\delta/2})-\widetilde E_\star(\widetilde Q_\star;D^{ext}_{\hat\delta}\setminus D^{ext}_{\delta/2}) \longrightarrow 0\quad\text{as }\xi\to 0,
 \end{align*}
 for any fixed $\delta$. In $D^{int}_{\hat\delta}$ we set $R_\xi=\widetilde Q_\xi$, so that the above estimate combined with Lemma~\ref{l:H1conv} implies
 \begin{equation*}
 \widetilde E_\xi(R_\xi;D^{int}_{\frac\delta 2})\leq \widetilde E_\xi(\widetilde Q_\xi;D^{int}_{\frac\delta 2})+\sigma(\delta,\xi).
 \end{equation*}
Finally, we interpolate in $\mathcal U_\star$ between $\widetilde Q_\star=\widetilde n_\star\otimes n_\star-\frac13 I$ and $Q_0^-:=n_0^-\otimes n_0^- -\frac13 I$ to define $R_\xi$ in the remaining region $D^{ext}_{\delta/2}\setminus D^{ext}_{\delta}$.  
This we do via the phase variables, $\varphi_\star$ and $\phi_0^-$, setting 
\[ 
 \hat \phi(r,\theta)=  \frac{2}{\delta}\left(r-\frac{\delta}{ 2}\right) \varphi_\star
                         +\frac{2}{\delta}\left(\delta-r\right) \phi_0^-, 
                            \qquad  \frac{\delta}{2}<r<\delta, \ 0<\theta< \theta_0(r).
\]
This defines a director $\hat n:= (\cos\hat \phi, 0, \sin\hat \phi)$ and an associated uniaxial $Q$-tensor
$R_\xi=\hat n\otimes\hat n -\frac13 I$. 
In this way, $R_\xi$ will be continuous in $D^{int}_\delta$ and satisfy each of the desired conditions on $\partial (D^{int}_\delta)$.  
As $R_\xi\in\mathcal U_\star$ in this region,  $f(R_\xi)=0$ and moreover,
\[  \widetilde E( R_\xi;   D^{ext}_{\delta/2}\setminus D^{ext}_{\delta}) = 2 \hat E(\hat n; D^{ext}_{\delta/2}\setminus D^{ext}_{\delta}) = 2 F(\hat\phi; D^{ext}_{\delta/2}\setminus D^{ext}_{\delta}).
\]
Thanks to Lemma~\ref{l:estimphi0} and the explicit form of $\varphi_0$ \eqref{eq:varphi0} and $\phi_0^-$ \eqref{eq:phi0pm}, we have
\begin{equation}\label{phases}
\| \varphi_\star - \phi_0^-\|_{H^1\cap L^\infty(D^{int}_\delta)} \longrightarrow 0, \qquad\text{as }\delta\to 0,
\end{equation}
and use this fact to estimate the energy of $R_\xi$ in $D^{ext}_{\delta/2}\setminus D^{ext}_{\delta}$.
  For instance, we have
\begin{align*}
\partial_r \hat\phi (r,\theta) &= 
   \frac{2}{\delta}\left( r-\frac{\delta}{ 2}\right) \partial_r\varphi_\star
                         +\frac{2}{\delta}\left(\delta-r\right) \partial_r\phi_0^-
                            + \frac{2}{ \delta} (\varphi_\star-\phi_0^-) \\
                    &= \partial_r \phi_0^- + \left\{\frac{2}{\delta}\left(r-\frac{\delta}{ 2}\right)  
                       [\partial_r \phi_0^- - \partial_r \varphi_\star] + \frac{2}{ \delta} (\varphi_\star-\phi_0^-)\right\},
\end{align*}
and similarly,
\[   \partial_\theta \hat\phi (r,\theta) = \partial_\theta \phi_0^- + \left\{\frac{2}{\delta}\left(r-\frac{\delta}{ 2}\right)  [\partial_\theta \phi_0^- - \partial_\theta \varphi_\star] \right\}.
\]
The estimate \eqref{phases} ensures that the bracketed terms on the right-hand side of each of the above equations tend to zero in $L^2(D^{ext}_{\delta/2}\setminus D^{ext}_{\delta})$ as $\delta\to 0$.  As a consequence we have
\begin{align*}   
\int_{D^{ext}_{\delta/2}\setminus D^{ext}_\delta} |\nabla\hat\phi|^2 \, \rho\, d\rho\, dz
   &\le \int_{D^{ext}_{\delta/2}\setminus D^{ext}_\delta} |\nabla\phi_0^-|^2 \, \rho\, d\rho\, dz
       + \zeta(\delta) \\
   &\le \int_{D^{ext}_{\delta/2}\setminus D^{ext}_\delta} |\nabla\varphi_*|^2 \, \rho\, d\rho\, dz
       + \zeta(\delta),
\end{align*}
where $\zeta(\delta)\to 0$ as $\delta\to 0$, and we used again \eqref{phases} for the last inequality.
  As 
\[  
\int_{D^{ext}_{\delta/2}\setminus D^{ext}_\delta} \frac{\hat n_1^2}{\rho} d\rho\, dz = O(\delta^2),  
\]
 we deduce that
\[  F(\hat \phi; D^{ext}_{\delta/2}\setminus D^{ext}_\delta) \le F(\varphi_\star;D^{ext}_{\delta/2}\setminus D^{ext}_\delta) + \zeta(\delta).  
\]
Finally, 
\begin{align*}
  \widetilde E(R_\xi; D^{ext}_{\delta/2}\setminus D^{ext}_\delta) &= 2 F(\hat \phi; D^{ext}_{\delta/2}\setminus D^{ext}_\delta) \le 2[F(\varphi_*) + \zeta(\delta)] \\
    &\le \widetilde E(\widetilde Q_\xi; D^{ext}_{\delta/2}\setminus D^{ext}_\delta) + \zeta(\delta) + \sigma(\delta,\xi),
\end{align*}
by Lemma~\ref{l:H1conv}.  Combined with the above estimate in $D^{int}_{\delta/2}$ this yields
\begin{equation*}
\widetilde E(R_\xi; D^{int}_\delta)\leq \widetilde E (\widetilde Q_\xi; D^{int}_\delta) + \zeta(\delta) + \sigma(\delta,\xi),
\end{equation*}        
and since $\widetilde R_\xi$ minimizes $\widetilde E_\xi$ with the same boundary conditions as $R_\xi$ this completes the proof of Lemma~\ref{l:Rxi}.    
\end{proof}

The final step consists in proving that we may \enquote{transform} the boundary conditions on $D\cap\partial D^{ext}_\delta$ from $n_0^-$ to $n_0^+$ without increasing the energy too much. 
This establishes the crucial core energy estimate mentioned in the Introduction:  that the energy of a positively or negatively charged line defect is the same up to $o(1)$.

\begin{lemma}\label{l:Pxi}
Consider $\widetilde P_\xi\colon D^{int}_{\delta}\to\mathcal S_0$  minimizing $\widetilde E(\cdot;D^{int}_{\delta})$ among all maps $P$ with the boundary constraints
\begin{align*}
P & = e_r\otimes e_r -\frac 13 I\quad\text{for }\rho^2+z^2=1,\\
P & = n_0^+\otimes n_0^+ -\frac 13 I\quad\text{for }r=\delta,\\
P& = S P S^t\quad\text{for }z=0.
\end{align*}
Then 
\begin{equation*}
\widetilde E_\xi(\widetilde P_\xi;D^{int}_{\delta}) \leq \widetilde E(\widetilde R_\xi;D^{int}_\delta) + \zeta_3(\delta),\\
\end{equation*}
where $\zeta_3(\delta)\to 0$ as $\delta\to 0$, and $\sigma_3(\delta,\xi)\to 0$ as $\xi\to 0$ for all fixed $\delta>0$. 
\end{lemma}
\begin{proof}[Proof of Lemma~\ref{l:Pxi}]
We define a map $\widehat P_\xi$ satisfying the boundary condition of $\widetilde P_\xi$ and an adequate upper bound on its energy. We do this in two steps: first we  define $\widehat P_\xi$ in the domain
\begin{equation*}
X_\delta:=\lbrace 0<\theta<\frac\pi 2 -r^{\frac 12}\rbrace \cap \lbrace 0<r<\delta-\delta^{\frac 32}\rbrace,
\end{equation*}
as the reflected map $S\widetilde R_\xi S$, appropriately rescaled to \enquote{fit} into this smaller domain, and then define $\widehat P_\xi$ on the remaining part by interpolating in $\mathcal U_\star$ between the boundary values of $S \widetilde R_\xi S$ and the boundary values of $\widetilde P_\xi$.

For the first step we start by defining a bi-Lipschitz change of variables which transforms $X_\delta$ into 
\begin{equation*}
D^{int}_\delta =\lbrace 0<\theta <\frac\pi 2 + \arcsin\frac r2 \rbrace \cap \lbrace 0<r<\delta\rbrace,
\end{equation*}
and keeps the subdomain
\begin{equation*}
Y_\delta:=\lbrace 0<\theta<\frac\pi 2 -r^{\frac 13}\rbrace \cap \lbrace 0<r<\frac\delta 2\rbrace
\end{equation*}
fixed. Explicitly, we set
\begin{align*}
\Phi(r,\delta)&=(r+g_1(r),\theta + g_2(r,\theta),\\
g_1,g_2 &\equiv 0\qquad\text{in }Y_\delta,\\
g_1(r)&=2\delta^{\frac 12}\frac{1}{1-2\delta^{\frac 12}}(r-\frac \delta 2)\qquad\text{for }\frac \delta 2 < r<\delta-\delta^{\frac 23},\\
g_2(r,\theta)&=r^{\frac 16}\frac{1+r^{-\frac 12}\arcsin\frac r2}{1-r^{\frac 16}}(\theta -\frac\pi 2 + r^{\frac 13})\qquad\text{for }\frac\pi 2 -r^{\frac 13} <\theta < \frac\pi 2 +\arcsin\frac r2.
\end{align*}
Direct computations show that $\Phi$ is one-to-one from $X_\delta$ into $D^{int}_\delta$, that $\Phi=id$ in $Y_\delta$, and that
\begin{equation*}
\abs{\det(D\Phi)-1} + \abs{g_1'} + \abs{\partial_\theta g_2} + r\abs{\partial_r g_2}\lesssim\delta^{\frac 16}.
\end{equation*}
Hence for any function $u(r,\theta)$, the function $\tilde u=u\circ\Phi$ satisfies
\begin{align*}
\abs{\nabla \tilde u}^2&=\abs{\partial_r \tilde u}^2 +\frac 1{r^2}\abs{\partial_\theta\tilde u}^2\\
&\leq \left[ (1+\abs{g_1'})\abs{\partial_r u}\circ\Phi + \abs{\partial_r g_2}\abs{\partial_\theta u}\circ\Phi\right]^2 + (1+\abs{\partial_\theta g_2})^2\abs{\partial_\theta u}^2\circ\Phi\\
&\leq (1+C\delta^{\frac 16})\abs{\nabla u}^2\circ\Phi,
\end{align*}
for some constant $C>0$. Therefore, setting
\begin{equation*}
\widehat P_\xi =S\widetilde R_\xi S \circ\Phi\qquad\text{in }X_\delta,
\end{equation*}
and recalling that $\Phi$ is the identity in $Y_\delta$, we have
\begin{align*}
\widetilde E_\xi(\widehat P_\xi;X_\delta)&=\widetilde E_\xi(\widehat P_\xi; Y_\delta) + \widetilde E_\xi(\widehat P_\xi;X_\delta\setminus Y_\delta)\\
&\leq \widetilde E_\xi(\widetilde R_\xi;Y_\delta) + (1+C \delta^{\frac 16} ) \widetilde E_\xi(\widetilde R_\xi;D^{int}_\delta\setminus Y_\delta) \\
&\leq \widetilde E_\xi(\widetilde R_\xi;D^{int}_\delta) + C\delta^{\frac 16}\widetilde E_\xi(\widetilde R_\xi;D^{int}_\delta\setminus Y_\delta).
\end{align*}
Next, note that the proof of upper and lower bound in Sections~\ref{s:upper} and \ref{s:lower} can be adapted to prove similar bounds on $\widetilde R_\xi$. 
Moreover, thanks to Remark~\ref{r:lowerbound} and the inclusion
\begin{align*}
 Y_\delta & \supset \lbrace r \leq \frac\delta 2\rbrace \cap \lbrace \rho\geq 1+ z^{\frac 43}\rbrace,
\end{align*}
 the lower bound can actually be obtained in $Y_\delta$. As a consequence
we have the upper bound
\begin{equation*}
\widetilde E_\xi(\widetilde R_\xi;D^{int}_\delta\setminus Y_\delta)\leq C\ln\frac 1\delta,
\end{equation*}
and  deduce that
\begin{equation}\label{eq:boundPXdelta}
\widetilde E_\xi(\widehat P_\xi;X_\delta)\leq \widetilde E_\xi(\widetilde R_\xi;D^{int}_\delta) + C\delta^{\frac 16}\ln\frac 1\delta
\end{equation}
On $\partial X_\delta\cap\lbrace z=0\rbrace$, $\widehat P_\xi$ satisfies the boundary condition $S\widehat P_\xi S =\widehat P_\xi$, since $\widetilde R_\xi$ satisfies it as well. 

Finally, we define $\widehat P_\xi$ in the region $D^{int}_\delta\setminus X_\delta$ to satisfy the desired boundary conditions, via interpolation in this thin region of width $O(\delta^{3/2})$.  We decompose 
\begin{align*}
D^{int}_\delta\setminus X_\delta &=Z^1_\delta\cup Z^2_\delta\cup Z^3_\delta,
\end{align*}
 using the arcs 
$\{ r=\delta-\delta^{3/2}, \ \theta\in (\theta_0(r),\underline\theta(r))\}$ and 
$\{ \theta=\underline{\theta}(r), \ r\in (\delta-\delta^{3/2}, \delta)\}$, where we denote $\underline{\theta}(r):= \frac{\pi}{2} - r^{1/2}$. Explicitly, we set
\begin{align*}
Z^1_\delta &=\{ \theta\in (\underline\theta(r),\theta_0(r)), \ 0<r<\delta-\delta^{3/2}\},\\
Z^2_\delta &=\{ 0<\theta<\underline{\theta}(r), \ r\in (\delta-\delta^{3/2}, \delta)\},\\
Z^3_\delta &=\{  r\in (\delta-\delta^{3/2},\delta), \ \theta\in (\underline\theta(r), \theta_0(r))\}
\end{align*}
As the boundary data are all taken with values in $\mathcal U_*$, we may define $\widehat P_\xi= \hat n\otimes\hat n -\frac13 I$, $\hat n=(\cos\hat\phi, 0, \sin\hat\phi)$, by specifying its phase $\hat\phi$.   Similarly, we define the boundary data for $R_\xi$ in terms of a director characterized by its phase, $R_\xi|_{\partial X_\delta}= n_R\otimes n_R-\frac13 I$, $n_R=(\cos\varphi, 0 ,\sin\varphi)$, with $\varphi=\underline\varphi(r)$ on $\underline\Gamma=\{ \theta=\underline\theta(r), \ r\in (0,\delta-\delta^{3/2})\}$, (corresponding to $n=Se_r\circ\Phi$,) and 
$\varphi=\overline\varphi(\theta)=-\phi_0^-\circ\Phi$ on $\overline\Gamma=\{ r=\delta-\delta^{3/2}, \ \theta\in (0,\underline\theta(r))\}$.

In $Z^1_\delta=\{ \theta\in (\underline\theta(r),\theta_0(r)), \ 0<r<\delta-\delta^{3/2}\}$, we interpolate in $\theta\in (\underline\theta(r),\theta_0(r))$ for each fixed $r$:  
\begin{equation*}
    \hat \phi(r,\theta) = \frac{\theta_0(r)- \theta}{ \theta_0(r)-\underline\theta(r)} \underline\varphi(r) + 
   \frac{\theta-\underline\theta(r)}{ \theta_0(r)-\underline\theta(r)} \theta_1(r),
\end{equation*}
where we recall that $\theta_1(r)=2\theta_0(r)-\pi$ gives the Dirichlet condition along the circle $\rho^2+z^2=1$.
As $\underline\varphi(r)-\theta_1(r)=O(r)$, we calculate
\begin{align}\label{gradest}  
(\partial_r\hat\phi)^2 + \frac{1}{ r^2}(\partial_\theta\hat\phi)^2 = O(r^{-2}), 
\end{align}
and hence
\begin{equation*}  
\widetilde E_\xi( \widehat P_\xi; Z^1_\delta) = 2F(\hat\phi; Z^1_\delta) \lesssim \delta^{1/2}.
\end{equation*}
In $Z^2_\delta=\{ 0<\theta<\underline{\theta}(r), \ r\in (\delta-\delta^{3/2}, \delta)\}$ we set
\begin{equation*}
 \hat\phi(r,\theta)= \frac{\delta-r}{ \delta^{3/2}} \overline\varphi(\theta) + 
   \frac{r-\delta-\delta^{3/2}}{ \delta^{3/2}} \phi_0^+(\theta,\delta).
\end{equation*}
As the phase difference $|\overline\varphi(\theta)-\phi_0^+(\theta,\delta)|=O(\delta)$, we again may estimate the gradient as in \eqref{gradest} to obtain
\begin{equation*}
  \widetilde E_\xi( \widehat P_\xi; Z^2_\delta) = 2F(\hat\phi; Z^2_\delta) \lesssim \delta^{1/2}.
\end{equation*}
Lastly, we consider the domain $Z^3_\delta=\{  r\in (\delta-\delta^{3/2},\delta), \ \theta\in (\underline\theta(r), \theta_0(r))\}$, for which $\hat\phi$ has already been defined on $\partial Z^3_\delta$ via the previous two steps.  Indeed, $\hat\phi|_{\partial Z^3_\delta}=\frac{\pi}{2}+h_\delta$, for $h_\delta$ Lipschitz continuous, with sup norm of order $\delta$, and $|\partial_\tau h_\delta|=O(\delta^{-1/2})$ on each edge.  
Define $v_\delta$ as the minimizer of the Dirichlet energy $\int_{Z^3_\delta} |\nabla v|^2$ with $v|_{\partial Z^3_\delta}= h_\delta$.  The domain $Z^3_\delta$ is nearly square, with side length $\delta^{3/2}$; indeed, after rescaling lengths by $\delta^{3/2}$, $\delta^{-3/2}Z^3_\delta$ approaches the unit square as $\delta\to 0$.   In particular elliptic estimates give
\[   
\int_{Z^3_\delta} |\nabla v_\delta|^2\, dx \lesssim \|h_\delta\|_{H^{1/2}(\partial Z^3_\delta)}^2
  \lesssim \|\partial_\tau h_\delta\|_{L^2(\partial Z^3_\delta)} 
     \|h_\delta\|_{L^2(\partial Z^3_\delta)}
        \lesssim \delta^2.
\]
Setting $\hat\phi=\frac{\pi}{2}+ v_\delta$ in $Z^3_\delta$, we then have that
\[ 
 \widetilde E_\xi(P_\xi; Z^3_\delta) 
 = 2 F(\hat\phi; Z^3_\delta)
    =2\int_{Z^3_\delta} \left[|\nabla v_\delta|^2 + \frac{\sin^2 v_\delta}{\rho^2} \right]
    =O(\delta^2).  
    \]
Together with the previous two constructions, we have defined $\widehat P_\xi$ in all $D^{int}_\delta$, satisfying the desired boundary conditions, with
\[  \widetilde E_\xi(\widehat P_\xi; D^{int}_\delta) = \widetilde E_\xi(\widehat  P_\xi; X_\delta) + O(\delta^{1/2}) \le  \widetilde E_\xi(\widetilde R_\xi;D^{int}_\delta)
+ O(\delta^{\frac 16}\ln\frac 1\delta),  
\]
by \eqref{eq:boundPXdelta}. Using $\widehat P_\xi$ as a comparison map thus proves Lemma~\ref{l:Pxi}.
\end{proof}

To conclude, we extend $\widetilde P_\xi$ to $D^{ext}_\delta$ by setting 
\begin{equation*}
P_\xi =n_\delta\otimes n_\delta -\frac 13 I\qquad\text{in }D^{ext}_\delta.
\end{equation*}
From the estimates in Lemma~\ref{l:Rxi} and \ref{l:Pxi}  we have that
\begin{align*}
\widetilde E_\xi(\widetilde P_\xi)&\leq \widetilde E_\xi(\widetilde Q_\xi)+ \widetilde E_\xi(\widetilde P_\xi;D^{ext}_\delta)-E_\xi(\widetilde Q_\xi;D^{ext}_\delta) +\sigma_4(\delta,\xi)+ \zeta_4(\delta),
\end{align*}
where $\zeta_4(\delta)\to 0$ as $\delta\to 0$, and $\sigma_4(\delta,\xi)\to 0$ as $\xi\to 0$ for all fixed $\delta>0$. Recalling from the definition of $n_\delta$ that $\widetilde E_\xi(\widetilde P_\xi;D^{ext}_\delta)=2\mathbb E^+[\delta]$, we also have that
\begin{align*}
\frac 12\widetilde E_\xi(\widetilde P_\xi;D^{ext}_\delta)-\frac 12 \widetilde E_\xi(\widetilde Q_\xi;D^{ext}_\delta)&=\mathbb E_+[\delta]-\mathbb E_-[\delta] + \sigma_5(\delta,\xi)+\zeta_5(\delta),\\
\sigma_5(\delta,\xi)&=\frac 12\widetilde E_\star(\widetilde Q_\star;D^{ext}_\delta)-\frac 12 \widetilde E_\xi(\widetilde Q_\xi;D^{ext}_\delta),\\
\zeta_5(\delta)&=\mathbb E^-[\delta]-\frac 12\widetilde E_\star(\widetilde Q_\star;D^{ext}_\delta) =F(\varphi_\delta^-;D^{ext}_\delta)-F(\varphi;D^{ext}_\delta).
\end{align*}
Note that $\sigma_5(\delta,\xi)\to 0$ as $\xi\to 0$, thanks to Lemma~\ref{l:H1conv}. Since $\varphi$ minimizes $F(\cdot;D^{ext}_\eta)$ for every $\eta>0$ and satisfies the estimates of Lemma~\ref{l:estimphi0}, one can argue exactly as for \eqref{eq:limFeta} to prove that $\max(\zeta_5(\delta),0)\to 0$ as $\delta$ to 0. (In fact similar arguments will show that $\zeta_5(\delta)\to 0$, but here we only need the upper bound.)

Gathering the above estimates and recalling Lemma~\ref{l:compareEtau}, we deduce that
\begin{equation*}
\limsup_{\delta\to 0} \limsup_{\xi\to 0}\left[ \widetilde E_\xi(\widetilde P_\xi)- \widetilde E_\xi(\widetilde Q_\xi) \right] \leq 2\limsup_{\delta\to 0}\left(\mathbb E^+[\delta]-\mathbb E^-[\delta]\right)<0,
\end{equation*}
and so we can find $\delta,\xi>0$ such that the map $\widetilde P_\xi$ has strictly lower energy than $\widetilde Q_\xi$. This contradicts minimality of $\widetilde Q_\xi$ and concludes the proof of Theorem~\ref{thm:main}.

\bibliographystyle{acm}
\bibliography{colloid}

\end{document}